\documentclass[a4paper,10pt,leqno]{amsart}
\usepackage[all]{xy}
\usepackage{enumerate,amssymb,amscd,color}

\theoremstyle{plain}
\newtheorem{theorem}{Theorem}[section]

\newtheorem{lemma}[theorem]{Lemma}
\newtheorem{corollary}[theorem]{Corollary}

\newtheorem{conjecture}[theorem]{Conjecture}

\theoremstyle{definition}
\newtheorem{definition}[theorem]{Definition}

\newtheorem{example}[theorem]{Example}
\newtheorem{question}[theorem]{Question}

\newtheorem{remark}[theorem]{Remark}
\newtheorem{notation}[theorem]{Notation}
\newtheorem{problem}[theorem]{Problem}

\DeclareMathOperator{\CAT}{CAT}
\newcommand{\cd}{\operatorname{cd}}
\DeclareMathOperator{\ch}{ch}

\DeclareMathOperator{\coker}{coker}
\DeclareMathOperator{\con}{con}
\DeclareMathOperator{\cone}{cone}

\DeclareMathOperator{\End}{end}

\DeclareMathOperator{\hur}{hur}
\DeclareMathOperator{\id}{id}
\DeclareMathOperator{\ind}{ind}

\DeclareMathOperator{\map}{map}

\DeclareMathOperator{\NK}{NK}

\DeclareMathOperator{\Out}{Out}
\DeclareMathOperator{\pr}{pr}
\DeclareMathOperator{\res}{res}

\DeclareMathOperator{\tr}{tr}

\DeclareMathOperator{\sign}{sign}
\DeclareMathOperator{\Spin}{Spin}

\DeclareMathOperator{\topo}{top}

\DeclareMathOperator{\vcd}{vcd}

\DeclareMathOperator*{\colim}{colim}

\newcommand{\All}{\mathcal{ALL}}
\newcommand{\Com}{\mathcal{COM}}
\newcommand{\Comop}{\mathcal{COMOP}}
\newcommand{\Cyc}{\mathcal{CYC}}
\newcommand{\Fin}{{\mathcal{FIN}}}

\newcommand{\MFin}{{\mathcal{MFIN}}}
\newcommand{\TR}{\mathcal{TR}}
\newcommand{\VCyc}{\mathcal{VCYC}}

\newcommand{\IC}{\mathbb{C}} 
 \newcommand{\IF}{\mathbb{F}}
 \newcommand{\IH}{\mathbb{H}}
\newcommand{\II}{\mathbb{I}}

\newcommand{\IP}{\mathbb{P}} \newcommand{\IQ}{\mathbb{Q}}
\newcommand{\IR}{\mathbb{R}}

\newcommand{\IZ}{\mathbb{Z}}

 \newcommand{\calb}{\mathcal{B}}
\newcommand{\calbc}{\mathcal{BC}}
 
 \newcommand{\calf}{\mathcal{F}}
\newcommand{\FGI}{\mathcal{FGI}}
\newcommand{\calfj}{\mathcal{F\!J}} \newcommand{\calg}{\mathcal{G}}
\newcommand{\calh}{\mathcal{H}} 
 \newcommand{\calk}{\mathcal{K}}
\newcommand{\call}{\mathcal{L}} 
\newcommand{\caln}{\mathcal{N}} 
\newcommand{\calp}{\mathcal{P}} 
 \newcommand{\cals}{\mathcal{S}}
\newcommand{\calt}{\mathcal{T}}

\newcommand{\bfA}{{\mathbf A}} 
 
\newcommand{\bfE}{{\mathbf E}} \newcommand{\bff}{{\mathbf f}}
 
\newcommand{\bfH}{{\mathbf H}} 
 \newcommand{\bfK}{{\mathbf K}}
 \newcommand{\bfL}{{\mathbf L}}
\newcommand{\bfKO}{\mathbf{KO}}
 
 \newcommand{\bfP}{{\mathbf P}}
 
\newcommand{\bfS}{{\mathbf S}}

\newcommand{\bfTHH}{\mathbf{THH}}
\newcommand{\bfTC}{\mathbf{TC}}

\newcommand{\EGF}[2]{E_{#2}(#1)}               % klassifizierender
\newcommand{\JGF}[2]{J_{#2}(#1)}               % klassifizierender
\newcommand{\eub}[1]{\underline{E}#1}
\newcommand{\jub}[1]{\underline{J}#1}
\newcommand{\OrGF}[2]{\Or_{#2}(#1)}

\newcommand{\inj}{\operatorname{inj}}

\newcommand{\Wh}{\operatorname{Wh}}

\newcommand{\cok}{\operatorname{cok}}

\newcommand{\HS}{\operatorname{HS}}

\newcommand{\class}{\operatorname{class}}

\DeclareMathAlphabet{\matheurm}{U}{eur}{m}{n}
\newcommand{\pt}{\{\bullet\}}

\newcommand{\K}{\mathbf{K}} \newcommand{\THH}{\mathbf{THH}}
\newcommand{\TC}{\mathbf{TC}}

\newcommand{\Groupoids}{{\matheurm{Groupoids}}}
\newcommand{\Or}{\matheurm{Or}} 

\newcommand{\Spectra}{\matheurm{Spectra}}

\newcommand{\emphred}[1]{\emph{\color{red} #1}}

\newcommand{\blue}[1]{{\color{blue} #1}}
\newcommand{\green}[1]{{\color{green} #1}}
\newcommand{\red}[1]{{\color{red} #1}}

\newcommand{\xycomsquare}[8]                   % kommutatives Quadrat (xy-Version)
{\xymatrix{#1 \ar[r]^{#2} \ar[d]^{#4} &
    #3 \ar[d]^{#5}  \\
    #6\ar[r]^{#7} & #8 }}

\newcounter{commentcounter}
\newcommand{\comment}[1]                      %comment of the author
{\stepcounter{commentcounter}
{\bf Comment \arabic{commentcounter} (by W.)}: {\ttfamily #1} }

\makeatletter\let\c@equation=\c@theorem\makeatother

\begin{document}
\title[On the Farrell-Jones and related Conjectures]%nil groups
{On the Farrell-Jones and related Conjectures}
\author{Wolfgang L\"uck}
\address{Mathematisches Institut
\newline\indent
Universit\"at M\"unster
\newline\indent
D-48149 M\"unster, Germany
\newline \indent
http://www.math.uni-muenster.de/u/lueck}
\email{lueck@math.uni-muenster.de}

\date{\today}
\begin{abstract} These extended notes are based on a series of six 
lectures presented  at the summer school 
\green{``Cohomology of groups and algebraic $K$-theory''} which took place in 
\green{Hangzhou, China} from July 1 until July 12 in 2007.
They give an introduction to the \emphred{Farrell-Jones} and the \emphred{Baum-Connes
Conjecture}.
\smallskip

\smallskip

\end{abstract}

\maketitle
\setcounter{section}{-1}
\section{Introduction}\label{one}

These extended notes are based on a series of six lectures presented
at the summer school \green{``Cohomology of groups and algebraic
  $K$-theory''} which took place in \green{Hangzhou, China} from July
1 until July 12 in 2007.  They contain an introduction to the
\emphred{Farrell-Jones} and the \emphred{Baum-Connes Conjecture}.

Given a group $G$, the
Farrell-Jones Conjecture and the Baum-Connes Conjecture respectively 
predict the values of the \red{algebraic $K$-} and
\red{$L$-theory} of the \red{group ring $RG$} and of the
\red{topological K-theory} of the \red{reduced group $C^*$-algebra} respectively.  These
are very hard to compute directly. These conjectures identify them via
\red{assembly maps} to much easier to handle equivariant homology
groups of certain classifying spaces. This is the \red{computational
  aspect} of these conjectures.

But also the \red{structural aspect} is very important. The
assembly maps have geometric or analytic interpretations.  Hence the
Farrell-Jones Conjecture and the Baum-Connes Conjecture imply many
very well-known conjectures such as the ones due to \green{Bass},
\green{Borel}, \green{Kaplansky}, and \green{Novikov}. The point is
that the Farrell-Jones Conjecture and the Baum-Connes Conjecture have
been proven for many groups for which the other conjectures were a
priori not known.

The prerequisites consist of a solid knowledge of homology theory
and $CW$-complexes and of a basic knowledge of rings, modules,
homological algebra, groups, group homology, finite dimensional representation theory of
finite groups, group actions, categories, homotopy
groups, and manifolds. The challenge for the reader but also the
beauty, impact and fascination of these conjectures come from the broad scope of
mathematics which they address and which is needed for proofs and applications.

For a more advanced survey on the Farrell-Jones and the Baum-Cones
Conjecture we refer to \green{L\"uck-Reich}~\cite{Lueck-Reich(2005)}.
There more details are given and more aspects are discussed but it is
addressed to a more advanced reader and requires much more previous knowledge.

We fix some notation. Ring will always mean associative ring with unit
(which is not necessarily commutative).  Examples are the ring of
integers $\IZ$, the fields of rational numbers $\IQ$, of real numbers
$\IR$ and of complex numbers $\IC$, the finite field $\IF_p$ of $p$
elements, and the group ring $RG$ for a ring $R$ and a group $G$.
Ring homomorphisms are unital. Modules are understood to be left
modules unless explicitly stated differently.  Groups are understood
to be discrete unless explicitly stated differently.

The notes are organized as the six lectures in Hangzhou have been:

\tableofcontents

The author wants to express his deep gratitude to all the organizers of the summerschool for
their excellent work, support and hospitality.

%%%%%%%%%%%%%%%%%%%%%%%%%%%%%%%%%%%%%%%%%%%%%%%%%%%%%%%%%%%%%%%%%%%%%%%%%%%%%%%%
%%%%%%%%%%%%%%%%%%%%%%%%%%%%%%%%%%%%%%%%%%%%%%%%%%%%%%%%%%%%%%%%%%%%%%%%%%%%%%%%
%%%%%%%%%%%%%%%%%%%%%%%%%%%%%%%%%%%%%%%%%%%%%%%%%%%%%%%%%%%%%%%%%%%%%%%%%%%%%%%%

\section{The role of  lower and middle K-theory in topology}

The outline of this section is:
\begin{itemize}

  \item  Introduce the \red{projective class group $K_0(R)$}.

  \item Discuss its algebraic and topological significance 
  (e.g., \red{finiteness obstruction}).

  \item Introduce \red{$K_1(R)$} and the \red{Whitehead group $\Wh(G)$}.

  \item Discuss its algebraic and topological significance 
  (e.g., \red{$s$-cobordism theorem}).

  \item Introduce \red{negative $K$-theory} and the 
  \red{Bass-Heller-Swan decomposition}.

  \end{itemize}

\begin{definition}[\blue{Projective $R$-module}]
  An $R$-module $P$ is called \emphred{projective}
  if it satisfies one  of the following equivalent conditions:
  \begin{enumerate}

  \item $P$ is a direct summand in a free $R$-module;

  \item The following lifting problem has always a solution
$$\xymatrix{
  M \ar[r]^-{p} & N \ar[r] & 0
  \\
  & P \ar@{-->}[lu]^-{\overline{f}} \ar[u]_-{f} & }
$$

\item If $0 \to M_0 \to M_1 \to M_2 \to 0$ is an exact sequence of
  $R$-modules, then $0 \to \hom_R(P,M_0) \to \hom_R(P,M_1) \to
  \hom_R(P,M_2) \to 0$ is exact.

\end{enumerate}

\end{definition}

\begin{example}[\blue{Principal ideal domains}]
  Over a field or, more generally, over a principal ideal domain every
  projective module is free. If $R$ is a principal ideal domain, then
  a finitely generated $R$-module is projective (and hence free) if
  and only if it is torsionfree. For instance $\IZ/n$ is for $n \ge 2$
  never projective as $\IZ$-module.
\end{example}

\begin{example}[\blue{Product of rings}]
  Let $R$ and $S$ be rings and $R \times S$ be their product. Then $R
  \times \{0\}$ is a finitely generated projective $R \times S$-module
  which is not free.
\end{example}

\begin{example}[\blue{Trivial representation of a finite group}]
  Let $F$ be a field of characteristic $p$ for $p$ a
  prime number or $p = 0$.  Let $G$ be a finite group.  Then $F$ with the
  trivial $G$-action is a projective $FG$-module if and only if $p =
  0$ or $p$ does not divide the order of $G$.  It is a free
  $FG$-module only if $G$ is trivial.
\end{example}

\begin{definition}[\blue{Projective class group}]
  Let $R$ be a ring.  Define its
  \emphred{projective class group} $K_0(R)$ to be the abelian group
  whose generators are isomorphism classes $[P]$ of finitely generated
  projective $R$-modules $P$ and whose relations are $[P_0] + [P_2] =
  [P_1]$ for every exact sequence $0 \to P_0 \to P_1 \to P_2 \to 0$ of
  finitely generated projective $R$-modules.
\end{definition}

The projective class group $K_0(R)$ is the same as the
\red{Grothendieck construction} applied to the abelian monoid of
isomorphism classes of finitely generated projective $R$-modules under
direct sum. There do exists rings $R$ with $K_0(R) = 0$, e.g., 
$R = \End(F)$ for a field $F$.

\begin{definition}[\blue{Reduced projective class group}]
  The \emphred{reduced projective class group $\widetilde{K}_0(R)$} is
  the quotient of $K_0(R)$ by the subgroup generated by the classes of
  finitely generated free $R$-modules, or, equivalently, the cokernel
  of $K_0(\IZ) \to K_0(R)$.
\end{definition}

\begin{remark}[\blue{Stably finitely generated free modules}]
  Let $P$ be a finitely generated projective $R$-module.  It is
  \red{stably finitely generated free}, i.e., $P \oplus R^m \cong R^n$
  for appropriate $m,n \in \IZ$, if and only if ${[P]} = 0$ in
  $\widetilde{K}_0(R)$.  Hence $\widetilde{K}_0(R)$ measures the
  deviation of finitely generated projective $R$-modules from being
  stably finitely generated free.

  There exists finitely generated projective $R$-modules which are 
  stably finitely generated free but not finitely generated free. 
  An example is $R = C(S^2)$ and $P = C(TS^2)$, where $C(S^2)$ is the
  ring of continuous functions $S^2 \to \IR$ and $C(TS^2)$ is 
  the $C(S^2)$-module of sections of the tangent bundle of $S^2$. 
  However, in most of the  applications the relevant question is
  whether a finitely generated projective $R$-module is stable 
  finitely generated free and not whether  it is finitely generated free.
\end{remark}

\begin{remark}[\blue{Universal dimension function}]
  The assignment $P \mapsto [P] \in K_0(R)$ is the \emphred{universal
    additive invariant} or \emphred{dimension function} for finitely
  generated projective $R$-modules in the following sense. Given an
  abelian group $A$ and an assignment associating to a finitely
  generated projective $R$-module $P$ an element $a(P) \in A$ such
  that $a(P_0) -a(P_1) + a(P_2) = 0$ holds for any exact sequence of
  finitely generated projective $R$-modules $0 \to P_0 \to P_1 \to P_2
  \to 0$, there exists precisely one homomorphism of abelian
  groups $\phi \colon K_0(R) \to A$ satisfying $\phi([P]) = a(P)$ for
  every finitely generated projective $R$-module $P$.
\end{remark}

\begin{remark}[\blue{Induction}] Let $f \colon R \to S$ be a ring
  homomorphism.  Consider $S$ as a $S$-$R$-bimodule via $f$.
   Given an $R$-module $M$, let $f_* M$ be the
  $S$-module $S \otimes_R M$.  We obtain a homomorphism of abelian
  groups
$$f_* \colon K_0(R) \to K_0(S), \quad [P] \mapsto [f_*P]$$
called \emphred{induction} or \emphred{change of rings homomorphism}.
Thus $K_0$ becomes a covariant functor from the category of rings to
the category of abelian algebras.
\end{remark}

\begin{lemma}[\blue{$K_0$ and products}]
  Let $R$ and $S$ be rings. Then the two projections from $R \times S$
  to $R$ and $S$ induce isomorphisms
$$K_0(R \times S) \xrightarrow{\cong} K_0(R) \times K_0(S).$$
\end{lemma}

\begin{theorem}[\blue{Morita equivalence}]
  Let $R$ be a ring and $M_n(R)$ be the ring of $(n,n)$-matrices over
  $R$.  We can consider $R^n$ as a $M_n(R)$-$R$-bimodule and as a
  $R$-$M_n(R)$-bimodule by scalar and matrix multiplication.  
  Tensoring with these yields mutually inverse isomorphisms
$$\begin{array}{rcllll}
K_0(R) & \xrightarrow{\cong} & K_0(M_n(R)), 
& \quad [P] &\mapsto & [{_{M_n(R)}{R^n}_R} \otimes_R P];
\\
K_0(M_n(R)) & \xrightarrow{\cong} & K_0(R), 
& \quad [Q] &\mapsto & [{_R{R^n}_{M_n(R)}} \otimes_{M_n(R)} Q].
\end{array}
$$
\end{theorem}

\begin{example}[\blue{Principal ideal domains}]
  Let $R$ be a principal ideal domain.  Let $F$ be its quotient field.
  Then we obtain mutually inverse isomorphisms
$$\begin{array}{rclrcl}
  \IZ &\xrightarrow{\cong} & K_0(R), & \quad n &\mapsto & n \cdot [R];
  \\
  K_0(R) &\xrightarrow{\cong} & \IZ , & \quad [P] &\mapsto & \dim_F(F \otimes_R P).
\end{array}
$$
\end{example}

\begin{example}[\blue{Representation ring}]
  Let $G$ be a finite group and let $F$ be a field of characteristic
  zero.  Then the \red{representation ring $R_F(G)$} is the same as
  $K_0(FG)$.  Taking the character of a representation yields an
  isomorphism
$$R_{\IC}(G) \otimes_{\IZ} \IC = K_0(\IC G) \otimes_{\IZ} \IC
\xrightarrow{\cong} \class(G,\IC),$$ where $\red{\class(G;\IC)}$ is
the complex vector space of \red{class functions} $G \to \IC$, i.e.,
functions, which are constant on conjugacy classes. We refer for instance
to the book of \green{Serre}~\cite{Serre(1977)} for more information about
the representation theory of finite groups.
\end{example}

\begin{definition}[\blue{Dedekind domain}] \label{def:Dedekind_domain}
A commutative ring $R$ is called \emphred{Dedekind domain}
if it is an integral domain, i.e., contains no non-trivial zero-divisors,
and for every pair of ideals $I \subseteq J$ of $R$ 
there exists an ideal $K \subseteq R$ with $I = JK$.
\end{definition}

A ring is called \emphred{hereditary}
if every ideal is projective, or, equivalently,
if every submodule of a projective $R$-module is projective.

\begin{theorem}[\blue{Characterization of Dedekind domains}]
\label{the:Characterization_of_Dedekind_domains}
The following assertions are equivalent for a commutative integral domain with
quotient field $F$:

\begin{enumerate}

\item $R$ is a Dedekind domain;

\item $R$ is hereditary;

\item Every finitely generated torsionfree $R$-module is projective;

\item $R$ is Noetherian and integrally closed in its quotient field
  $F$ and every non-zero prime ideal is maximal.

\end{enumerate}
\end{theorem}
\begin{proof}
  This follows from~\cite[Proposition~4.3 on page 76 and
  Proposition~4.6 on page~77]{Curtis-Reiner(1981)} and the fact that a
  finitely generated torsionfree module over an integral domain $R$
  can be embedded into $R^n$ for some integer $n \ge 0$
  (see \green{Auslander-Buchsbaum}~\cite[Proposition~3.3 in Chapter~9 
  on page~321]{Auslander-Buchsbaum(1974)}).
   \end{proof}

\begin{example}[\blue{Ring of integers}]
Recall that an \emphred{algebraic number field}
is a finite algebraic extension of $\IQ$ and the
\emphred{ring of integers} in $F$ is the integral closure of $\IZ$ in
$F$. The ring of integers in an algebraic number field is a Dedekind domain.
(see~\cite[Theorem~1.4.18 on page~22]{Rosenberg(1994)}). 
\end{example}

\begin{example}[\blue{Dedekind domains}]
  Let $R$ be a Dedekind domain.  We call two ideals $I$ and $J$ in $R$
  equivalent if there exists non-zero elements $r$ and $s$ in $R$ with
  $rI = sJ$.  The \emphred{ideal class group C(R)} is the abelian group of
  equivalence classes of ideals under multiplication of ideals.  Then
  $C(R)$ is finite and we obtain an isomorphism
$$C(R) \xrightarrow{\cong} \widetilde{K}_0(R), \quad [I] \mapsto [I].$$
A proof of the claim above 
can be found for instance in~\cite[Corollary~11 on page 14]{Milnor(1971)}
and~\cite[Theorem~1.4.12 on page~20 and and Theorem~1.4.19 on page~23]{Rosenberg(1994)}.

The structure of the finite abelian group
$$C(\IZ[\exp(2\pi i /p)]) \cong \widetilde{K}_0(\IZ[\exp(2\pi i /p)]) 
\cong \widetilde{K}_0(\IZ[\IZ/p])$$ is only known for small prime
numbers $p$ (see~\cite[Remark 3.4 on page~30]{Milnor(1971)}).
\end{example}

\begin{theorem}[\green{Swan (1960)}]
  If $G$ is finite, then $\widetilde{K}_0(\IZ G)$ is finite.
\end{theorem}
\begin{proof}
See~\cite[Theorem~8.1 and Proposition~9.1]{Swan(1960a)}.
\end{proof}

Let $X$ be a compact space.  Let $K^0(X)$ be the Grothendieck group of
isomorphism classes of finite-dimensional complex vector bundles over
$X$.  This is the zero-th term of a generalized cohomology theory
$K^*(X)$ called \red{topological $K$-theory}.  It is $2$-periodic,
i.e., $K^n(X) = K^{n+2}(X)$, and satisfies $K^0(\pt) = \IZ$ and
$K^1(\pt) = \{0\}$, where $\pt$ is the space consisting of one point.

Let $C(X)$ be the ring of continuous functions from $X$ to $\IC$.

\begin{theorem}[\green{Swan (1962)}] 
If $X$ is a compact space, then there is an isomorphism
$$K^0(X) \xrightarrow{\cong} K_0(C(X)).$$
\end{theorem}
\begin{proof} See~\cite{Swan(1962)}.
\end{proof}

\begin{definition}[\blue{Finitely dominated}]
  A $CW$-complex $X$ is called \emphred{finitely dominated} if there
  exists a finite  $CW$-complex $Y$ together with maps $i
  \colon X \to Y$ and $r \colon Y \to X$ satisfying $r\circ i \simeq
  \id_X$.
\end{definition}

Obviously a finite $CW$-complex is finitely dominated.

\begin{problem}
  Is a given finitely dominated $CW$-complex homotopy equivalent to a
  finite $CW$-complex?
\end{problem}

A finitely dominated $CW$-complex $X$ defines an element
$$o(X) \in K_0(\IZ[\pi_1(X)])$$
called its \emphred{finiteness obstruction} as follows.  Let
$\widetilde{X}$ be the universal covering.  The fundamental group $\pi
= \pi_1(X)$ acts freely on $\widetilde{X}$.  Let $C_*(\widetilde{X})$
be the cellular chain complex.  It is a free $\IZ\pi$-chain complex.
Since $X$ is finitely dominated, there exists a finite projective
$\IZ\pi$-chain complex $P_*$ with $P_* \simeq_{\IZ \pi}
C_*(\widetilde{X})$. Finite projective means 
that every $P_i$ is finitely generated projective and $P_i \not= 0$
holds only for finitely many element $i \in \IZ$.

\begin{definition}[\green{Wall's} \blue{finiteness obstruction}]
  Define
$$o(X) := \sum_{n} (-1)^n \cdot [P_n] \in K_0(\IZ \pi).$$
\end{definition}

This definition is indeed independent of the choice of $P_*$.

\begin{theorem}[\green{Wall (1965)}] 
  A finitely dominated $CW$-complex $X$ is homotopy equivalent to a
  finite $CW$-complex if and only if its reduced finiteness
  obstruction $\widetilde{o}(X) \in \widetilde{K}_0(\IZ[\pi_1(X)])$
  vanishes.

  Given a finitely presented group $G$ and $\xi \in K_0(\IZ G)$, there
  exists a finitely dominated $CW$-complex $X$ with $\pi_1(X) \cong G$
  and $o(X) = \xi$.
\end{theorem}
\begin{proof}
See~\cite{Wall(1965a)} and~\cite{Wall(1966)}.
\end{proof}

A finitely dominated simply connected $CW$-complex is always homotopy
equivalent to a finite $CW$-complex since $\widetilde{K}_0(\IZ) =
\{0\}$.

\begin{corollary}[\blue{Geometric characterization of $\widetilde{K}_0(\IZ G) = \{0\}$}]
  The following statements are equivalent for a finitely presented
  group $G$:
  \begin{enumerate}
  \item Every finite dominated $CW$-complex with $G \cong \pi_1(X)$ is
    homotopy equivalent to a finite $CW$-complex;

  \item $\widetilde{K}_0(\IZ G) = \{0\}$.

  \end{enumerate}
\end{corollary}

\begin{conjecture}[\blue{Vanishing of $\widetilde{K}_0(\IZ G)$  for torsionfree $G$}]
  If $G$ is torsionfree, then
$$\widetilde{K}_0(\IZ G) = \{0\}.$$
\end{conjecture}

For more information about the finiteness obstruction we refer for
instance to~\cite{Ferry(1981a)}, \cite{Ferry-Ranicki(2001)},
\cite{Lueck(1987b)}, \cite{Mislin(1995)}, \cite{Ranicki(1985)},
\cite{Rosenberg(2005)}, \cite{Varadarajan(1989a)}, \cite{Wall(1965a)}
and \cite{Wall(1966)}.

\begin{definition}[\blue{$K_1$-group}]
  \label{def:K_1(R)}
  Define the $\red{K_1(R)}$ to be the abelian group whose generators
  are conjugacy classes $[f]$ of automorphisms $f\colon P \to P$ of
  finitely generated projective $R$-modules with the following
  relations:
  \begin{enumerate}
  \item Given an exact sequence $0 \to (P_0,f_0) \to (P_1,f_1) \to
    (P_2,f_2) \to 0$ of automorphisms of finitely generated projective
    $R$-modules, we get $[f_0] + [f_2] = [f_1]$;

  \item $[g \circ f] = [f] + [g]$.
  \end{enumerate}
\end{definition}

\begin{theorem}[\blue{$K_1(R)$ and matrices}] There is a natural isomorphism
$$K_1(R) \cong GL(R)/[GL(R),GL(R)],$$
where the target is the abelianization of the general linear group
$GL(R) = \bigcup_{n \ge 1} GL_n(R)$. 
\end{theorem}
\begin{proof} See~\cite[Theorem~3.1.7 on page~113]{Rosenberg(1994)}.
\end{proof}

\begin{remark}[\blue{$K_1(R)$ and row and column operations}]
\label{rem:K_1(R)_and_operations}
  An invertible matrix $A \in GL(R)$ can be reduced by \red{elementary
  row and column operations} and \red{(de-)stabilization} to the
  empty matrix if and only if $[A] = 0$ holds in the
  \emphred{reduced $K_1$-group}
$$\red{\widetilde{K}_1(R)} := K_1(R)/\{\pm 1\} = 
\cok\left(K_1(\IZ) \to K_1(R)\right).$$
\end{remark}

\begin{remark}[\blue{$K_1(R)$ and determinants}]
If $R$ is commutative, the determinant induces an epimorphism
$$\det \colon K_1(R) \to R^{\times},$$
which in general is not bijective.

The assignment $A \mapsto [A] \in K_1(R)$ can be thought of as the
\red{universal determinant for $R$}, where $R$ is not necessarily
commutative. Namely, given any abelian group $A$ together with an
assignment which associates to an $R$-automorphism $f\colon P \to P$ of a
finitely generated projective $R$-module an element $[f]$ such that the obvious
analogues of the relations appearing in Definition~\ref{def:K_1(R)}
hold, there exists precisely one homomorphism of abelian groups $\phi
\colon K_1(R) \to A$ sending $[f]$ to $a(f)$ for every
$R$-automorphism $f$ of a finitely generated projective $R$-module.
\end{remark}

There do exists rings $R$ with $K_1(R) = 0$, e.g. $R = \End(F)$ for a field $F$.

\begin{remark}[$K_1(R)$ of principal ideal domains]
\label{rem:K_1(R)_of_principal_ideal_domains}
There exists principal ideal domains $R$ such that $\det \colon K_1(R)
\to R^{\times}$ is not bijective. For instance
\green{Grayson}~\cite{Grayson(1981)} gives such an example, namely, take
$\IZ[x]$ and invert $x$ and all polynomials of the shape $x^m -1$ for
$m \ge 1$.  Other examples can be found in
\green{Ischebeck}~\cite{Ischebeck(1980)}.
\end{remark}

\begin{theorem}[\blue{$K_1$ of ring of integers},
\green{Bass-Milnor-Serre (1967)}]
Let $R$ be the ring of integers in an algebraic number field. 
Then the determinant induces an isomorphism
$$\det \colon K_1(R) \xrightarrow{\cong} R^{\times}.$$
\end{theorem}
\begin{proof}
See~\cite[4.3]{Bass-Milnor-Serre(1967)}.
\end{proof}

\begin{definition}[\blue{Whitehead group}]
 The \emphred{Whitehead group} of a group $G$ is defined to be
$$\red{\Wh(G)} = K_1(\IZ G)/\{\pm g \mid g \in G\}.$$
\end{definition}

\begin{lemma}
  We have $\Wh(\{1\}) = \{0\}$.
\end{lemma}
\begin{proof}
  The ring $\IZ$ possesses an Euclidean algorithm.  Hence every
  invertible matrix over $\IZ$ can be reduced via elementary row and
  column operations and destabilization to a $(1,1)$-matrix $(\pm 1)$.
  For every ring such operations do not change the class of a matrix in $K_1(R)$.
  \end{proof}

Let $G$ be a finite group.  Let $F$ be $\IQ$, $\IR$ or $\IC$.  Define
\red{$r_{F}(G)$} to be the number of irreducible $F$-representations
{of $G$}.  This is the same as the number of $F$-conjugacy classes of
elements of $G$.  Here \red{$g_1 \sim_{\IC} g_2$} if and only if
\red{$g_1 \sim g_2$}, i.e., $gg_1g^{-1} = g_2$ for some $g \in G$.  We
have \red{$g_1 \sim_{\IR} g_2$} if and only if $g_1 \sim g_2$ or $g_1
\sim g_2^{-1}$ holds.  We have \red{$g_1 \sim_{\IQ} g_2$} if and only
if $\langle g_1 \rangle$ and $\langle g_1 \rangle$ are conjugated as
subgroups of $G$.

\begin{theorem}[\blue{$\Wh(G)$ for finite groups $G$}]\ \newline
\vspace*{-4mm}
  \begin{enumerate}

  \item The Whitehead group $\Wh(G)$ is a finitely generated abelian
    group;

  \item Its rank is $r_{\IR}(G) - r_{\IQ}(G)$.

  \item The torsion subgroup of $\Wh(G)$ is the kernel of the map
    $K_1(\IZ G) \to K_1(\IQ G)$. 
  \end{enumerate}
\end{theorem}

In contrast to $\widetilde{K}_0(\IZ G)$ the Whitehead group $\Wh(G)$
is computable (see \green{Oliver}~\cite{Oliver(1989)}).

\begin{definition}[\blue{$h$-cobordism}]
  An \emphred{$h$-cobordism} over a closed manifold $M_0$ is a compact
  manifold $W$ whose boundary is the disjoint union $M_0 \amalg M_1$
  such that both inclusions $M_0 \to W$ and $M_1 \to W$ are homotopy
  equivalences.
\end{definition}

\begin{theorem}[\blue{$s$-Cobordism Theorem}, \green{Barden, Mazur,
    Stallings, Kirby-Siebenmann}]\ \newline
\vspace*{-8mm}
 \begin{enumerate}
  \item Let $M_0$ be a closed (smooth) manifold of dimension $\ge 5$.
    Let $(W;M_0,M_1)$ be an $h$-cobordism over $M_0$.

    Then $W$ is homeomorphic (diffeomorpic) to $M_0 \times [0,1]$
    relative $M_0$ if and only if its \red{Whitehead torsion}
$$\red{\tau(W,M_0)} \in \Wh(\pi_1(M_0)).$$
vanishes;

\item Let $G$ be a finitely presented group $G$, $n$ an integer $n \ge
  5$ and $x$ an element in $\Wh(G)$. Then there exists an
  $n$-dimensional $h$-cobordism $(W;M_0,M_1)$ over $M_0$ with
  $\tau(W,M_0) = x$.
\end{enumerate}
\end{theorem}

\begin{corollary}[\blue{Geometric characterization of $\Wh(G) = \{0\}$}]
The following statements are equivalent for a finitely presented group $G$
and a fixed integer $n \ge 6$

\begin{enumerate}

\item Every compact $n$-dimensional $h$-cobordism $W$ with $G \cong \pi_1(W)$ is trivial;

\item $\Wh(G) = \{0\}$.

\end{enumerate}
\end{corollary}

\begin{conjecture}[\blue{Vanishing of $\Wh(G)$  for torsionfree $G$}] 
If $G$ is torsionfree, then
$$\Wh(G) = \{0\}.$$
\end{conjecture}

\begin{conjecture}[\blue{Poincar\'e Conjecture}]
  Let $M$ be an $n$-dimensional topological manifold which is a
  \red{homotopy sphere}, i.e., homotopy equivalent to $S^n$.

  Then $M$ is homeomorphic to $S^n$.
\end{conjecture}

\begin{theorem} For $n \ge 5$ the Poincar\'e Conjecture is true.
\end{theorem}
\begin{proof} We sketch the proof for $n \ge 6$.  Let $M$ be an
  $n$-dimensional homotopy sphere.  Let $W$ be obtained from $M$ by
  deleting the interior of two disjoint embedded disks $D^n_1$ and
  $D^n_2$. Then $W$ is a simply connected $h$-cobordism.  Since
  $\Wh(\{1\})$ is trivial, we can find a homeomorphism $f \colon W
  \xrightarrow{\cong} \partial D^n_1 \times [0,1]$ which is the
  identity on $\partial D^n_1 = \partial D^n_1 \times \{0\}$.  By the
  \red{Alexander trick}, i.e., by coning the 
  homeomorphism of $\partial D^n$ to the cone of 
  $\partial D^n$ which is $D^n$,  we can extend the homeomorphism
  $f|_{\partial D^n_2}  \colon  \partial D^n_2 \xrightarrow{\cong} 
  \partial D^n_1 = \partial  D^n_1 \times \{1\}$ 
  to a homeomorphism $g \colon D^n_2 \to D^n_1$.
  The three homeomorphisms $id_{D^n_1}$, $f$ and $g$ fit together to a
  homeomorphism $h \colon M \to D^n_1 \cup_{\partial D^n_1 \times
    \{0\}} \partial D^n_1 \times [0,1] \cup_{\partial D^n_1 \times
    \{1\}} D^n_1$. The target is obviously homeomorphic to $S^n$.
\end{proof}

\begin{remark}[\blue{Exotic spheres}]
  The argument above does not imply that for a smooth manifold $M$ we
  obtain a diffeomorphism $g \colon M \to S^n$.  The problem is that
  the Alexander trick does not work smoothly.  Indeed, there exists so
  called \red{exotic spheres}, i.e., closed smooth manifolds which are
  homeomorphic but not diffeomorphic to $S^n$. For
  more information about exotic spheres we refer for instance
  to~\cite{Kervaire-Milnor(1963)},~\cite{Lance(2000)},~\cite{Levine(1983)}
  and~\cite[Chapter~6]{Lueck(2002c)}.
\end{remark}

\begin{remark}[\blue{The Poincar\'e Conjecture and the $s$-cobordism theorem
in low dimensions}]
The Poincar\'e Conjecture has been proved in dimension $4$ by 
\green{Freedman}~\cite{Freedman(1982)} 
and in dimension $3$ by \green{Perelman}
(see~\cite{Perelman(2002entropy)}, \cite{Perelman(2003extinction)} 
and for more details for instance~\cite{Kleiner-Lott(2006)}, 
\cite{Morgan-Tian(2006)}).
It is true in dimensions $1$ and $2$ for elementary reasons.

The $s$-cobordism theorem is known to be false (smoothly) for $n =
\dim(M_0) = 4$ in general, by the work of 
\green{Donaldson}~\cite{Donaldson(1987a)}, but it is true for $n = \dim(M_0) = 4$ for so
called ``good'' fundamental groups in the topological category by
results of \green{Freedman}~\cite{Freedman(1982)}, \cite{Freedman(1983)}. The
trivial group is an example of a ``good'' fundamental group.
Counterexamples in the case $n = \dim(M_0) = 3$ are constructed by
\green{Cappell-Shaneson}~\cite{Cappell-Shaneson(1985)}.
\end{remark}

\begin{remark}[\blue{Surgery program}]
The $s$-cobordism theorem is a key ingredient in the \red{surgery
program} for the classification of closed manifolds due to
\green{Browder, Novikov, Sullivan} and \green{Wall}. 
For more information about surgery theory we refer for instance to
\cite{Browder(1972)}, 
\cite{Cappell-Ranicki-Rosenberg(2000)}, 
\cite{Cappell-Ranicki-Rosenberg(2001)}, 
\cite{Farrell-Goettsche-Lueck(2002a)}, 
\cite{Farrell-Goettsche-Lueck(2002b)},
\cite{Karoubi(2004)},
\cite{Kreck(1999)}, 
\cite{Ranicki(2002a)},
\cite{Stark(2000)},
\cite{Stark(2002)},
and 
\cite{Wall(1999)}.

\end{remark}

More information about Whitehead torsion and the $s$-cobordism theorem 
can be found for instance in~\cite{Cohen(1973)},
\cite{Kervaire(1965)}, \cite[Chapter 1]{Lueck(2002c)},
\cite{Milnor(1965b)}, \cite{Milnor(1966)}, 
\cite[page 87-90]{Rourke-Sanderson(1982)}.  

\begin{definition}[\blue{Bass-Nil-groups}]
Define for $n = 0,1$
$$\red{\NK_n(R)} := \coker\left(K_n(R) \to K_n(R[t])\right).$$
\end{definition}

\begin{theorem}[\blue{Bass-Heller-Swan decomposition for $K_1$}, 
  \green{Bass-Heller-Swan(1964)}]
  There is an isomorphism, natural in $R$,
  \begin{eqnarray*}
    & K_0(R) \oplus K_1(R) \oplus \NK_1(R) \oplus \NK_1(R)
    \xrightarrow{\cong}  K_1(R[t,t^{-1}]) = K_1(R[\IZ]).&
  \end{eqnarray*}
\end{theorem}
\begin{proof}
See for instance~\cite{Bass-Heller-Swan(1964)} (for regular rings),
\cite[Chapter XII]{Bass(1968)}, 
\cite[Theorem~3.2.22 on page~149]{Rosenberg(1994)}.
\end{proof}

Notice that the Bass-Heller-Swan decomposition for $K_1$ gives the possibility
to define $K_0(R)$ in terms of $K_1$. This motivates the following definition.

\begin{definition}[\blue{Negative $K$-theory}]
Define inductively for $n = -1, -2, \ldots$
$$\red{K_n(R)} :=
\coker\left(K_{n+1}(R[t]) \oplus K_{n+1}(R[t^{-1}])\to K_{n+1}(R[t,t^{-1}])\right).$$

Define for $n = -1, -2, \ldots$
$$\red{\NK_n(R)} := \coker\left(K_n(R) \to K_n(R[t])\right).$$
\end{definition}

\begin{theorem}[\blue{Bass-Heller-Swan decomposition for negative $K$-theory}]
\label{the_Bass-Heller_Swan_negative}
For $n \le 1$ there is an isomorphism, natural in $R$,
\begin{eqnarray*}
& K_{n-1}(R) \oplus K_n(R) \oplus \NK_n(R) \oplus \NK_n(R)
 \xrightarrow{\cong}  K_n(R[t,t^{-1}]) = K_n(R[\IZ]).&
\end{eqnarray*}
\end{theorem}

\begin{definition}[\blue{Regular ring}]
A ring $R$ is called \emphred{regular} if it is Noetherian and every finitely
generated $R$-module possesses a finite projective resolution.
\end{definition}

Principal ideal domains are regular. In particular $\IZ$ and any field are regular.
If $R$ is regular, then $R[t]$ and $R[t,t^{-1}] = R[\IZ]$ are regular.
If $R$ is Noetherian, then $RG$ is not in general Noetherian.
Theorem~\ref{the_Bass-Heller_Swan_negative} implies

\begin{theorem}[\blue{Bass-Heller-Swan decomposition for regular rings}]
Suppose that $R$ is regular.  Then
\begin{eqnarray*}
K_n(R) & = & 0 \quad  \text{for } n \le -1;
\\
\NK_n(R) & = & 0 \quad \text{for } n \le 1,
\end{eqnarray*}
and the Bass-Heller-Swan decomposition
reduces for $n \le 1$ to the natural isomorphism
$$K_{n-1}(R) \oplus K_n(R) \xrightarrow{\cong} K_n(R[t,t^{-1}]) = K_n(R[\IZ]).$$
\end{theorem}

There are also \red{higher algebraic $K$-groups $K_n(R)$} for $n \ge 2$
due to \green{Quillen (1973)}.
They are defined as homotopy groups of certain spaces or spectra.
We refer to the lectures of \green{Grayson}.
Most of the well known  features of $K_0(R)$ and $K_1(R)$ extend
to both negative and higher algebraic $K$-theory. 
For instance the Bass-Heller-Swan decomposition holds also for higher algebraic
$K$-theory.

\begin{remark}[\blue{Similarity between $K$-theory and group homology}]
  Notice the following formulas for a regular ring $R$ and a
  generalized homology theory $\calh_*$, which look similar:
\begin{eqnarray*}
K_n(R[\IZ]) & \cong & K_n(R) \oplus K_{n-1}(R);
\\
\calh_n(B\IZ) & \cong & \calh_n(\pt) \oplus \calh_{n-1}(\pt).
\end{eqnarray*}
If $G$ and $K$ are groups, then we have the following formulas,
which look similar:
\begin{eqnarray*}
\widetilde{K}_n(\IZ[G \ast K]) & \cong & 
\widetilde{K}_n(\IZ G) \oplus \widetilde{K}_n(\IZ K);
\\
\widetilde{\calh}_n(B(G \ast K)) & \cong & 
\widetilde{\calh}_n(BG) \oplus \widetilde{\calh}_n(BK).
\end{eqnarray*}
\end{remark}

\begin{question}[\blue{$K$-theory of group rings and group homology}]
Is there a relation between $K_n(RG)$ and group homology of $G$?
\end{question}

%%%%%%%%%%%%%%%%%%%%%%%%%%%%%%%%%%%%%%%%%%%%%%%%%%%%%%%%%%%%%%%%%%%%%%%%%%%%%%%%
%%%%%%%%%%%%%%%%%%%%%%%%%%%%%%%%%%%%%%%%%%%%%%%%%%%%%%%%%%%%%%%%%%%%%%%%%%%%%%%%
%%%%%%%%%%%%%%%%%%%%%%%%%%%%%%%%%%%%%%%%%%%%%%%%%%%%%%%%%%%%%%%%%%%%%%%%%%%%%%%%

\section{The Isomorphism Conjectures in the torsionfree case}

The outline of this section is:

\begin{itemize}

  \item We introduce \red{spectra} and how they yield \red{homology theories}.
  
  \item We state the \red{Farrell-Jones-Conjecture} and the 
  \red{Baum-Connes Conjecture}
   for torsionfree groups.

  \item We discuss applications of these conjectures such as the 
  \red{Kaplansky Conjecture}
  and the \red{Borel Conjecture}.

  \item We explain that the formulations for torsionfree groups 
  cannot extend to arbitrary groups.

\end{itemize}

Given two pointed spaces $X = (X,x_0)$ and $Y = (Y,y_0)$, 
their \emphred{one-point-union} and their
\emphred{smash product} are defined to be the pointed spaces
\begin{eqnarray*}
\red{X \vee Y} & := & \{(x,y_0) \mid x \in X\} \cup 
\{(x_0,y) \mid y \in Y\} \subseteq X \times Y;
\\
\red{X \wedge Y} & := & (X \times Y)/ (X \vee Y).
\end{eqnarray*}
We have $S^{n+1} \cong S^n \wedge S^1$.

\begin{definition}[\blue{Spectrum}]
A \emphred{spectrum}
$$\bfE = \{(E(n),\sigma(n)) \mid n \in \IZ\}$$
is a sequence of pointed spaces
$\{E(n) \mid n \in \IZ\}$ together with pointed maps
called \emphred{structure maps}
$$\sigma(n) \colon  E(n) \wedge S^1 \longrightarrow E(n+1).$$
A \emphred{map of spectra}
$$\bff \colon  \bfE \to \bfE^{\prime}$$
is a sequence of maps
$f(n) \colon  E(n) \to E^{\prime}(n)$
which are compatible with the structure maps $\sigma(n)$, i.e.,
$f(n+1) \circ \sigma(n) = 
\sigma^{\prime}(n) \circ \left(f(n) \wedge \id_{S^1}\right)$
holds for all $n \in \IZ$.
\end{definition}

\begin{example}[\red{Sphere spectrum}]
The \emphred{sphere spectrum $\bfS$} has as $n$-th space $S^n$ 
and as $n$-th structure map the homeomorphism
$S^n \wedge S^1 \xrightarrow{\cong} S^{n+1}$.
\end{example}

\begin{example}[\red{Suspension spectrum}]
  Let $X$ be a pointed space.  Its \emphred{suspension spectrum
    $\Sigma^{\infty} X$} is given by the sequence of spaces $\{X
  \wedge S^n\mid n \ge 0\}$ with the homeomorphisms $(X \wedge S^n)
  \wedge S^1 \cong X \wedge S^{n+1}$ as structure maps.  We have $\bfS
  = \Sigma^{\infty} S^0$.
\end{example}

\begin{definition}[\blue{$\Omega$-spectrum}]
Given a spectrum $\bfE$, we can consider instead of the structure map 
$\sigma(n) \colon  E(n) \wedge S^1 \to E(n+1)$ its adjoint
$$\sigma'(n) \colon E(n) \to \Omega E(n+1) = \map(S^1,E(n+1)).$$
We call $\bfE$ an \emphred{$\Omega$-spectrum} 
if each map $\sigma'(n)$ is a weak homotopy equivalence.
\end{definition}

\begin{definition}[\blue{Homotopy groups of a spectrum}]
Given a spectrum $\bfE$, define for $n \in \IZ$ its \emphred{$n$-th homotopy group} 
$$\red{\pi_n(\bfE)} := \colim_{k \to \infty} \pi_{k+n}(E(k))$$
to be the abelian group which is given by the colimit over the 
directed system indexed by  $\IZ$ with $k$-th structure map
$$\pi_{k+n}(E(k)) \xrightarrow{\sigma'(k)}  \pi_{k+n}(\Omega E(k+1)) 
=  \pi_{k+n+1}(E(k+1)).$$ 
\end{definition}

Notice that a spectrum can have in contrast to a space non-trivial 
negative homotopy groups.
If $\bfE$ is an $\Omega$-spectrum, then
$\pi_n(\bfE) = \pi_n(E(0))$ for all $n \ge 0$.

\begin{example}[\red{Eilenberg-MacLane spectrum}]
Let $A$ be an abelian group. 
The $n$-th \emphred{Eilenberg-MacLane space $K(A,n)$}
associated to $A$ for $n \ge 0$ is a $CW$-complex with
$\pi_m(K(A,n)) = A$ for $m = n$ and $\pi_m(K(A,n)) = \{0\}$ for $m \not= n$. 

The associated
\emphred{Eilenberg-MacLane spectrum $\bfH(A)$} has as $n$-th space 
$K(A,n)$ and as $n$-th structure map a homotopy equivalence
$K(A,n) \to \Omega K(A,n+1)$.
\end{example}

\begin{example}[\red{Algebraic $K$-theory spectrum}]
For a ring $R$ there is the \emphred{algebraic $K$-theory spectrum $\bfK(R)$}
with the property 
$$\pi_n(\bfK(R)) = K_n(R) \quad \text{ for } n \in \IZ.$$
For its definition see~\cite{Carlsson(2004)}, \cite{Loday(1976)},
and~\cite{Pedersen-Weibel(1985)}.

\end{example}

Next we state the $L$-theoretic version. Since we will not focus on
$L$-theory in these lectures, we will use
$L$-theory as a black box and will later explain its relevance when we
discuss applications. At least we mention that $L$-theory may be
thought of a kind of $K$-theory not for finitely generated projective
modules and their automorphisms but for quadratic forms over finitely
generated projective modules and their automorphisms modulo hyperbolic
forms.

\begin{example}[\red{Algebraic $L$-theory spectrum}]
For a ring with involution $R$ there is the \emphred{algebraic $L$-theory 
spectrum $\bfL^{\langle -\infty\rangle}(R)$}
with the property 
$$\pi_n(\bfL^{\langle -\infty\rangle}(R)) = 
L_n^{\langle-\infty\rangle}(R) \quad \text{ for } n \in \IZ.$$
For its construction we refer for instance to \green{Quinn}~\cite{Quinn(1970)} and  
\green{Ranicki}~\cite{Ranicki(1992)}. 
\end{example}

\begin{example}[\red{Topological $K$-theory spectrum}]
By \emphred{Bott periodicity} there is a homotopy equivalence
$$\beta \colon BU \times \IZ \xrightarrow{\simeq} \Omega^2(BU \times \IZ).$$

The \emphred{topological $K$-theory spectrum} $\bfK^{\topo}$ has in
even degrees $BU \times \IZ$ and in odd degrees $\Omega(BU \times
\IZ)$.  The structure maps are given in even degrees by the map
$\beta$ and in odd degrees by the identity $\id \colon \Omega(BU
\times \IZ) \to \Omega(BU \times \IZ)$.
\end{example}

\begin{definition}[\blue{Homology theory}]
Let $\Lambda$ be a commutative ring, for instance $\IZ$ or $\IQ$.
A \emphred{homology theory $\calh_*$} with values in $\Lambda$-modules is a
covariant functor  from the category of
$CW$-pairs to the category of $\IZ$-graded
$\Lambda$-modules together with natural transformations
$$\partial_n(X,A)\colon \calh_n(X,A) \to \calh_{n-1}(A)$$
for $n \in \IZ$
satisfying the following axioms:
\begin{itemize}
\item 
\red{Homotopy invariance}

\item 
\red{Long exact sequence of a pair}

\item 
\red{Excision}\\[1mm] 
If $(X,A)$ is a $CW$-pair and $f \colon A \to B$ is a cellular map , then
$$\calh_n(X,A) \xrightarrow{\cong} \calh_n(X \cup_f B,B).$$

\item  \red{Disjoint union axiom}
$$\bigoplus_{i \in I} \calh_n(X_i) \xrightarrow{\cong} 
\calh_n\left(\coprod_{i \in I} X_i\right).$$

\end{itemize}
\end{definition}

\begin{definition}[\blue{Smash product}]
Let $\bfE$ be a spectrum and $X$ be a pointed space.
Define the \emphred{smash product} $X \wedge \bfE$ to be the spectrum whose
$n$-th space is $X \wedge E(n)$ and whose $n$-th structure map is
$$X \wedge E(n)  \wedge S^1 \xrightarrow{\id_X \wedge \sigma(n)} X \wedge E(n+1).$$
\end{definition}

\begin{theorem}[\blue{Homology theories and spectra}]
Let $\bfE$ be a spectrum. 
Then we obtain a homology theory \red{$H_*(-;\bfE)$}
by
$$H_n(X,A;\bfE) := 
\pi_n\left(\left(X \cup_A \cone(A)\right) \wedge \bfE\right).$$
It satisfies
$$H_n(\pt;\bfE) = \pi_n(\bfE).$$
\end{theorem}

\begin{example}[\blue{Stable homotopy theory}]
  The homology theory associated to the sphere spectrum $\bfS$ is
  \emphred{stable homotopy $\pi_*^s(X)$}.  The groups $\pi_n^s(\pt)$
  are finite abelian groups for $n \not=0$ by a result of \green{Serre
    (1953)}.  Their structure is only known for small $n$.
\end{example}

\begin{example}[\blue{Singular homology theory with coefficients}]
The homology theory associated to the Eilenberg-MacLane spectrum $\bfH(A)$
is \red{singular homology with coefficients in $A$}.
\end{example}

\begin{example}[\blue{Topological $K$-homology}]
The homology theory associated to the topological $K$-theory spectrum 
$\bfK^{\topo}$ is \emphred{$K$-homology $K_*(X)$}. We have
$$K_n(\pt) \cong  \left\{
\begin{array}{lll}
\IZ & & n \text{ even};
\\
\{0\} & & n \text{ odd}.
\end{array}
\right.
$$
\end{example}

Next we give the formulation of the Farrell-Jones Conjecture for
$K$- and $L$-theory and the Baum-Connes Conjecture 
in the case of a torsionfree group. The general formulations for arbitrary 
groups will require more prerequisites and will be presented later.  
We begin with the $K$-theoretic version. Recall:
\begin{itemize}
\item \red{$K_n(RG)$} is the algebraic $K$-theory of  the group ring $RG$;
\item \red{$\bfK(R)$} is the (non-connective) algebraic $K$-theory spectrum of $R$;
\item $H_n(\pt;\bfK(R)) \cong  \pi_n(\bfK(R)) \cong K_n(R)$ for $n \in \IZ$.
\item \red{BG} is the \red{classifying space} of the group $G$, i.e., 
the base space of the universal $G$-principal $G$-bundle $G \to EG \to BG$. 
Equivalently, $BG = K(G,1)$. The space $BG$ is unique up to homotopy.
\end{itemize}

\begin{conjecture}[\blue{$K$-theoretic Farrell-Jones Conjecture for torsionfree groups}]
\label{con:FJC_for_K_torsionfree}
The \emphred{$K$-theoretic Farrell-Jones Conjecture} with coefficients
in the regular ring $R$ for the torsionfree group $G$ predicts that
the \red{assembly map}
$$H_n(BG;\bfK(R)) \to K_n(RG)$$
is bijective for all $n \in \IZ$.
\end{conjecture}

Recall:
\begin{itemize}
\item \red{$L_n^{\langle -\infty \rangle}(RG)$} is the algebraic
  $L$-theory of $RG$ with decoration $\langle -\infty \rangle$;

\item \red{$\bfL^{\langle -\infty \rangle}(R)$} is the algebraic
  $L$-theory spectrum of $R$ with decoration $\langle -\infty
  \rangle$;

\item $H_n(\pt;\bfL^{\langle -\infty \rangle}(R)) \cong
  \pi_n(\bfL^{\langle -\infty \rangle}(R)) \cong L_n^{\langle -\infty
    \rangle}(R)$ for $n \in \IZ$.
\end{itemize}

\begin{conjecture}[\blue{$L$-theoretic Farrell-Jones Conjecture for torsionfree groups}]
\label{con:FJC_for_L_torsionfree}
The \emphred{$L$-theoretic Farrell-Jones Conjecture} 
with coefficients in the ring with involution $R$ for the torsionfree 
group $G$ predicts that
the \red{assembly map}
$$H_n(BG;\bfL^{\langle -\infty \rangle}(R)) \to L_n^{\langle -\infty \rangle}(RG)$$
is bijective for all $n \in \IZ$. 
\end{conjecture}

Recall:
\begin{itemize}
\item \red{$K_n(BG)$} is the topological $K$-homology of $BG$, where
  $K_*(-) = H_*(-;\bfK^{\topo})$ for $\bfK^{\topo}$ the topological
  $K$-theory spectrum.

\item \red{$K_n(C^*_r(G))$} is the topological $K$-theory of the
  reduced complex group $C^*$-algebra \red{$C^*_r(G)$} of $G$ which is
  the closure in the norm topology of $\IC G$ considered as subalgebra
  of $\calb(l^2(G))$.

\end{itemize}

\begin{conjecture}[\blue{Baum-Connes Conjecture for torsionfree groups}]
\label{con:BCC_torsionfree}
The \emphred{Baum-Connes Conjecture} for the torsionfree group $G$ predicts that
the \red{assembly map}
$$K_n(BG) \to K_n(C^*_r(G))$$
is bijective for all $n \in \IZ$.

There is also a \red{real version} of the Baum-Connes Conjecture
$$KO_n(BG) \to K_n(C^*_r(G;\IR)).$$ 
\end{conjecture}

In order to illustrate the depth of the Farrell-Jones Conjecture
and the Baum-Connes Conjecture, 
we present some conclusions which are interesting in their own right.
 
\begin{notation}
Let \red{$\calfj_K(R)$} and \red{$\calfj_L(R)$} respectively 
be the class of groups which satisfy the 
$K$-theoretic and $L$-theoretic respectively
Farrell-Jones Conjecture for the coefficient ring (with involution) $R$.

Let \red{$\calbc$} be the class of groups which satisfy the Baum-Connes Conjecture. 
\end{notation}

\begin{theorem}[\blue{Lower and middle $K$-theory of group rings in the torsionfree case}] 
\label{the:Lower_and_middle_K-theory_torsionfree}
Suppose that $G$ is torsionfree.

\begin{enumerate} 
\item \label{the:Lower_and_middle_K-theory_torsionfree:n_le_0}
If $R$ is regular and $G \in \calfj_K(R)$, then   
\begin{enumerate}
\item $K_n(RG) = 0$ for $n \le -1$;
\item The change of rings map $K_0(R) \to K_0(RG)$ is bijective;
\item In particular $\widetilde{K}_0(RG)$ is trivial if and only if
      $\widetilde{K}_0(R)$ is trivial.
\end{enumerate}

\item \label{the:Lower_and_middle_K-theory_torsionfree:Wh}
If $G \in \calfj_K(\IZ)$, then the  Whitehead group $\Wh(G)$ is trivial.
\end{enumerate}
\end{theorem}
\begin{proof}
The idea of the proof is to study the \red{Atiyah-Hirzebruch spectral sequence}.
It converges to $H_n(BG;\bfK(R))$ which is isomorphic to $K_n(RG)$ by the
assumption that $G$ satisfies the Farrell-Jones Conjecture. The $E^2$-term is given by
$$E^2_{p,q} = H_p(BG,K_q(R)).$$
\eqref{the:Lower_and_middle_K-theory_torsionfree:n_le_0}
Since $R$ is regular by assumption, we get $K_q(R) = 0$ for $q \le -1$.
Hence the spectral sequence is a first quadrant spectral sequence.
This implies $K_n(RG) \cong H_n(BG;\bfK(R)) = 0$ for $n \le -1$ and that
the edge homomorphism yields an isomorphism
$$K_0(R) = H_0(\pt,K_0(R)) \xrightarrow{\cong} H_0(BG;\bfK(R)) \cong K_0(RG).$$
\\[1mm]
\eqref{the:Lower_and_middle_K-theory_torsionfree:Wh}
We have $K_0(\IZ) = \IZ$ and $K_1(\IZ) = \{\pm 1\}$.
We get an exact sequence
\begin{multline*}
0 \to H_0(BG;K_1(\IZ)) = \{\pm 1\} \to H_1(BG;\bfK(\IZ)) \cong K_1(\IZ G)
\\
\to H_1(BG;K_0(\IZ)) = G/{[G,G]}\to 0.
\end{multline*}
This implies
$\Wh(G) := K_1(\IZ G)/\{\pm g\mid g \in G\} = 0$.
\end{proof}

We summarize that we get for a torsionfree group $G \in \calfj_K(\IZ)$:
\begin{enumerate}

\item $K_n(\IZ G) = 0$ for $n \le -1$;

\item $\widetilde{K}_0(\IZ G) = 0$;

\item $\Wh(G) = 0$;

\item Every finitely dominated $CW$-complex $X$ with $G = \pi_1(X)$ is homotopy equivalent
      to a finite $CW$-complex;

\item Every compact $h$-cobordism $W$ of dimension $\ge 6$ with $\pi_1(W) \cong G$ 
      is trivial;

\item If $G$ belongs to $\calfj_K(\IZ)$, then it is of type FF if and only if
      it is of type FP (\green{Serre's} \red{problem}).     

\end{enumerate}

%%%%%%%%%%%%%%%%%%%%%%%%%%%%%%%%%%%%%%%%%%%%%%%%%%%%%%%%%%%%%%%%%%%%%%%%%%%%%%%%%%%%%%%%%%%%%%%

\begin{conjecture}[\blue{Kaplansky Conjecture}] 
The \red{Kaplansky Conjecture} says for a torsionfree group $G$ and
an integral domain $R$ that $0$ and $1$ are the only idempotents in $RG$.
\end{conjecture}

In the next theorem we will use the notion of a \emphred{sofic group} that was introduced by
Gromov and originally called \emphred{subamenable group}.  Every residually amenable group is
sofic but the converse is not true.  The class of sofic groups is closed under taking
subgroups, direct products, free amalgamated products, colimits and inverse limits, and,
if $H$ is a sofic normal subgroup of $G$ with amenable quotient $G/H$, then $G$ is sofic.
This is a very general notion, e.g., no group is known which is not sofic. 
For more information about the notion of a sofic group we refer
to~\cite{Elek-Szabo(2006)}. The next result is taken 
from~\cite[Theorem~0.12]{Bartels-Lueck-Reich(2007appl)}.

\begin{theorem}[\blue{The Farrell-Jones Conjecture and the Kaplansky Conjecture},
\green{Bartels-L\"uck-Reich(2007)}] 

Let $F$ be a skew-field and let $G$ be a group with $G \in \calfj_K(F)$.
Suppose that one of the following conditions is satisfied:

\begin{enumerate} 
\item $F$ is commutative and has characteristic zero and $G$ is torsionfree;

\item $G$ is torsionfree and sofic;

\item
 The characteristic of $F$ is $p$, all finite subgroups of $G$ are 
$p$-groups and $G$ is sofic.
\end{enumerate}

Then $0$ and $1$ are the only idempotents in $FG$.
\end{theorem}
\begin{proof}
Let $p$ be an idempotent in $FG$. We want to show $p \in \{0,1\}$.
Denote by $\epsilon \colon FG \to F$ the augmentation homomorphism
sending $\sum_{g \in G} r_g \cdot g$ to $\sum_{g \in G} r_g$.  Obviously 
$\epsilon(p) \in F$ is $0$ or $1$. Hence it  suffices to show
$p = 0$ under the assumption that $\epsilon(p) = 0$.  

Let $(p)  \subseteq FG$ be the ideal generated by $p$ which is a finitely
generated projective $FG$-module. Since $G \in \calfj_K(F)$, we can conclude that 
$$i_* \colon K_0(F)  \otimes_{\IZ} \IQ \to K_0(FG) \otimes_{\IZ} \IQ$$ 
is surjective. Hence we can find a finitely generated projective
$F$-module $P$ and integers $k,m,n \ge 0$ satisfying
$$(p)^k \oplus FG^m \cong_{FG} i_*(P) \oplus FG^n.$$
If we now apply $i_{\ast} \circ \epsilon_{\ast}$ and use 
$\epsilon \circ i = \id$, $i_{\ast} \circ \epsilon_{\ast} ( FG^l ) \cong FG^l$ and $\epsilon(p) =0$
we obtain
$$FG^m \cong i_{\ast} (P) \oplus FG^n.$$
Inserting this in the first equation yields
$$(p)^k \oplus FG^m \cong FG^m.$$ 
Our assumptions on $F$ and $G$ imply that $FG$ is \red{stably finite}, i.e.,
if $A$ and $B$ are square matrices over $FG$ with $AB= I$, then $BA = I$.
This implies  $(p)^k = 0$ and hence $p = 0$.
\end{proof}

\begin{theorem}[\blue{The Baum-Connes Conjecture and the Kaplansky Conjecture}] 
Let $G$ be a torsionfree group with $G \in \calbc$.
Then $0$ and $1$ are the only idempotents in $C^*_r(G)$ and in particular in $\IC G$.
\end{theorem}
\begin{proof}
There is a trace map
$$\tr \colon C_r^*(G) \to \IC$$
which sends $f \in C^*_r(G) \subseteq \calb(l^2(G))$ to $\langle f(e),e \rangle_{l^2(G)}$.
The \red{$L^2$-index theorem} due to \green{Atiyah (1976)} 
(see~\cite{Atiyah(1976)}) shows that
the composite 
$$K_0(BG) \to K_0(C^*_r(G)) \xrightarrow{\tr} \IC$$
coincides with
$$K_0(BG) \xrightarrow{K_0(\pr)} K_0(\pt) = \IZ \to \IC.$$
Hence $G \in \calbc$ implies $\tr(p) \in \IZ$. 
Since $\tr(1) = 1$, $\tr(0) = 0$, $0 \le p \le 1$ and $p^2 = p$, we get
$\tr(p) \in \IR$ and $0 \le \tr(p) \le 1$.
We conclude $\tr(0) = \tr(p)$ or $\tr(1)=\tr(p)$. 
Since the trace $\tr$ is faithful, this implies already $p =0$ or $p = 1$.
\end{proof}

The next conjecture is one of the basic conjectures about the classification of 
topological manifolds.

\begin{conjecture}[\blue{Borel Conjecture}]
\label{con:borel}
The \emphred{Borel Conjecture for $G$} predicts for two closed aspherical manifolds
$M$ and $N$ with $\pi_1(M) \cong \pi_1(N) \cong  G$ that any homotopy equivalence
$M \to N$ is homotopic to a homeomorphism and in particular that 
$M$ and $N$ are homeomorphic.
\end{conjecture}

\begin{remark}[\blue{Borel versus Mostow}]
The Borel Conjecture can be viewed as the topological version of \red{Mostow rigidity}.
A special case of Mostow rigidity says that any homotopy equivalence between 
closed hyperbolic manifolds of dimension $\ge 3$ is homotopic to an 
isometric diffeomorphism.
\end{remark}

\begin{remark}[\blue{The Borel Conjecture fails in the smooth category}]
The Borel Conjecture is not true in the smooth category by results of
\green{Farrell-Jones}~\cite{Farrell-Jones(1989)}, i.e., there
exists aspherical closed manifolds
which are homeomorphic but not diffeomorphic. 
\end{remark}

\begin{remark}[\blue{Topological rigidity for non-aspherical manifolds}]
There are also non-aspherical manifolds which are topologically rigid in the sense
of the Borel Conjecture (see \green{Kreck-L\"uck}~\cite{Kreck-Lueck(2005nonasp)}).
\end{remark}

\begin{theorem}[\blue{The Farrell-Jones Conjecture and the Borel Conjecture}]
\label{the:The_Farrell-Jones_Conjecture_and_the_Borel_Conjecture}
If the $K$- and $L$-theoretic Farrell-Jones Conjecture hold for $G$ 
in the case $R = \IZ$,
then the Borel Conjecture is true in dimension $\ge 5$ and in dimension $4$ 
if $G$ is good in the sense of Freedman.
\end{theorem}

\begin{remark}[\blue{The Borel Conjecture in dimension $\le 3$}] 
\red{Thurston's Geometrization Conjecture} implies the Borel Conjecture
in dimension three. The Borel Conjecture in dimension one and two is obviously true.
\end{remark}

Next we give some explanations about the proof of 
Theorem~\ref{the:The_Farrell-Jones_Conjecture_and_the_Borel_Conjecture}.

  \begin{definition}[\blue{Structure set}]
  
    The \emphred{structure set} $S^{top} (M)$ of a manifold $M$
    consists of equivalence classes of orientation preserving homotopy
    equivalences $N \to M$ with a manifold $N$ as source.

    Two such homotopy equivalences $f_0 \colon N_0 \to M$ and $f_1 \colon
    N_1 \to M$ are equivalent if there exists a homeomorphism $g
    \colon N_0 \to N_1$ with $f_1 \circ g\simeq f_0$.
  \end{definition}

The next result follows directly from the definitions.
\begin{theorem} 
\label{the:Borel_and_structure_set}
The Borel Conjecture holds for a closed manifold $M$
if and only if  $\cals^{\topo}(M)$ consists of one element.
\end{theorem}

Let \red{$\bfL\langle 1 \rangle$} be the \emphred{$1$-connective cover} of the
$L$-theory spectrum $\bfL$. It is characterized by the following property.
There is a natural map
of spectra $\bfL\langle 1 \rangle \to \bfL$ which induces an isomorphism on the homotopy
groups in dimensions $n \ge 1$ and the homotopy groups of $\bfL\langle 1 \rangle$ vanish 
in dimensions $n \le 0$.

\begin{theorem}[\green{Ranicki (1992)}]
\label{the:Ranickis_surgery_sequence}
There is an exact sequence of abelian groups,
 called \red{algebraic surgery exact sequence},
 for an $n$-dimensional closed manifold $M$
\begin{multline*}
\ldots  \xrightarrow{\sigma_{n+1}}  H_{n+1}(M;\bfL\langle 1 \rangle)
\xrightarrow{A_{n+1}} L_{n+1}(\IZ\pi_1(M)) \xrightarrow{\partial_{n+1}}
\\
 \cals^{\topo}(M) \xrightarrow{\sigma_n}
H_{n}(M;\bfL\langle 1 \rangle) \xrightarrow{A_{n}} L_{n}(\IZ\pi_1(M)) 
\xrightarrow{\partial_{n}}  \ldots
\end{multline*}

It can be identified with the classical geometric surgery 
sequence due to \green{Sullivan and Wall} in high dimensions.
\end{theorem}
\begin{proof}
See~\cite[Definition~15.19 on page~169 and Theorem~18.5 on page~198]{Ranicki(1992)}.
\end{proof}

The $K$-theoretic version of the Farrell-Jones Conjecture
ensures that we do not have to deal with decorations, e.g., it does not matter
if we consider $\bfL$ or $\bfL^{\langle - \infty\rangle}$. (This follows from the so called 
\emphred{Rothenberg sequences}). The $L$-theoretic version of the
Farrell-Jones Conjecture implies that
$H_{n}(M;\bfL) \to L_{n}(\IZ\pi_1(M))$ is bijective for all $n \in \IZ$.
An easy spectral sequence argument shows
that $H_{k}(M;\bfL\langle 1 \rangle) \to H_{k}(M;\bfL)$ is bijective for $k \ge n+1$
and injective for $k = n$. For $k = n$ and $k = n+1$
the map $A_k$ is the composite of the map 
$H_{k}(M;\bfL\langle 1 \rangle) \to H_{k}(M;\bfL)$ with the map
$H_{k}(M;\bfL) \to L_{k}(\IZ\pi_1(M))$.
Hence $A_{n+1}$ is surjective and $A_n$ is injective. 
Theorem~\ref{the:Ranickis_surgery_sequence} implies that
$\cals^{\topo}(M)$ consist of one element.
Now Theorem~\ref{the:The_Farrell-Jones_Conjecture_and_the_Borel_Conjecture} follows
from Theorem~\ref{the:Borel_and_structure_set}.

More information on the Borel Conjecture can be found
for instance in~\cite{Farrell(1996)},
\cite{Farrell-Jones(1989)}, \cite{Farrell-Jones(1990)}, \cite {Farrell-Jones(1993c)},
\cite{Farrell-Jones(1998)}, \cite{Farrell(2002)}
\cite{Ferry-Ranicki-Rosenberg(1995)}, \cite{Kreck-Lueck(2005)}, \cite{Lueck(2002c)},
\cite{Lueck-Reich(2005)}.

Next we explain that the versions of the Farrell-Jones and the Baum-Connes Conjecture
above cannot be true if we drop the assumption that $G$ is torsionfree 
or that $R$ is regular

\begin{example}[\blue{The condition torsionfree is essential}]
The versions of the Farrell-Jones Conjecture and the Baum-Connes Conjecture above
become false for finite groups unless the group is trivial.
For instance the version of the Baum-Connes Conjecture  above
would predict for a finite group $G$
$$K_0(BG) \cong K_0(C^*_r(G)) \cong R_{\IC}(G).$$
However, $K_0(BG) \otimes_{\IZ} \IQ \cong_{\IQ} K_0(\pt) \otimes_{\IZ} \IQ \cong_{\IQ} \IQ$
and $R_{\IC}(G) \otimes_{\IZ} \IQ \cong_{\IQ} \IQ$ holds if and only if $G$ is trivial.
\end{example}

\begin{example}[\blue{The condition regular is essential}]
If $G$ is torsionfree, then the version 
of the $K$-theoretic Farrell-Jones Conjecture predicts
\begin{multline*}
H_n(B\IZ;\bfK(R)) = H_n(S^1;\bfK(R)) = H_n(\pt;\bfK(R)) \oplus H_{n-1}(\pt;\bfK(R))
\\
= K_n(R) \oplus K_{n-1}(R) \cong K_n(R\IZ).
\end{multline*}
In view of the Bass-Heller-Swan decomposition this is only possible
if $\NK_n(R)$ vanishes which is true for regular rings $R$ but not for general rings $R$.
\end{example}

Next we want to discuss what we may have to take into account if we want to give a
formulation of the Farrell-Jones and the Baum-Connes Conjecture which may have a chance to
be true for all groups.

\begin{remark}[\blue{Assembly}]
  For a field $F$ of characteristic zero and some groups $G$ one knows that there is an
  isomorphism
$$\colim_{\substack{H \subseteq G\\|H| <  \infty}} K_0(FH) \xrightarrow{\cong} K_0(FG).$$
This indicates that one has at least to take into account the values for all finite
subgroups to assemble $K_n(FG)$.
\end{remark}

\begin{remark}[\blue{Degree Mixing}]
The Bass-Heller-Swan decomposition shows that
the $K$-theory of finite subgroups in degree $m \le n$ can affect 
the $K$-theory in degree $n$
and that at least in the Farrell-Jones setting finite subgroups are not enough.
\end{remark}

\begin{remark}[\blue{No Nil-phenomena occur in the Baum-Connes setting}]
In the Baum-Connes setting Nil-phenomena do not appear. Namely, a special 
case of a result due to 
\green{Pimsner-Voiculescu}~\cite{Pimsner-Voiculescu(1982)} says
$$K_n(C^*_r(G \times \IZ)) \cong K_n(C^*_r(G)) \oplus  K_{n-1}(C^*_r(G)).$$
\end{remark}

\begin{remark}[\blue{Homological behavior}]
There is still a lot of homological behavior known for $K_*(C^*_r(G))$.
For instance there exists a long exact \red{Mayer-Vietoris sequence} associated
to  amalgamated products $G_1 \ast_{G_0} G_2$  by 
\green{Pimsner-Voiculescu}~\cite{Pimsner-Voiculescu(1982)}.
\begin{multline*}
\dots \to K_n(C^*_r(G_0)) \to K_n(C^*_r(G_1)) \oplus K_n(C^*_r(G_2)) \to K_n(C^*_r(G)) 
\\
\to K_{n-1}(C^*_r(G_0)) \to K_{n-1}(C^*_r(G_1)) \oplus K_{n-1}(C^*_r(G_2)) \to \cdots 
\end{multline*}
This is very similar to the corresponding Mayer-Vietoris sequence in group homology theory
\begin{multline*}
\dots \to H_n(G_0) \to H_n(G_1)) \oplus H_n(G_2) \to H_n(G) 
\\
\to H_{n-1}(G_0) \to H_{n-1}(G_1) \oplus H_{n-1}G_2) \to \cdots 
\end{multline*}
It comes from the fact that there is a model for $BG$ which contains
$BG_0$, $BG_1$ and $BG_2$ as $CW$-subcomplexes such that $BG = BG_1 \cup BG_2$ and
$BG_0 = BG_1 \cap BG_2$.

An analogous similarity exists for the \red{Wang-sequence} associated to 
a semi-direct product $G \rtimes \IZ$

Similar versions of the Mayer-Vietoris sequence and the Wang sequence 
in algebraic $K$-and $L$-theory of group rings are 
due to \green{Cappell (1974)} and \green{Waldhausen (1978)} 
provided one makes certain assumptions on $R$ or 
ignores certain Nil-phenomena.
\end{remark}

\begin{question}[\blue{Classifying spaces for families}]

Is there a version \red{$\EGF{G}{\calf}$} of the classifying space $EG$ 
which takes the structure
of the family of finite subgroups or other families $\calf$ of subgroups
into account and can be 
used for a general formulation of the Farrell-Jones Conjecture?
\end{question}

\begin{question}[\blue{Equivariant homology theories}]

Can one define appropriate $G$-homology theories $\calh^G_*$ that are in some sense
computable and yield when applied to $\EGF{G}{\calf}$ a term which 
potentially is isomorphic to the groups
$K_n(RG)$, $L_n^{-\langle \infty\rangle}(RG)$ or $K_n(C^*_r(G))$?

In the torsionfree case they should reduce
to $H_n(BG;\bfK(R))$, $H_n(BG;\bfL^{-\langle \infty\rangle})$ and $K_n(BG)$.
\end{question}

%%%%%%%%%%%%%%%%%%%%%%%%%%%%%%%%%%%%%%%%%%%%%%%%%%%%%%%%%%%%%%%%%%%%%%%%%%%%%%%%
%%%%%%%%%%%%%%%%%%%%%%%%%%%%%%%%%%%%%%%%%%%%%%%%%%%%%%%%%%%%%%%%%%%%%%%%%%%%%%%%
%%%%%%%%%%%%%%%%%%%%%%%%%%%%%%%%%%%%%%%%%%%%%%%%%%%%%%%%%%%%%%%%%%%%%%%%%%%%%%%%

\section{Classifying spaces for families }

The outline of this section is: 

\begin{itemize}

  \item  We introduce the notion of the \red{classifying space of a family
  $\calf$ of subgroups $\EGF{G}{\calf}$} and \red{$\JGF{G}{\calf}$}.
  
  \item In the case, where $\calf$ is the family $\Com$ of compact
        subgroups, we present some
       nice geometric models for $\EGF{G}{\calf}$ and explain
       $\EGF{G}{\calf} \simeq \JGF{G}{\calf}$.
  
  \item We discuss \red{finiteness properties} of these classifying spaces.

  \end{itemize}

The material of this section is an extract of the survey article
by \green{L\"uck}~\cite{Lueck(2005s)},
where more information and proofs of the results stated below are given.

In this section group means \red{locally compact Hausdorff topological group with a
countable basis for its topology}, unless explicitly stated differently.

\begin{definition}[\blue{$G$-$CW$-complex}]
A \emphred{$G$-$CW$-complex}
$X$ is a $G$-space together with a $G$-invariant filtration
$$\emptyset = X_{-1} \subseteq X_0  \subseteq \ldots \subseteq
X_n \subseteq \ldots \subseteq\bigcup_{n \ge 0} X_n = X$$
such that $X$ carries the \red{colimit topology}
with respect to this filtration,
and $X_n$ is obtained
from $X_{n-1}$ for each $n \ge 0$ by \red{attaching
equivariant $n$-dimensional cells}, i.e., there exists
a $G$-pushout
$$
\xymatrix@!C=10em{
 \coprod_{i \in I_n} G/H_i \times S^{n-1}
 \ar[r]^-{\coprod_{i \in I_n} q_i^n}
 \ar[d]
 &
 X_{n-1}
 \ar[d]
  \\
  \coprod_{i \in I_n} G/H_i \times D^{n}
  \ar[r]^-{\coprod_{i \in I_n} Q_i^n}
  &
  X_n
 }
 $$
\end{definition}

\begin{example}[\blue{Simplicial actions}]
\label{exa:simplicial_actions}
Let $X$ be a (geometric) simplicial complex. 
Suppose that $G$ acts simplicially on $X$.
Then $G$ acts simplicially also on the \red{barycentric
subdivision $X'$}, and all isotropy groups are open and closed. 
The $G$-space $X'$ inherits  the structure of a $G$-$CW$-complex.
\end{example}

\begin{definition}[\blue{Proper $G$-action}]
A $G$-space $X$ is called \emphred{proper}
if for each pair of points $x$ and $y$ in $X$ there are open neighborhoods
$V_x$ of $x$ and $W_y$ of $y$ in $X$ such that the closure of the subset
$\{g \in G \mid gV_x \cap W_y \not= \emptyset\}$ of $G$ is compact.
\end{definition}

\begin{lemma}
\begin{enumerate}
\item A proper $G$-space has always compact isotropy groups.
\item
A $G$-$CW$-complex $X$ is proper
if and only if all its isotropy groups are compact.
\end{enumerate}
\end{lemma}
\begin{proof}
See~\cite[Theorem~1.23 on page~19]{Lueck(1989)}.
\end{proof}

\begin{example}[\blue{Smooth actions}]
Let $G$ be a Lie group acting properly and smoothly on a smooth
manifold $M$. Then $M$ inherits the structure of $G$-$CW$-complex
(see \green{Illman}~\cite{Illman(2000)}).
\end{example}

\begin{definition}[\blue{Family of subgroups}]
\label{def:family_of_subgroups}
A \emphred{family $\calf$ of subgroups}
of $G$ is a set of (closed) subgroups of $G$ which is closed under conjugation
and finite intersections.
\end{definition}

Examples for $\calf$ are:
\\
\begin{tabular}{lll}
\red{$\TR$}
& = & \{trivial subgroup\};
\\

\red{$\Fin$}
& = & \{finite subgroups\};
\\

\red{$\VCyc$}
& = & \{virtually cyclic subgroups\};
\\

\red{$\Com$}
& = & \{compact subgroups\};
\\

\red{$\Comop$}
& = & \{compact open subgroups\};
\\

\red{$\All$}
& = & \{all subgroups\}.
\end{tabular}

\begin{definition}[\blue{Classifying $G$-$CW$-complex for a family of subgroups}]
\label{def:EGF(G)(F)}
Let $\calf$ be a family of subgroups of $G$.
A model for the \emphred{classifying $G$-$CW$-complex for the family $\calf$}
is a $G$-$CW$-complex \red{$\EGF{G}{\calf}$} which has the following
properties:
\begin{enumerate}
\item All isotropy groups of $\EGF{G}{\calf}$ belong to $\calf$;

\item For any $G$-$CW$-complex $Y$, whose isotropy groups belong to $\calf$,
there is up to $G$-homotopy precisely one $G$-map $Y \to \EGF{G}{\calf}$.

\end{enumerate}

We abbreviate $\red{\eub{G}} := \EGF{G}{\Com}$
and call it the \emphred{universal $G$-$CW$-complex for proper $G$-actions}.
We also write $\red{EG} = \EGF{G}{\TR}$.
\end{definition}

\begin{theorem} [\blue{Homotopy characterization of $\EGF{G}{\calf}$}]
\label{the:G-homotopy_characterization_of_EGF(G)(calf)}
Let $\calf$ be a family of subgroups.
\begin{enumerate}

\item \label{the:G-homotopy_characterization_of_EGF(G)(calf):existence}
There exists a model for $\EGF{G}{\calf}$ for any family $\calf$;

\item \label{the:G-homotopy_characterization_of_EGF(G)(calf):characterization}
A $G$-$CW$-complex $X$ is a model for $\EGF{G}{\calf}$ if and only if 
all its isotropy groups belong to $\calf$ and for each $H \in\calf$ the $H$-fixed point
set $X^H$ is weakly contractible.
\end{enumerate}

\end{theorem}

\begin{example}[\blue{$\EGF{G}{\All}$}]
A model for $\EGF{G}{\All}$ is $G/G$;
\end{example}

\begin{example}[\blue{Universal principal $G$-bundle}]
The projection $EG \to BG:= G\backslash EG$ is the \red{universal $G$-principal bundle}
for $G$-$CW$-complexes.
\end{example}

\begin{example}[\blue{Infinite dihedral group}]
  Let $D_{\infty} = \IZ \rtimes \IZ/2 = \IZ/2 \ast \IZ/2$ be the infinite dihedral
  group. A model for $ED_{\infty}$ is the universal covering of $\IR\IP^{\infty} \vee
  \IR\IP^{\infty}$.  A model for $\eub{D_{\infty}}$ is $\IR$ with the obvious
  $D_{\infty}$-action.  Notice that every model for $ED_{\infty}$ or $BD_{\infty}$ must be
  infinite-dimensional, whereas there exists a cocompact $1$-dimensional model for
  $\eub{D_{\infty}}$.
\end{example}

\begin{lemma} \label{lem:totally_disconnected_groups_and_calcom}
If $G$ is totally disconnected, then $\EGF{G}{\Comop} = \eub{G}$.
\end{lemma}

\begin{definition}[\blue{$\calf$-numerable $G$-space}]
An \emphred{$\calf$-numerable $G$-space}
is a $G$-space, for which there exists an open covering $\{U_i \mid i
\in I\}$ by $G$-subspaces satisfying:
\begin{enumerate}

\item For each $i \in I$ there exists a
$G$-map $U_i \to G/G_i$ for some $G_i \in \calf$;

\item  There is a
locally finite partition of unity $\{e_i \mid i \in I\}$ subordinate
to $\{U_i \mid i \in I\}$ by
$G$-invariant functions $e_i \colon X \to [0,1]$.
\end{enumerate}
\end{definition}

Notice that we do not demand that the isotropy groups of a
$\calf$-numerable $G$-space belong to $\calf$.

If $f \colon X \to Y$
is a $G$-map and $Y$ is $\calf$-numerable, then $X$ is also $\calf$-numerable.

\begin{lemma}
A $G$-$CW$-complex is $\calf$-numerable if and only if each isotropy group
appears as a subgroup of an element in $\calf$.
\end{lemma}

\begin{definition}[\blue{Classifying numerable $G$-space for a family of subgroups}]
\label{def:Classifying_numerable_G-space_for_a_family_of_subgroups}
Let $\calf$ be a family of subgroups of $G$. 
A model \red{$\JGF{G}{\calf}$} for the 
\emphred{classifying numerable $G$-space for the family of subgroups $\calf$}
is a $G$-space which has  the following properties: 
\begin{enumerate}
\item  $\JGF{G}{\calf}$ is $\calf$-numerable;
\item For any $\calf$-numerable
$G$-space $X$ there is up to $G$-homotopy precisely one $G$-map $X \to \JGF{G}{\calf}$.
\end{enumerate}

We abbreviate $\red{\jub{G}} := \JGF{G}{\Com}$
and call it the \emphred{universal numerable $G$-space for proper $G$-actions}
or briefly \emph{the universal space for proper $G$-actions}.
We also write $\red{JG} = \JGF{G}{\TR}$.
\end{definition}

\begin{theorem}[\blue{Homotopy characterization of $\JGF{G}{\calf}$}]
\label{the:G-homotopy_characterization_of_JGF(G)(calf)}
Let $\calf$ be a family of subgroups.
\begin{enumerate}

\item \label{the:G-homotopy_characterization_of_JGF(G)(calf):existence}
For any family $\calf$ there exists a model for $\JGF{G}{\calf}$
whose isotropy groups belong to $\calf$;

\item \label{the:G-homotopy_characterization_of_JGF(G)(calf):characterization}
Let $X$ be an $\calf$-numerable $G$-space. Equip $X \times X$ with the diagonal action and
let $\pr_i \colon X \times X \to X$ be the projection onto the $i$-th factor for $i = 1,2$.
Then $X$ is a model for $\JGF{G}{\calf}$ if and only if for each $H \in \calf$ there 
is $x \in X$ with $H \subseteq G_x$ and $\pr_1$ and $\pr_2$ are $G$-homotopic.

\item \label{the:G-homotopy_characterization_of_JGF(G)(calf):necessary_condition}
For $H \in \calf$ the $H$-fixed point set $\JGF{G}{\calf}^H$ is
contractible.
\end{enumerate}
\end{theorem}
\begin{proof}
See~\cite[Theorem~2.5]{Lueck(2005s)}.
\end{proof}

\begin{example}[\blue{Universal $G$-principal bundle}]
The projection $JG \to G\backslash JG$ is the \red{universal $G$-principal bundle}
for numerable free proper $G$-spaces.
\end{example}

\begin{theorem}[\blue{Comparison of $\EGF{G}{\calf}$ and $\JGF{G}{\calf}$}, 
\green{L\"uck (2005)}]
\label{the:comparison_EGFGcalf_and_JGFGcalf}
\ \newline
\vspace*{-4mm}
\begin{enumerate}

\item There is up to $G$-homotopy precisely one $G$-map
$$\phi \colon \EGF{G}{\calf} \to \JGF{G}{\calf};$$

\item
It is a $G$-homotopy equivalence if one of the following conditions is satisfied:

\begin{enumerate}

\item Each element in $\calf$ is open and closed;

\item $G$ is discrete;

\item $\calf$ is $\Com$;
\end{enumerate}

\item Let $G$ be totally disconnected. Then $EG \to JG$ is a $G$-homotopy equivalence
if and only if $G$ is discrete.
\end{enumerate}
\end{theorem}
\begin{proof}
See~\cite[Theorem~3.7]{Lueck(2005s)}.
\end{proof}

Next we want to illustrate that the space $\eub{G} = \jub{G}$ often has
\red{very nice geometric models}
and \red{appear naturally in many interesting situations}.

Let \red{$C_0(G)$} be the Banach space of complex valued functions of $G$ 
vanishing at infinity with the supremum-norm. 
The group $G$ acts isometrically on $C_0(G)$ by
$(g\cdot f)(x) := f(g^{-1}x)$ for $f \in C_0(G)$ and $g,x \in G$.
Let \red{$PC_0(G)$} be the subspace of $C_0(G)$ consisting of functions $f$ such that
$f$ is not identically zero and has non-negative real numbers as values.

\begin{theorem}[\blue{Operator theoretic model}, \green{Abels (1978)}]
The $G$-space $PC_0(G)$ is a model for $\jub{G}$.
\end{theorem}
\begin{proof}
See~\cite[Theorem~2.4]{Abels(1978)}.
\end{proof}

\begin{theorem}
Let $G$ be discrete.  A model for $\jub{G}$  is the space
$$
X_G = \bigg\{f \colon G \to [0,1] \;\bigg| \;f \text{ has finite support, }
\sum_{g \in G} f(g) = 1\bigg\}
$$
with the topology coming from the supremum norm.
\end{theorem}

\begin{theorem}[\blue{Simplicial Model}]
Let $G$ be discrete.  Let $P_{\infty}(G)$ be the geometric realization of the simplicial
set whose $k$-simplices consist of $(k+1)$-tupels
$(g_0,g_1, \ldots , g_k)$ of elements $g_i$ in $G$. 

Then $P_{\infty}(G)$ is a model for $\eub{G}$.
\end{theorem}

\begin{remark}[\blue{Comparison of $X_G$ and $P_{\infty}(G)$}]
The spaces $X_G$ and $P_{\infty}(G)$ have the same underlying sets
but in general they have different topologies.
The identity map induces a $G$-map $P_{\infty}(G) \to X_G$
which is a $G$-homotopy equivalence, but
in general not a $G$-homeomorphism.
\end{remark}

\begin{theorem}[\blue{Almost connected groups},  \green{Abels (1978)}.]
Suppose that $G$ is \red{almost connected}, i.e., the group $G/G^0$ is compact for
$G^0$ the component of the identity element.

Then $G$ contains a
maximal compact subgroup $K$ which is unique up to conjugation,
and the $G$-space $G/K$ is a model for $\jub{G}$.
\end{theorem}
\begin{proof}
See~\cite[Corollary~4.14]{Abels(1978)}.
\end{proof}

As a special case we get:
\begin{theorem}[\blue{Discrete subgroups of almost connected Lie groups}]
\label{the:Discrete_subgroups_of_almost_connected_Lie_groups}
Let $L$ be a Lie group with finitely many path components.

Then $L$ contains a maximal compact subgroup $K$ which is unique up to conjugation, and
the $L$-space $L/K$ is a model for $\eub{L}$.

If $G \subseteq L$ is a discrete subgroup of $L$, then
$L/K$ with the obvious left $G$-action is a finite dimensional
$G$-$CW$-model for $\eub{G}$.
\end{theorem}

\begin{theorem}[\blue{Actions on $\CAT(0)$-spaces}]
Let $G$ be a (locally compact second countable Hausdorff) topological group.
Let $X$ be a proper $G$-$CW$-complex.
Suppose that $X$ has the structure of a complete simply connected $\CAT(0)$-space 
for which $G$ acts by isometries.

Then $X$ is a model for $\eub{G}$.
\end{theorem}
\begin{proof}
By~\cite[Corollary~II.2.8 on page~179]{Bridson-Haefliger(1999)} 
the $K$-fixed point set
of $X$ is a non-empty convex subset of $X$ and hence contractible for any
compact subgroup $K \subset G$. 
\end{proof}

\begin{remark} \label{rem:non-positive_curvature}
The result above contains as special case
\red{isometric $G$ actions on
simply-connected complete Riemannian manifolds with non-positive sectional curvature}
and \red{$G$-actions on trees}.
\end{remark}

Let $\Sigma$ be an \emphred{affine building}
sometimes also called  \emphred{Euclidean building}.
This is a simplicial complex together with a system of subcomplexes called 
\emphred{apartments}
satisfying the following axioms:

\begin{enumerate}

\item  Each apartment is isomorphic to an affine Coxeter complex;

\item  Any two simplices of $\Sigma$ are contained in some common apartment;
 
\item  If two apartments both contain two simplices $A$ and $B$ of $\Sigma$,
then there is an isomorphism of one apartment onto the other which fixes the two
simplices $A$ and $B$ pointwise. 

\end{enumerate}

The precise definition of an \emphred{affine Coxeter complex}, which is sometimes called
also \emphred{Euclidean Coxeter complex}, can be found 
in~\cite[Section 2 in Chapter VI]{Brown(1998)}, where also more 
information about affine buildings is given. An affine building comes with metric
$d \colon \Sigma \times \Sigma \to [0,\infty)$ 
which is non-positively curved and complete. The building with this metric is a
CAT(0)-space. A simplicial automorphism
of $\Sigma$ is always an isometry with respect to $d$.
For two points $x,y$ in the affine building
there is a unique line segment $[x,y]$ joining $x$ and $y$.
It is the set of points $\{z \in \Sigma \mid d(x,y) = d(x,z) + d(z,y)\}$.
For $x,y \in \Sigma$ and $t \in [0,1]$ let $tx + (1-t)y$ be the point $z \in \Sigma$
uniquely determined by the property that $d(x,z) = td(x,y)$ and $d(z,y) = (1-t)d(x,y)$.
Then the map 
$$r \colon \Sigma \times \Sigma \times [0,1] \to \Sigma, 
\hspace{5mm} (x,y,t) \mapsto tx + (1-t)y$$ 
is continuous. This implies that
$\Sigma$ is contractible. All these facts are taken
from~\cite[Section~3 in Chapter~VI]{Brown(1998)}
and~\cite[Theorem~10A.4 on page~344]{Bridson-Haefliger(1999)}.

Suppose that the group $G$ acts on $\Sigma$ by isometries.  If $G$ maps a non-empty
bounded subset $A$ of $\Sigma$ to itself, then the $G$-action has a fixed point
(see~\cite[Theorem~1 in Section~4 in Chapter~VI on page~157]{Brown(1998)}).  Moreover the
$G$-fixed point set must be contractible since for two points $x,y \in \Sigma^G$ also the
segment $[x,y]$ must lie in $\Sigma^G$ and hence the map $r$ above induces a continuous
map $\Sigma^G \times \Sigma^G \times[0,1] \to \Sigma^G$. This implies together with
Example~\ref{exa:simplicial_actions},
Theorem~\ref{the:G-homotopy_characterization_of_EGF(G)(calf)}
\eqref{the:G-homotopy_characterization_of_EGF(G)(calf):characterization},
Lemma~\ref{lem:totally_disconnected_groups_and_calcom} and
Theorem~\ref{the:comparison_EGFGcalf_and_JGFGcalf}

\begin{theorem}[\blue{Affine buildings}]
\label{the:affine_buildings}
Let $G$ be a topological (locally compact second countable Hausdorff) group.
Suppose that $G$ acts on the affine building by simplicial automorphisms
such that each isotropy group is compact. Then each isotropy group is compact open,
$\Sigma$ is a model
for $\JGF{G}{\Comop}$  and the barycentric subdivision $\Sigma^{\prime}$
is a model for both $\JGF{G}{\Comop}$ and  $\EGF{G}{\Comop}$.
If we additionally assume that $G$ is totally disconnected, then
$\Sigma$ is a model for both $\underline{J}G$ and $\underline{E}G$.
\end{theorem}

\begin{example}[\blue{Bruhat-Tits building}] \label{exa:Bruhat-Tits_building}
An important example is the case of a reductive $p$-adic algebraic group $G$ and 
its associated \red{affine Bruhat-Tits building $\beta(G)$}
(see~\cite{Tits(1974)}, \cite{Tits(1979)}).
Then $\beta(G)$ is a model for $\underline{J}G$ and
$\beta(G)^{\prime}$ is a model for $\underline{E}G$
by Theorem~\ref{the:affine_buildings}.
\end{example}

For more information about buildings we refer to the lectures of
\green{Abramenko}.

The \red{Rips complex} $P_d(G,S)$
of a group $G$ with a symmetric finite set $S$ of generators for a
natural number $d$ is the geometric realization of the simplicial set whose set of
$k$-simplices consists of $(k+1)$-tuples $(g_0,g_1, \ldots g_k)$ of
pairwise distinct elements $g_i \in G$ satisfying $d_S(g_i,g_j) \le d$
for all $i,j \in \{0,1,\ldots ,k\}$.

The obvious $G$-action by simplicial automorphisms on $P_d(G,S)$ induces a
$G$-action by simplicial automorphisms on the barycentric
subdivision $P_d(G,S)^{\prime}$.

\begin{theorem}[\blue{Rips complex}, \green{Meintrup-Schick (2002)}]
Let $G$ be a discrete group with a finite symmetric set of generators.
Suppose that $(G,S)$ is $\delta$-hyperbolic for the real number
$\delta \ge 0$. Let $d$ be a natural number with $d \ge 16\delta +8$.

Then the barycentric subdivision  of
the Rips complex $P_d(G,S)^{\prime}$ is a finite $G$-$CW$-model for
$\eub{G}$.
\end{theorem}
\begin{proof}
See~\cite{Meintrup-Schick(2002)}.
\end{proof}

\red{Arithmetic groups} in a semisimple connected linear $\IQ$-algebraic group possess
finite models for $\eub{G}$. Namely, let $G(\IR)$ be the $\IR$-points of a semisimple
$\IQ$-group $G(\IQ)$ and let $K\subseteq G(\IR)$ be a maximal compact subgroup.  If $A
\subseteq G(\IQ)$ is an arithmetic group, then $G(\IR)/K$ with the left $A$-action is a
model for $\eub{A}$ as already explained above.  However, the $A$-space $G(\IR)/K$ is not
necessarily cocompact.  But there is a finite model for $\eub{A}$ by the following result.

\begin{theorem} [\blue{Borel-Serre compactification}]
The \red{Borel-Serre compactification} 
(see~\cite{Borel-Serre(1973)}, \cite{Serre(1979)})  of  $G(\IR)/K$
is a finite $A$-$CW$-model for $\eub{A}$.
\end{theorem}
\begin{proof}
This is pointed out in \green{Adem-Ruan}~\cite[Remark~5.8]{Adem-Ruan(2003)},
where a private communication with \green{Borel} and \green{Prasad} is mentioned.
A detailed proof is given by \green{Ji}~\cite{Ji(2006torsion)}.
\end{proof}

For more information about arithmetic groups we refer to the lectures of
\green{Abramenko}.

Let \red{$\Gamma^s_{g,r}$} be the \red{mapping class group} of an
orientable compact surface $F$ of genus $g$ with $s$ punctures and $r$
boundary components.  We will always assume that $2g +s +r > 2$, or,
equivalently, that the Euler characteristic of the punctured surface
$F$ is negative.  It is well-known that the associated
\red{Teich\-m\"ul\-ler space $\calt^s_{g,r}$} is a contractible space
on which $\Gamma^s_{g,r}$ acts properly.

We could not find a clear reference in the literature for the to
experts known statement that there
exist a finite $\Gamma^s_{g,r}$-$CW$-model for 
 $\eub{\Gamma^s_{g,r}}$. The work of \green{Harer}~\cite{Harer(1986)}
on the existence of a spine and the construction of the spaces
$T_S(\epsilon)^H$ due to \green{Ivanov}~\cite[Theorem~5.4.A]{Ivanov(2002)} seem to lead
to such models. However, a detailed proof can be found in a manuscript
by \green{Mislin}~\cite{Mislin(2004)}.

\begin{theorem}[\blue{Teichm\"uller space}]
The  $\Gamma^s_{g,r}$-space $\calt^s_{g,r}$ is a model for $\eub{\Gamma^s_{g,r}}$.
\end{theorem}

Let $F_n$ be the free group of rank $n$. Denote by $\Out(F_n)$
the group of outer automorphisms of $F_n$, i.e., the quotient of the group 
of all automorphisms of $F_n$ by the normal subgroup of inner automorphisms. 
\green{Culler-Vogtmann} (see~\cite{Culler-Vogtmann(1986)}, 
\cite{Vogtmann(2003)}) have constructed a space $X_n$ called \emph{outer space}
on which $\Out(F_n)$ acts with finite isotropy groups. It is analogous to the 
Teich\-m\"ul\-ler space of a surface with the action of the mapping class group of 
the surface. Fix a graph $R_n$ with one vertex $v$ and $n$-edges and identify $F_n$
with $\pi_1(R_n,v)$. A \emph{marked metric graph}
$(g,\Gamma)$ consists of a graph $\Gamma$ with all vertices
of valence at least three, a homotopy equivalence $g \colon R_n \to \Gamma$ 
called marking and to every edge of $\Gamma$ there is assigned a positive length 
which makes $\Gamma$ into a metric space by the path metric. 
We call two marked metric graphs 
$(g,\Gamma)$ and $(g',\Gamma')$ equivalent if there is a homothety 
$h \colon \Gamma \to \Gamma'$ such that $g \circ h$ and $h'$ are homotopic.
Homothety means that there is a constant $\lambda > 0$ with 
$d(h(x),h(y)) = \lambda \cdot d(x,y)$ for all
$x,y$. Elements in outer space $X_n$ are equivalence classes of marked graphs.
The main result in~\cite{Culler-Vogtmann(1986)} is that $X$ is contractible.
Actually, for each  finite subgroup $H \subseteq \Out(F_n)$
the $H$-fixed point set $X_n^H$ is contractible
(see~\cite[Proposition~3.3 and Theorem~8.1]{Krstic-Vogtmann(1993)}, 
\cite[Theorem~5.1]{White(1993)}).

The space $X_n$ contains a \emph{spine}
$K_n$ which is an $\Out(F_n)$-equivariant deformation retraction.
This space $K_n$ is a simplicial complex of dimension $(2n-3)$
on which the $\Out(F_n)$-action is by simplicial automorphisms
and cocompact. Actually the group of simplicial automorphisms of $K_n$ is $\Out(F_n)$
(see \green{Bridson-Vogtmann}~\cite{Bridson-Vogtmann(2001)}). We conclude

\begin{theorem}[\blue{Spine of outer space}]
The barycentric subdivision $K_n^{\prime}$ is a finite
$(2n-3)$-dimensional model of $\underline{E}\Out(F_n)$.
\end{theorem}

\begin{example}[\blue{$SL_2(\IR)$ and $SL_2(\IZ)$}]
\label{exa:eub_for_SL_2(Z)}
In order to illustrate some of the general statements above we consider
the special example $SL_2(\IR)$ and $SL_2(\IZ)$.

Let $\IH^2$ be the $2$-dimensional hyperbolic space. We will use
either the upper half-plane model or the Poincar\'e disk model. The
group $SL_2(\IR)$ acts by isometric diffeomorphisms on the upper half-plane by
Moebius transformations, i.e., a matrix 
$\left(\begin{array}{cc} a & b \\ c & d \end{array}\right)$ 
acts by sending a complex number $z$ with positive imaginary part to
$\frac{az +b}{cz +d}$. This action  is proper and transitive. The isotropy group
of $z = i$ is $SO(2)$. Since $\IH^2$ is a
simply-connected Riemannian manifold, whose sectional curvature is
constant $-1$, the $SL_2(\IR)$-space $\IH^2$ is a model for
$\underline{E}SL_2(\IR)$ by Remark~\ref{rem:non-positive_curvature}.

One easily checks that $SL_2(\IR)$ is a connected Lie group 
and $SO(2) \subseteq SL_2(\IR)$ is a maximal compact subgroup. 
Hence $SL_2(\IR)/SO(2)$ is a model for $\underline{E}SL_2(\IR)$
by Theorem~\ref{the:Discrete_subgroups_of_almost_connected_Lie_groups}.
Since the $SL_2(\IR)$-action on $\IH^2$ is transitive and
$SO(2)$ is the isotropy group at $i \in \IH^2$, we see that
the $SL_2(\IR)$-manifolds $SL_2(\IR)/SO(2)$ and $\IH^2$ are
$SL_2(\IR)$-diffeomorphic.

Since $SL_2(\IZ)$ is a discrete subgroup of $SL_2(\IR)$, 
the space $\IH^2$ with the obvious $SL_2(\IZ)$-action is a model for 
$\underline{E}SL_2(\IZ)$ 
(see Theorem~\ref{the:Discrete_subgroups_of_almost_connected_Lie_groups}).

The group $SL_2(\IZ)$ is isomorphic to the amalgamated product 
$\IZ/4 \ast_{\IZ/2} \IZ/6$. This implies that
$SL_2(\IZ)$ acts on a tree $T$ which consists of two
$0$-dimensional equivariant cells with isotropy groups $\IZ/4$ and $\IZ/6$
and one $1$-dimensional equivariant cell with isotropy group $\IZ/2$.
From Remark~\ref{rem:non-positive_curvature}
we conclude that a model for $\underline{E}SL_2(\IZ)$ is given 
by the following $SL_2(\IZ)$-pushout
$$
\xymatrix@!C=15em{SL_2(\IZ)/(\IZ/2) \times \{-1,1\} \ar[d]
\ar[r]^-{F_{-1} \coprod F_1}
&
SL_2(\IZ)/(\IZ/4) \coprod SL_2(\IZ)/(\IZ/6) \ar[d]
\\
SL_2(\IZ)/(\IZ/2) \times [-1,1] \ar[r]
&
T = \underline{E}SL_2(\IZ)
}
$$
where $F_{-1}$ and $F_1$ are the obvious projections.
This model for $\underline{E}SL_2(\IZ)$ is a tree, which has
alternately two and three edges emanating from each vertex. The other model 
$\IH^2$ is a manifold. These two models must be $SL_2(\IZ)$-homotopy equivalent.
They can explicitly be related by the following construction.

Divide the Poincar\'e disk into fundamental domains for the $SL_2(\IZ)$-action.
Each fundamental domain is a geodesic triangle with one vertex at infinity, i.e.,
a vertex on the boundary sphere, and two vertices in the interior. Then the union of
the edges, whose end points lie in the interior of the Poincar\'e disk, is a tree
$T$ with $SL_2(\IZ)$-action. This is the tree model above. The tree is a
$SL_2(\IZ)$-equivariant deformation retraction of the Poincar\'e disk.
A retraction is given by moving a point $p$ in the Poincar\'e disk along a geodesic
starting at the vertex at infinity, which belongs to the triangle containing $p$,
through $p$ to the first intersection point of this geodesic with $T$.

The tree $T$ above can be identified with the Bruhat-Tits building
of $SL_2(\IQ\widehat{_p})$ and hence is a model for $\underline{E}SL_2(\IQ\widehat{_p})$
(see~\cite[page~134]{Brown(1998)}).
Since $SL_2(\IZ)$ is a discrete subgroup of $SL_2(\IQ\widehat{_p})$,
we get another reason why this tree is a model for $\eub{SL_2(\IZ)}$.
\end{example}

\begin{definition}[\blue{Cohomological dimension}]
Let $\Lambda$ be a commutative ring.
The \red{cohomological dimension} $\cd_{\Lambda}(G)$
of a group $G$ over $\Lambda$ is 
defined to be the infimum over all integers $d$
for which there exist a $d$-dimensional projective 
$\Lambda G$-resolution of the trivial $\Lambda G$-module $\Lambda$. 
If $\Lambda = \IZ$, we abbreviate $\cd(G) = \cd_{\IZ}(G)$.
\end{definition}

By definition  $\cd_{\Lambda}(G) = \infty$ if there is no finite-dimensional projective 
$\Lambda G$-resolution of the trivial $\Lambda G$-module $\Lambda$. 

\begin{example}
If $G$ is a non-trivial finite group, then $\cd(G) = \infty$ and
$\cd_{\IQ}(G) = 0$. We conclude that a  group $G$ with $\cd(G) < \infty$ 
must be torsionfree.
\end{example}

\begin{definition}[\blue{Virtual cohomological dimension}]
A group $G$ is called \red{virtually torsionfree} if it contains
a torsionfree subgroup $\Delta \subset G$ with finite index $[G:\Delta]$.

Let $\Lambda$ be a commutative ring.
Define the  \red{virtual cohomological dimension} 
of a virtually torsionfree group $G$ over $\Lambda$ by
$$\red{\vcd_{\Lambda}(G)} = \cd_{\Lambda}(\Delta)$$
for any torsionfree subgroup $\Delta \subset G$ with finite index $[G:\Delta]$.

If $\Lambda = \IZ$, we abbreviate $\vcd(G) = \vcd_{\IZ}(G)$.
\end{definition}

This definition is indeed independent of the choice of
$\Delta \subseteq G$. 

Next we investigate the relation between the minimal dimension of a model
$\eub{G}$ with the virtual cohomological dimension provided that
$G$ is virtually torsionfree.

\begin{theorem}[\blue{Discrete subgroups of Lie groups}]
Let $L$ be a Lie group with finitely many path components.
Let $K \subseteq L$ be a maximal compact subgroup $K$.
Let $G \subseteq L$ be a discrete subgroup of $L$.
Then $L/K$ with the left $G$-action is a model for $\underline{E}G$.

Suppose additionally that $G$ is \red{virtually torsionfree}, i.e.,
contains a torsionfree subgroup $\Delta \subseteq G$ of
finite index. 

Then we have for its \red{virtual cohomological dimension}
$$\vcd(G) \le \dim(L/K).$$

Equality holds if and only if $G\backslash L$ is compact.
\end{theorem}
\begin{proof} We have already mentioned in 
Theorem~\ref{the:Discrete_subgroups_of_almost_connected_Lie_groups}
that $L/K$ is a model for $\underline{E}G$. The restriction of
$\underline{E}G$ to $\Delta$ is a $\Delta$-$CW$-model for
$\underline{E}\Delta$ and hence $\Delta\backslash\underline{E}G$ is a
$CW$-model for $B\Delta$. This implies $\vcd(G) := \cd(\Delta) \le
\dim(L/K)$. Obviously $\Delta\backslash L/K$ is a manifold without
boundary. Suppose that $\Delta\backslash L/K$ is compact.
Then $\Delta\backslash L/K$ is a closed manifold and hence
its homology with $\IZ/2$-coefficients in the top dimension is
non-trivial. This implies $\cd(\Delta) \ge \dim(\Delta\backslash L/K)$
and hence $\vcd(G) = \dim(L/K)$. If $\Delta\backslash L/K$ is not
compact, it contains a $CW$-complex $X \subseteq \Delta\backslash L/K$
of dimension smaller than $\Delta\backslash L/K$ such that
the inclusion of $X$ into $\Delta\backslash L/K$ is a homotopy
equivalence. Hence $X$ is another model for $B\Delta$. This implies
$\cd(\Delta) < \dim(L/K)$ and hence $\vcd(G) < \dim(L/K)$. 
\end{proof}

\begin{theorem}[\blue{A criterion for $1$-dimensional models for $BG$},
  \green{Stallings (1968), Swan (1969)}]

Let $G$ be a discrete group. 

The following statements are equivalent:

\begin{itemize}

\item There exists a $1$-dimensional model for $EG$;

\item There exists a $1$-dimensional model for $BG$;

\item The cohomological dimension of $G$ is less or equal to one;

\item $G$ is a free group.
\end{itemize}
\end{theorem}
\begin{proof}
See~\cite{Stallings(1968)} and~\cite{Swan(1969)}.
\end{proof}

\begin{theorem}[\blue{A criterion for $1$-dimensional models for $\eub{G}$},
  \green{Dunwoody (1979)}]
Let $G$ be a discrete group. 
Then there exists a $1$-dimensional model
for $\underline{E}G$ if and only if the cohomological dimension of $G$ over
the rationals $\IQ$ is less or equal to one.
\end{theorem}
\begin{proof}
See \green{Dunwoody}~\cite[Theorem~1.1]{Dunwoody(1979)}.
\end{proof}

\begin{theorem} [\blue{Virtual cohomological dimension and
    $\dim(\underline{E}G)$}, \green{L\"uck (2000)}]
Let $G$ be a discrete group which is virtually torsionfree.
\begin{enumerate}
\item 
Then  
$$\vcd(G) \le \dim(\underline{E}G)$$
for any model for $\underline{E}G$. 
\item 
Let $l \ge 0$ be an  integer such that for any chain of finite subgroups
$H_0 \subsetneq H_1 \subsetneq \ldots \subsetneq H_r$ we have $r \le l$.

Then there exists a model for $\underline{E}G$ whose dimension is 
$$\max\{3,\vcd(G)\} + l.$$
\end{enumerate}
\end{theorem}
\begin{proof}
See \green{L\"uck}~\cite[Theorem~6.4]{Lueck(2000a)}. 
\end{proof}

The following problem has been stated by \green{Brown}~\cite[page~32]{Brown(1979)}
and has created a lot of activities.

\begin{problem} 
For which discrete groups $G$, which are virtually torsionfree,
does there exist
a $G$-$CW$-model for $\underline{E}G$ of dimension $\vcd(G)$?
\end{problem}

\begin{remark}
The results above give some evidence  for the hope
that the problem above has a positive answer for every discrete group.
However, \green{Leary-Nucinkis}~\cite{Leary-Nucinkis(2003)} have constructed 
virtually torsionfree groups $G$ for which 
the answer is negative, i.e., for which the dimension of any model for
$\underline{E}G$ is different from $\vcd(G)$.
\end{remark}

The following result shows that in general one can say nothing about
the quotient $G\backslash \eub{G}$ although in many interesting cases
there do exist small models for it.

\begin{theorem}[\green{Leary-Nucinkis (2001)}]
\label{the:Leary-Nucinkis(2001)}
Let $X$ be a $CW$-complex. Then there exists a group $G$ with
$X \simeq G\backslash \eub{G}$.
\end{theorem}
\begin{proof}
See~\cite{Leary-Nucinkis(2001a)}.
\end{proof}

\begin{question}[\blue{Homological Computations based on nice models for $\eub{G}$}]
  Can nice geometric models for $\eub{G}$ be used to compute the group homology and more
  general homology and cohomology theories of a group $G$?
\end{question}

\begin{question}[\blue{$K$-theory of group rings and group homology}]
  Is there a relation between $K_n(RG)$ and the group homology of $G$?
\end{question}

\begin{question}[\blue{Isomorphism  Conjectures and classifying spaces of families}]
  Can classifying spaces of families be used to formulate a version of the Farrell-Jones
  Conjecture and the Baum-Connes Conjecture which may hold for all group $G$ and all
  rings?
\end{question}

%%%%%%%%%%%%%%%%%%%%%%%%%%%%%%%%%%%%%%%%%%%%%%%%%%%%%%%%%%%%%%%%%%%%%%%%%%%%%%%%
%%%%%%%%%%%%%%%%%%%%%%%%%%%%%%%%%%%%%%%%%%%%%%%%%%%%%%%%%%%%%%%%%%%%%%%%%%%%%%%%
%%%%%%%%%%%%%%%%%%%%%%%%%%%%%%%%%%%%%%%%%%%%%%%%%%%%%%%%%%%%%%%%%%%%%%%%%%%%%%%%

\section{Equivariant homology theories}

The outline of this section is:

\begin{itemize}

  \item  We introduce the notion of an \red{equivariant homology theory}.
  
  \item We present the general  formulation of the \red{Farrell-Jones
          Conjecture} and the \red{Baum-Connes Conjecture}.
  
  \item We discuss \red{equivariant Chern characters}.
  
  \item We present some explicit \red{computations}
  of equivariant topological $K$-groups and of homology groups
  associated to classifying spaces of  groups.

\end{itemize}

\begin{definition}[\blue{$G$-homology theory}]\ \\
A \emphred{$G$-homology theory $\calh^G_*$} is a
covariant functor  from the category of
$G$-$CW$-pairs to the category of $\IZ$-graded
$\Lambda$-modules together with natural transformations
$$\partial_n^G(X,A)\colon \calh^G_n(X,A) \to
\calh^G_{n-1}(A)$$
for $n \in \IZ$
satisfying the following axioms:

\begin{itemize}
\item
$G$-homotopy invariance;

\item
Long exact sequence of a pair;

\item
Excision;

\item
Disjoint union axiom.

\end{itemize}
\end{definition}

The following definition is taken from~\cite[Section~1]{Lueck(2002b)}.

\begin{definition}[\blue{Equivariant homology theory}]
\label{def:equivariant_homology_theory}
An \emphred{equivariant homology theory $\calh^?_*$} assigns to every group $G$ a
$G$-homology theory $\calh^G_*$.
These are linked together with the following so called \emphred{induction structure}:
given a group homomorphism $\alpha\colon H \to G$ and  a $H$-$CW$-pair
$(X,A)$, there are for all $n \in \IZ$
natural homomorphisms
\begin{eqnarray*}
\ind_{\alpha}\colon  \calh_n^H(X,A)
&\to &
\calh_n^G(\ind_{\alpha}(X,A))
\end{eqnarray*}
satisfying

\begin{itemize}

\item Bijectivity\\
If $\ker(\alpha)$ acts freely on $X$, then $\ind_{\alpha}$ is a bijection;

\item Compatibility with the boundary homomorphisms;

\item Functoriality in $\alpha$;

\item Compatibility with conjugation.

\end{itemize}
\end{definition}

We have the following examples of equivariant homology theories.

\begin{example}[\blue{Borel homology}]
Given a non-equivariant homology theory $\calk_*$, put
\begin{eqnarray*}
\calh^G_*(X) & := & \calk_*(X/G);
\\
\calh^G_*(X) & := & \calk_*(EG \times_G X) \quad \red{\text{(Borel homology)}}.
\end{eqnarray*}
\end{example}

\begin{example}[\blue{Equivariant bordism}]
\red{Equivariant bordism} $\Omega^?_*(X)$ based on proper cocompact
equivariant smooth manifolds with reference map to the $G$-space $X$;
\end{example}

\begin{example}[\blue{Equivariant topological $K$-theory}]
\red{Equivariant topological $K$-theory} $K^?_*(X)$ defined for proper equivariant
$CW$-complexes has the property that for any finite subgroup $H \subseteq G$
we get
$$K^H_n(\pt) \cong K_0^G(G/H) \cong \left\{
\begin{array}{lcl} R_{\IC}(H) & & n \text{ even};
\\
0 & & n \text{ odd}.
\end{array}
\right.
$$
\end{example}

\begin{theorem}[\green{L\"uck-Reich (2005)}]
Given a functor $\bfE \colon \Groupoids \to \Spectra$ sending
equivalences to weak equivalences, 
there exists an equivariant homology theory \red{$\calh^?_*(-;\bfE)$} satisfying
$$\calh_n^H(\pt) \cong \calh_n^G(G/H) \cong \pi_n(\bfE(H)).$$
\end{theorem}
\begin{proof}
See~\cite[Proposition~6.4 on page~738]{Lueck-Reich(2005)}.
\end{proof}

\begin{theorem}[\blue{Equivariant homology theories associated to $K$ and $L$-theory},
\green{Davis-L\"uck (1998)}]

Let $R$ be a ring (with involution). There exist covariant functors

\begin{eqnarray*}
\red{\bfK_R}
\colon \Groupoids & \to & \Spectra;
\\
\red{\bfL^{\langle \infty \rangle}_R} \colon \Groupoids& \to & \Spectra;
\\
\red{\bfK^{\topo}} \colon \Groupoids^{\inj}  & \to & \Spectra
\end{eqnarray*}
with the following properties:

\begin{itemize}

\item
They send equivalences of groupoids to  weak equivalences of spectra;

\item
For every group $G$ and all $n \in \IZ$ we have

\begin{eqnarray*}
\pi_n(\bfK_R(G)) & \cong & K_n(RG);
\\
\pi_n(\bfL^{\langle -\infty  \rangle}_R(G)) & \cong & L_n^{\langle -\infty \rangle}(RG);
\\
\pi_n(\bfK^{\topo}(G)) & \cong & K_n(C^*_r(G)).
\end{eqnarray*}

\end{itemize}
\end{theorem}
\begin{proof}
See~\cite[Section~2]{Davis-Lueck(1998)}.
\end{proof}

Combining the last two theorems we get

\begin{example}[\blue{Equivariant homology theories 
associated to $K$ and $L$-theory}]
\label{exa:equivariant_homology_theories_associated_toK-and_L-theory}
We get equivariant homology theories
\begin{eqnarray*}
& \red{H_*^?(-;\bfK_R)}; &
\\
& \red{H_*^?(-;\bfL^{\langle -\infty \rangle}_R)}; &
\\
& \red{H_*^?(-;\bfK^{\topo})}, &
\end{eqnarray*}
satisfying for $H\subseteq G$
$$
\begin{array}{lclcl}
H_n^G(G/H;\bfK_R) & \cong & H_n^H(\pt;\bfK_R) & \cong & K_n(RH);
\\
H_n^G(G/H;\bfL^{\langle -\infty \rangle}_R) & \cong & H_n^H(\pt;\bfL^{\langle -\infty \rangle}_R)
& \cong & L^{\langle -\infty \rangle}_n(RH);
\\
H_n^G(G/H;\bfK^{\topo}) & \cong & H_n^H(\pt;\bfK^{\topo}) & \cong & K_n(C_r^*(H)).
\end{array}
$$
\end{example}

Now we are ready to give the general formulation of the Farrell-Jones and the 
Baum-Connes Conjecture.

\begin{conjecture}[\blue{$K$-theoretic Farrell-Jones-Conjecture}]
\label{con:FJC_for_K}
The \emphred{$K$-theoretic Farrell-Jones Conjecture}
with coefficients in $R$ for the group $G$ predicts that
the \red{assembly map}
$$H_n^G(\EGF{G}{\VCyc},\bfK_R) \to H_n^G(\pt,\bfK_R) = K_n(RG),$$
which is the map induced by the projection
$\EGF{G}{\VCyc} \to \pt$, is bijective for all $n \in \IZ$.
\end{conjecture}

\begin{conjecture}[\blue{$L$-theoretic Farrell-Jones-Conjecture}]
\label{con:FJC_for_L}
The \emphred{$L$-theoretic Farrell-Jones Conjecture}
with coefficients in $R$ for the group $G$ predicts that
the \red{assembly map}
$$H_n^G(\EGF{G}{\VCyc},\bfL_R^{\langle -\infty\rangle}) \to 
H_n^G(\pt,\bfL_R^{\langle-\infty\rangle}) = L_n^{\langle-\infty\rangle}(RG),$$
which is the map induced by the projection
$\EGF{G}{\VCyc} \to \pt$,
is bijective for all $n \in \IZ$.
\end{conjecture}

\begin{conjecture}[\blue{Baum-Connes Conjecture}]
\label{con:BCC}
The \emphred{Baum-Connes Conjecture} predicts that
the assembly map
$$K_n^G(\eub{G}) = H_n^G(\EGF{G}{\Fin},\bfK^{\topo}) \to H_n^G(\pt,\bfK^{\topo}) 
= K_n(C^*_r(G))$$
which is the map induced by the projection
$\EGF{G}{\Fin} \to \pt$, is bijective for all $n \in \IZ$.
\end{conjecture}

\begin{remark}[\blue{Original sources for the Farrell-Jones and the Baum-Connes Conjecture}]
These conjectures were stated in
\green{Farrell-Jones}~\cite[1.6 on page~257]{Farrell-Jones(1993a)} and
\green{Baum-Connes-Higson}~\cite[Conjecture~3.15 on page~254]{Baum-Connes-Higson(1994)}.  
Our formulations differ from the original ones, but are equivalent
(see~\cite[Section~6]{Bartels-Farrell-Jones-Reich(2004)},
\cite[Section~6]{Davis-Lueck(1998)}, and~\cite{Hambleton-Pedersen(2004)}).
In the case of the Farrell-Jones Conjecture
we slightly generalize the original conjecture 
by allowing arbitrary coefficient rings instead of $\IZ$. 
\end{remark}

We will discuss these conjectures and their applications in the next section.
We will now continue with equivariant homology theories.

Let $\calh_*$ be a (non-equivariant) homology theory.
There is the \red{Atiyah-Hirzebruch spectral sequence} which converges to
$\calh_{p+q}(X)$ and has as $E^2$-term
$$E_{p,q}^2 = H_p(X;\calh_q(\pt)).$$
Rationally it collapses completely by the following result.

\begin{theorem}[\blue{Non-equivariant Chern character}, \green{Dold (1962)}]

  Let $\calh_*$ be a homology theory with values in $\Lambda$-modules for $\IQ \subseteq
  \Lambda$.

  Then there exists for every $n \in \IZ$ and every $CW$-complex $X$ a natural isomorphism
$$\bigoplus_{p+q=n} H_p(X;\Lambda)\otimes_{\Lambda} \calh_q(\pt) 
\xrightarrow{\cong} \calh_n(X),$$
where $H_p(X;\Lambda)$ is the singular or cellular homology of $X$ with coefficients in
$\Lambda$.
\end{theorem}
\begin{proof} 
At  least we give the definition of \green{Dold's} \red{Chern character} 
for a $CW$-complex $X$, for more details we refer to
\green{Dold}~\cite{Dold(1962)}. It is given by the following composite:
$$
\ch_n \colon  \bigoplus_{p+q = n} H_p(X;\calh_q(*))
\xrightarrow{\red{\alpha^{-1}}}
\bigoplus_{p+q = n} H_p(X;\IZ) \otimes_{\IZ} \calh_q(*)
$$

$$
\xrightarrow{\bigoplus_{p+q = n} (\red{\hur} \otimes \id)^{-1}}
\bigoplus_{p+q = n} \pi^s_p(X_+,*) \otimes_{\IZ}  \calh_q(*)
\xrightarrow{\bigoplus_{p+q = n} \red{D_{p,q}}} \calh_n(X),
$$
Here the canonical map $\red{\alpha}$
is bijective, since any $\Lambda$-module 
is flat over $\IZ$ because of the assumption $\IQ \subset \Lambda$,
The map $\red{\hur}$ is the \red{Hurewicz homomorphism}
which is bijective because of Serre's Theorem (see~\cite{Serre(1953)},
\cite{Klaus-Kreck(2004)}) which says
$$
\pi_m^s \otimes \mathbb{Q} \cong \left\{
\begin{array}{ll}
  \mathbb{Q} &  \text{for }    m = 0, \\
  0  & \text{else.}
\end{array}
\right.
$$
The map
$D_{p,q}$ sends $[f \colon (S^{p+k},\pt) \to (S^k \wedge X_+,\pt)] \otimes \eta$ 
to the image of $\eta$ under the composite
\begin{multline*}
\calh_q(*)
\cong \calh_{p+k+q}(S^{p+k},\pt)
\xrightarrow{\calh_{p+k+q}(f)}
\calh_{p+k+q}(S^k \wedge X_+,\pt)
\cong \calh_{p+q}(X).
\end{multline*}
\end{proof}

We want to extend this to the equivariant setting.
This requires an extra structure on the coefficients of an equivariant 
homology theory $\calh^?_*$.

We define a covariant functor called \red{induction}
$$\red{\ind} \colon \FGI \to \Lambda\text{-}\operatorname{Mod}$$
from the category $\red{\FGI}$ of finite groups
with injective group homomorphisms
as morphisms to the category of $\Lambda$-modules as follows.
It sends $G$ to $\calh_n^G(\pt)$
and an injection of finite groups $\alpha \colon H \to G$
to the morphism given by the induction structure

$$\calh_n^H(\pt)
\xrightarrow{\ind_{\alpha}} \calh_n^G(\ind_{\alpha} \pt)
\xrightarrow{\calh_n^G(\pr)} \calh_n^G(\pt).$$

\begin{definition}[\blue{Mackey extension}]
We say that $\calh^?_*$ has a \red{Mackey extension} if for every $n \in \IZ$
there is a contravariant functor
called \red{restriction}
$$\red{\res} \colon \FGI \to \Lambda\text{-}\operatorname{Mod}$$
such that the two functors $\ind$ and $\res$ agree on
objects and satisfy the \red{double coset formula},
i.e., we have for two subgroups $H,K \subset G$ of the finite group $G$
$$\res_G^K \circ \ind_H^G = \sum_{KgH \in K\backslash G/H}
\ind_{c(g): H\cap g^{-1}Kg \to K}
 \circ \res_{H}^{ H\cap g^{-1}Kg},$$
where $c(g)$ is conjugation with $g$, i.e., $c(g)(h) = ghg^{-1}$.
\end{definition}

\begin{remark}[\blue{Existence of Mackey extensions}]
In every case we will consider such a Mackey extension does exist and is
given by an actual restriction.
For instance for
$H_0^?(-;\bfK^{\topo})$ induction is the functor complex representation ring $R_{\IC}$
with respect to induction of representations. 
The restriction part is given by the restriction of representations.
\end{remark}

We need some notation.
Consider a subgroup $H \subseteq G$. Denote by 
\red{$C_GH$} the \red{centralizer} and by \red{$N_GH$} the \red{normalizer} 
of $H \subseteq G$. Put 
$$\red{W_GH} := N_GH/H \cdot C_GH.$$
This is always a finite group. Define for an equivariant homology theory
$\calh^?_*$
$$S_H\left(\calh^H_q(*)\right) := 
\cok\left(\bigoplus_{\substack{K \subset H\\K \not= H}} \ind_K^H:
\bigoplus_{\substack{K \subset H\\K \not= H}} \calh^K_q(*) \to \calh^H_q(*)\right).$$

\begin{theorem}[\blue{Equivariant Chern character}, \green{L\"uck (2002)}] 
\label{the:equivariant_Chern_character}
Let $\calh^?_*$ be an equivariant
homology theory with values in $\Lambda$-modules for $\IQ \subseteq \Lambda$.
Suppose that $\calh^?_*$ has a \red{Mackey extension}.
Let $I$ be the set of conjugacy classes $(H)$ of finite subgroups $H$ of $G$.

Then there is for every group $G$, every proper $G$-$CW$-complex $X$ and
every $n \in \IZ$ a natural isomorphism
called \red{equivariant Chern character}
$$\red{\ch^G_n}\colon \bigoplus_{p+q = n} \bigoplus_{(H) \in I}
H_p(C_GH\backslash X^H;\Lambda) \otimes_{\Lambda[W_GH]} S_H\left(\calh^H_q(*)\right)
\xrightarrow{\cong} \calh^G_n(X).$$
Actually $\ch^?_*$  is an \red{equivalence of equivariant homology theories}.
\end{theorem}
\begin{proof}
See~\cite[Theorem~0.2]{Lueck(2002b)}
\end{proof}

Recall the following basic result from
complex representation theory of finite groups.

\begin{theorem}[\blue{Artin's Theorem}]
Let $G$ be finite.  Then the map
$$\bigoplus_{C \subset G} \ind_C^G:
\bigoplus_{C \subset G} R_{\IC}(C) \to R_{\IC}(G)$$
is surjective after inverting $|G|$, where $C \subset G$ runs through the
cyclic subgroups of $G$.
\end{theorem}
\begin{proof}
See for instance~\cite[Theorem~17 in~9.2 on page~70]{Serre(1977)}.
\end{proof}

Let $C$ be a finite cyclic group.  The \red{Artin defect} is the cokernel of
the map
$$\bigoplus_{D \subset C, D \not= C}
\ind_D^C: \bigoplus_{D \subset C, D \not= C} R_{\IC}(D) \to R_{\IC}(C).$$

For an appropriate idempotent 
$\theta_C \in R_{\IQ}(C) \otimes_{\IZ} \IZ\left[\frac{1}{|C|}\right]$
the Artin defect is after inverting the order of $|C|$ canonically
isomorphic to
$$\theta_C \cdot R_{\IC}(C)  \otimes_{\IZ} \IZ\left[\frac{1}{|C|}\right]$$
by~\cite[Lemma~7.4]{Lueck(2002b)}.

\begin{example}[\blue{An improvement of Artin's Theorem}]
Let $K^G_* = H_*^?(-;\bfK^{\topo})$ be equivariant topological $K$-theory.
We get for a finite subgroup $H \subseteq G$
$$K_n^G(G/H) = K^H_n(\pt) =
\left\{\begin{array}{lll}
 R_{\IC}(H) & & \text{if } n \text{ is even};
\\
\{0\} && \text{if } n \text{ is odd}.
\end{array}\right.
$$
Hence $S_H\left(K^H_q(*)\right)\otimes_{\IZ} \IQ  = 0$, if $H$ is not cyclic and 
$q$ is even or if $q$ is odd, and we have
$S_C\left(K^C_q(*)\right) \otimes_{\IZ} \IQ= \theta_C \cdot
R_{\IC}(C) \otimes_{\IZ} \IQ$, 
if $C$ is finite cyclic and $q$ is even.

Let $G$ be finite, $X = \{\ast\}$ and $\calh^?_* = K^?_*$.
In this very special case Theorem~\ref{the:equivariant_Chern_character}
yields already something new, namely, an improvement of Artin's theorem, i.e., 
the equivariant Chern character induces an isomorphism 
\begin{multline*}
\ch^G_0(\pt) \colon \bigoplus_{(C)} \IZ  \otimes_{\IZ[W_GC]} \theta_C \cdot
R_{\IC}(C) \otimes_{\IZ} \IZ\left[\frac{1}{|G|}\right]
\xrightarrow{\cong}
R_{\IC}(G)  \otimes_{\IZ} \IZ\left[\frac{1}{|G|}\right]
\end{multline*}
where $(C)$ runs over the conjugacy classes of finite cyclic subgroups.
(Theorem~\ref{the:equivariant_Chern_character} yields only a statement
after applying $- \otimes_{\IZ} \IQ$ but the statement above, where we 
only invert the order of the group $G$ is proved in~\cite[Theorem~0.7]{Lueck(2002d)}).
\end{example}

\begin{theorem}[\blue{Rational computation of $K_*^G(\eub{G})$}]
\label{the:Rational_computation_of_K_astG(eubG)}
For every group $G$ and every $n \in \IZ$ we obtain an isomorphism
\begin{multline*}
\bigoplus_{(C)} \bigoplus_{k} H_{p+2k}(BC_GC) \otimes_{\IZ[W_GC]} \theta_C \cdot
R_{\IC}(C) \otimes_{\IZ} \IQ
\xrightarrow{\cong}
K_n^G(\eub{G}) \otimes_{\IZ} \IQ.
\end{multline*}
\end{theorem}
\begin{proof}
This follows from 
Theorem~\ref{the:equivariant_Chern_character} applied to the case
$X = \eub{G}$ and $\calh^?_* = K^?_*$ using the following facts.
\begin{itemize}
\item $\eub{G}^C$ is a contractible proper $C_GC$- space.  Hence the
      canonical map $BC_GC \to C_GC\backslash \eub{G}^C$ induces an
      isomorphism
      $$H_p(BC_GC) \otimes_{\IZ} \IQ \xrightarrow{\cong}
      H_p(C_GC\backslash \eub{G}^C) \otimes_{\IZ} \IQ.$$

\item $S_H\left(K^H_q(*)\right)\otimes_{\IZ} \IQ  = 0$ if $H$ is not cyclic and $q$ is even
or if $q$ is odd.

\item $S_C\left(K^C_q(*)\right) \otimes_{\IZ} \IQ= \theta_C \cdot
R_{\IC}(C) \otimes_{\IZ} \IQ $ if $C$ is finite cyclic and $q$ is even.
\end{itemize}
\end{proof}

\begin{remark}[\blue{Rational computation of $K_*(C^*_r(G))$}]
If the Baum-Connes Conjecture holds for $G$, 
Theorem~\ref{the:Rational_computation_of_K_astG(eubG)} yields an isomorphism
\begin{multline*}
\bigoplus_{(C)} \bigoplus_{k} H_{p+2k}(BC_GC) \otimes_{\IZ[W_GC]} \theta_C \cdot
R_{\IC}(C) \otimes_{\IZ} \IQ
\xrightarrow{\cong}
K_n(C^*_r(G)) \otimes_{\IZ} \IQ.
\end{multline*}
\end{remark}

Next we introduce some notation.
For a prime $p$ denote by $\red{r(p)} = |\con_p(G)|$ the number of conjugacy 
classes $(g)$ of elements $g \not= 1$ in $G$ of $p$-power order.
Let \red{$\II_G$} is the augmentation ideal of $R_{\IC}(G)$.
Denote by \red{$\II_p(G)$} the image of the restriction 
homomorphism $\II(G) \to \II(G_p)$ for the inclusion of the $p$-Sylow subgroup
$G_p \to G$.

\begin{theorem}[\blue{Completion Theorem}, \green{Atiyah-Segal (1969)}]
Let $G$ be a finite group.
Then there are isomorphisms of abelian groups
\begin{eqnarray*}
K^0(BG) & \cong & R_{\IC}(G)\widehat{_{\II_G}}
\\ & & \quad  \cong  \IZ \times \prod_{p\text{ prime}}  \II_p(G)
\otimes_{\IZ} \IZ\widehat{_p}
\cong 
\IZ \times \prod_{p\text{ prime}} (\IZ\widehat{_p})^{r(p)};
\\
K^1(BG) & \cong & 0.
\end{eqnarray*}
\end{theorem}
\begin{proof}
See~\cite{Atiyah-Segal(1969)} and for the explicit formula
for instance~\cite[page~125]{Jackowski-Oliver(1996)} 
or~\cite[Theorem~3.5]{Lueck(2005e)}.
\end{proof}

\begin{theorem}[\green{L\"uck (2005)}]
\label{the:rational_computation_of_Kast(BG)}
Let $G$ be a discrete group.
Denote by $K^*(BG)$
the topological (complex) K-theory of its classifying space $BG$.
Suppose that there is a
cocompact  $G$-$CW$-model for the classifying space $\eub{G}$ for
proper $G$-actions.

Then there is a $\IQ$-isomorphism
\begin{multline*}
\overline{\ch}^n_G \colon K^n(BG) \otimes_{\IZ} \IQ  \xrightarrow{\cong}
\\
\left(\prod_{i \in \IZ} H^{2i+n}(BG;\IQ)\right) \times
\prod_{p \text{ prime}}  \prod_{(g) \in \con_p(G)}
\left(\prod_{i \in \IZ} H^{2i+n}(BC_G\langle g \rangle;\IQ\widehat{_p})\right).
\end{multline*}

\end{theorem}
\begin{proof}
See~\cite{Lueck(2005e)}.
\end{proof}

\begin{remark}[\blue{Multiplicative structure}]
The multiplicative structure is also determined in~\cite{Lueck(2005e)}.
\end{remark}

\begin{remark}[\blue{Finiteness condition about $\eub{G}$}]
We have presented in the previous section many groups for which  a
cocompact  $G$-$CW$-model for $\eub{G}$ exists, e.g., hyperbolic groups.
Notice that this condition appears in 
Theorem~\ref{the:rational_computation_of_Kast(BG)} although the conclusion
in Theorem~\ref{the:rational_computation_of_Kast(BG)} is about $BG$ and not
about $\eub{G}$ or $G\backslash\eub{G}$.
\end{remark}

\begin{example}[\blue{$SL_3(\IZ)$}]
It is well-known that its rational cohomology satisfies
$\widetilde{H}^n(BSL_3(\IZ);\IQ) = 0$ for all $n \in \IZ$.
Actually, by a result of \green{Soul\'e}~\cite[Corollary on page~8]{Soule(1978)}
the quotient space $SL_3(\IZ)\backslash\eub{SL_3(\IZ)}$ is contractible
and compact. From the classification of finite subgroups of $SL_3(\IZ)$ we
see that $SL_3(\IZ)$ contains up to conjugacy  two elements of order $2$,
two elements of order $4$ and two elements of order $3$ and no 
further conjugacy classes of non-trivial elements of prime power order.
The rational homology of each of the
centralizers of elements in $\con_2(G)$ and $\con_3(G)$ agrees 
with the one of the trivial group. Hence we get
\begin{eqnarray*}
K^0(BSL_3(\IZ)) \otimes_{\IZ} \IQ & \cong &
\IQ \times (\IQ\widehat{_2})^4 \times (\IQ\widehat{_3})^2;
\\
K^1(BSL_3(\IZ)) \otimes_{\IZ} \IQ & \cong &  0.
\end{eqnarray*}
The identification of $K^0(BSL_3(\IZ)) \otimes_{\IZ} \IQ$ above
is compatible with the multiplicative structures.

Actually the computation using \red{Brown-Petersen cohomology}
and the \red{Conner-Floyd relation} by
\green{Tezuka-Yagita}~\cite{Tezuka-Yagita(1992)} 
gives the integral computation
\begin{eqnarray*}
K^0(BSL_3(\IZ)) & \cong &
\IZ \times (\IZ\widehat{_2})^4 \times (\IZ\widehat{_3})^2;
\\
K^1(BSL_3(\IZ)) & \cong &  0.
\end{eqnarray*}
\green{Soul\'e}~\cite{Soule(1978)} 
has computed the integral cohomology of $SL_3(\IZ)$.
\end{example}

Let $G$ be a discrete group.
Let \red{$\MFin$} be the subset  of $\Fin$ consisting
of elements in $\Fin$ which are maximal in $\Fin$.
Consider the following conditions about $G$:

\begin{itemize}

\item[\red{(M)}]
Every non-trivial finite subgroup of $G$ is contained in a unique maximal 
finite subgroup;

\item[\red{(NM)}] If $M \in \MFin, M \not= \{1\}$, then $N_GM = M$.
\end{itemize}

\begin{example}[\blue{Groups satisfying (M) and (NM)}]
\label{exa:Groups_satisfying_(M)_and_(NM)}
By \green{Davis-L\"uck}~\cite[page~101-102]{Davis-Lueck(2003)} the following 
groups satisfy conditions (M) and (NM):

\begin{itemize}

\item Extensions $1 \to \IZ^n \to G \to F \to 1$ for finite $F$ such that the conjugation
action of $F$ on $\IZ^n$ is free outside $0 \in \IZ^n$;

\item Fuchsian groups;

\item One-relator groups $G$.

\end{itemize}
\end{example}

For such a group there is a  nice model for $\underline{E}G$ 
with as few non-free cells as possible.
Let 
$$\red{\{(M_i) \mid i \in I\}}$$
be the set of conjugacy classes of maximal
finite subgroups of $M_i \subseteq G$.
By attaching free $G$-cells we get
an inclusion of $G$-$CW$-complexes
$j_1 \colon \coprod_{i \in I} G \times_{M_i} EM_i  \to EG$.
Define the $G$-$CW$-complex $X$ as the $G$-pushout
\begin{eqnarray}
& \xymatrix{
\coprod_{i \in I} G \times_{M_i} EM_i
\ar[r]^-{j_1} \ar[d]^-{u_1}
&
EG
\ar[d]^-{f_1}
\\
\coprod_{i \in I} G/M_i
\ar[r]^-{k_1}
& X
}
&
\label{pushout_for_EubG_for_(M)_and_(NM)}
\end{eqnarray}
where $u_1$ is the obvious $G$-map obtained by collapsing each $EM_i$ to a point.
\begin{theorem}[\blue{Model for $\eub{G}$ for groups satisfying (M) and (NM)}] 
\label{the:Model_for_eub(G)_for_(M)_and_(NM)}
Let $G$ be a group satisfying conditions (M) and (NM).
Then the $G$-$CW$-complex $X$ defined by the 
$G$-pushout~\eqref{pushout_for_EubG_for_(M)_and_(NM)} 
is a model for $\eub{G}$.
\end{theorem}
\begin{proof}
The isotropy groups of $X$ are all finite.
We have to show for $H \subseteq G$ finite
that $X^H$ contractible.
We begin with the case $H \not= \{1\}$.
Because of conditions (M) and (NM)
there is precisely one index $i_0 \in I$ such that $H$ is subconjugated to
$M_{i_0}$ and is not subconjugated to $M_i$ for $i \not= i_0$.
We get
$$\left(\coprod_{i \in I} G/M_i\right)^H   =  \left(G/M_{i_0}\right)^H  =  \pt.$$
Hence $X^H = \pt$.
It remains to treat $H = \{1\}$.
Since $u_1$ is a non-equivariant homotopy equivalence
and $j_1$ is a cofibration, $f_1$ is a non-equivariant homotopy equivalence.
Hence $X$ is contractible.
\end{proof}

\begin{example}[\blue{The homology of groups satisfying (M) and (NM)}]
Let $G$ be a group satisfying conditions (M) and (NM).
Because of Theorem~\ref{the:Model_for_eub(G)_for_(M)_and_(NM)}
we obtain the following pushout by taking the $G$-quotient
of the $G$-pushout~\eqref{pushout_for_EubG_for_(M)_and_(NM)} 
$$
\xymatrix{
\coprod_{i \in I} BM_i
\ar[r] \ar[d]
&
BG
\ar[d]
\\
\coprod_{i \in I} \pt
\ar[r]
& G\backslash \eub{G}
}
$$
The associated long exact Mayer-Vietoris sequence yields
\begin{multline*} \cdots \to \widetilde{H}_{n+1}(G\backslash \underline{E}G) \to
\bigoplus_{i \in I} \widetilde{H}_n(BM_i) \to
\widetilde{H}_n(BG)
\to \widetilde{H}_n(G\backslash\underline{E}G) \to \cdots.
\end{multline*}
In particular we obtain an isomorphism for $n \ge \dim(\underline{E}G) + 2$
$$
\bigoplus_{i \in I} H_n(BM_i) \xrightarrow{\cong} H_n(BG).
$$
So we get an explicit computation of $H_n(BG)$ for large $n$ and it 
is obvious why it is useful to have models for $\eub{G}$ of 
as small as possible dimension. Computations for low values of $n$ can 
sometimes be carried out by spectral sequence arguments or specific arguments.
\end{example}

The following identifications follow
from the definition of the 
\emphred{Whitehead groups} \red{$\Wh_n(G)$} for $n \ge 0$
due to \green{Waldhausen}~\cite[Definition~15.6 on page~228 and
Proposition~15.7 on page~229]{Waldhausen(1978a)}
which also makes sense for all $n \in \IZ$ 
if we use the non-connective $K$-theory spectrum
\begin{eqnarray*}
\begin{array}{lcl}
\Wh(G) = \Wh_1(G); & & 
\\
\widetilde{K}_0(\IZ G) = \Wh_0(G); & &
\\
K_n(\IZ G) = \Wh_n(G) & & \text{for } n \le -1.
\end{array}
\end{eqnarray*}
For a finite group $H$ define  
$\red{\widetilde{R}_{\IC}(H)}$ as the kernel of the ring homomorphism
$R_{\IC}(H) \to \IZ$ sending $[V]$ to $\dim_{\IC}(V)$.

\begin{theorem}[\green{Davis-L\"uck (2003)}]
Let $G$ be a discrete group which satisfies the conditions (M) and (NM) above.
\begin{enumerate}
\item Then there is an isomorphism
$$K_1^G(\eub{G}) \xrightarrow{\cong} K_1(G\backslash\eub{G}),$$
and a short exact sequence
$$
0 \to \bigoplus_{i \in I} \widetilde{R}_{\IC}(M_i) \to
K_0(\eub{G}) \to K_0(G\backslash\eub{G}) \to 0.
$$

\item The short exact sequence above splits 
if we invert the orders of all finite subgroups of $G$;

\item Suppose that $G$ belongs to $\calbc$.
(This is the case for the groups appearing in 
Example~\ref{exa:Groups_satisfying_(M)_and_(NM)}). Then
$$K_n(C^*_r(G)) \cong K_n^G(\eub{G});$$

\item Suppose that $G$ belongs to $\calfj_K(\IZ)$. Then  there is for
$n \in \IZ$ an isomorphism of Whitehead groups
$$\bigoplus_{i \in I} \Wh_n(M_i) \xrightarrow{\cong} \Wh_n(G),$$
where $\Wh_n(M_i) \to \Wh_n(G)$ is induced by the inclusion $M_i \to G$.

\end{enumerate}
\end{theorem}
\begin{proof}
See~\cite[Theorem~5.1]{Davis-Lueck(2003)}.
\end{proof}

\begin{remark}[\blue{Small models for $\eub{G}$ and computations}]
We see that for computations of group homology or of $K$- and
$L$-groups of group rings and group $C^*$-algebras it is important
to understand the spaces $G\backslash \eub{G}$.
Often geometric input ensures that $G\backslash
\eub{G}$ is a fairly small $CW$-complex.

On the other hand recall from Theorem~\ref{the:Leary-Nucinkis(2001)} 
that for any $CW$-complex $X$ there exists a group $G$ with
$X \simeq G\backslash \eub{G}$.
\end{remark}

\begin{question}[\blue{Consequences}]
What are the consequences of the Farrell-Jones Conjecture 
and the Baum-Connes Conjecture?
\end{question}

%%%%%%%%%%%%%%%%%%%%%%%%%%%%%%%%%%%%%%%%%%%%%%%%%%%%%%%%%%%%%%%%%%%%%%%%%%%%%%%%
%%%%%%%%%%%%%%%%%%%%%%%%%%%%%%%%%%%%%%%%%%%%%%%%%%%%%%%%%%%%%%%%%%%%%%%%%%%%%%%%
%%%%%%%%%%%%%%%%%%%%%%%%%%%%%%%%%%%%%%%%%%%%%%%%%%%%%%%%%%%%%%%%%%%%%%%%%%%%%%%%

\section{The Isomorphism Conjectures for arbitrary groups}

The outline of this section is:

  \begin{itemize}

  \item We discuss the difference between the families $\Fin$ and $\VCyc$.

  \item We discuss consequences of the Farrell-Jones  and the Baum-Connes Conjecture.

  \end{itemize}

Throughout this section $G$ will always be a discrete group.
We have introduced the following notions and conjectures:

\begin{itemize}

\item \red{Families of subgroups}, e.g., 
the families \red{$\Fin$} and \red{$\VCyc$}
of finite and virtually cyclic subgroups
(see Definition~\ref{def:family_of_subgroups});

\item \red{Classifying $G$-$CW$-complex $\EGF{G}{\calf}$ for a family 
of subgroups $\calf$} (see Definition~\ref{def:EGF(G)(F)}).

\item \red{Equivariant homology theories $\calh^?_*$} (see
Definition~\ref{def:equivariant_homology_theory});

\item Specific examples of \red{equivariant homology theories 
associated to $K$- and $L$-theory}
(see Example~\ref{exa:equivariant_homology_theories_associated_toK-and_L-theory})
\begin{eqnarray*}
& \red{H_*^?(-;\bfK_R)}; &
\\
& \red{H_*^?(-;\bfL^{\langle -\infty \rangle}_R)}; &
\\
& \red{H_*^?(-;\bfK^{\topo})}, &
\end{eqnarray*}
satisfying for $H\subseteq G$
$$
\begin{array}{lclcl}
H_n^G(G/H;\bfK_R) & \cong & H_n^H(\pt;\bfK_R) & \cong & K_n(RH);
\\
H_n^G(G/H;\bfL^{\langle -\infty \rangle}_R) & \cong 
& H_n^H(\pt;\bfL^{\langle -\infty \rangle}_R)
& \cong & L^{\langle -\infty \rangle}_n(RH);
\\
H_n^G(G/H;\bfK^{\topo}) & \cong & H_n^H(\pt;\bfK^{\topo}) & \cong & K_n(C_r^*(H));
\end{array}
$$

\item \red{The Farrell-Jones Conjecture for algebraic $K$-theory}
which predicts the bijectivity of the \red{assembly map} induced by 
the projection $\EGF{G}{\VCyc} \to \pt$
$$H_n^G(\EGF{G}{\VCyc},\bfK_R) \to H_n^G(\pt,\bfK_R) = K_n(RG)$$
for all $n \in \IZ$ (see Conjecture~\ref{con:FJC_for_K});

\item \red{The Farrell-Jones Conjecture for algebraic $L$-theory}
which predicts the bijectivity of the \red{assembly map} induced by 
the projection $\EGF{G}{\VCyc} \to \pt$
$$H_n^G(\EGF{G}{\VCyc},\bfL_R^{\langle -\infty\rangle}) \to 
H_n^G(\pt,\bfL_R^{\langle-\infty\rangle}) = L_n^{\langle-\infty\rangle}(RG)$$
for all $n \in \IZ$ (see Conjecture~\ref{con:FJC_for_L});

\item \red{The Baum-Cones Conjecture}
which predicts the bijectivity of the \red{assembly map} induced by 
the projection $\EGF{G}{\Fin} = \eub{G} \to \pt$
$$K_n^G(\eub{G}) = H_n^G(\EGF{G}{\Fin},\bfK^{\topo}) \to H_n^G(\pt,\bfK^{\topo}) 
= K_n(C^*_r(G))$$
for all $n \in \IZ$ (see Conjecture~\ref{con:BCC}).

\end{itemize}

\begin{remark}[\blue{The Isomorphism Conjectures interpreted 
as induction theorems}]
These Conjecture can be thought of a kind of \red{generalized induction theorem}.
They allow to compute the value of a functor such as $K_n(RG)$ on $G$ in terms of
its values $K_m(RH)$ for all $m \le n$ and all virtually cyclic subgroups 
subgroups $H$ of $G$.
\end{remark}

Next we want to investigate, whether one can pass to smaller or larger
families in the formulations of the Conjectures. The point  is to find
the  family as small as possible.

\begin{theorem}[\blue{Transitivity Principle}]
\label{the:transitivity_principle}
Let $\calf \subseteq \calg$ be two families of subgroups of $G$.
Let $\calh^?_*$ be an equivariant homology theory.
Assume that for every element $H \in \calg$ and $n \in \IZ$ the assembly map
$$\calh_n^H(\EGF{H}{\calf|_H}) \to \calh_n^H(\pt)$$
is bijective, where $\calf|_H = \{K \cap H \mid  K \in \calf\}$.

Then the \red{relative assembly map} induced by the up 
to $G$-homotopy unique $G$-map
$\EGF{G}{\calf} \to \EGF{G}{\calg}$
$$\calh^G_n(\EGF{G}{\calf}) \to \calh^G_n(\EGF{G}{\calg})$$
is bijective for all $n \in \IZ$.
\end{theorem}
\begin{proof}
See~\cite[Theorem~1.4]{Bartels-Lueck(2007ind)}.
\end{proof}

\begin{example}[\blue{Passage from $\Fin$ to $\VCyc$ for the Baum-Connes Conjecture}]
The Baum-Connes Conjecture~\ref{con:BCC} is known to be true for virtually cyclic groups.
The Transitivity Principle~\ref{the:transitivity_principle} 
implies that the relative assembly
$$K_n^G(\eub{G}) \xrightarrow{\cong} K_n^G(\EGF{G}{\VCyc})$$
is bijective for all $n \in \IZ$.
 
Hence it does not matter in the context of the Baum-Connes Conjecture whether we
consider the family $\Fin$ or $\VCyc$.
\end{example}

\begin{example}[\blue{Passage from $\Fin$ to $\VCyc$ for the Farrell-Jones Conjecture}]
In general the relative assembly maps
\begin{eqnarray*}
 H_n^G(\eub{G};\bfK_R) & \to & H_n^G(\EGF{G}{\VCyc};\bfK_R);
 \\
 H_n^G(\eub{G};\bfL^{\langle - \infty \rangle}_R) & \to & H_n^G(\EGF{G}{\VCyc};\bfL^{\langle -
   \infty \rangle}_R),
\end{eqnarray*}
are not bijective. Hence in the Farrell-Jones setting one has to pass 
to $\VCyc$ and cannot use the easier to handle family $\Fin$.
\end{example}

\begin{example}[\blue{The Farrell-Jones Conjecture 
for algebraic $K$-theory for the group $\IZ$}]
The Farrell-Jones Conjecture~\ref{con:FJC_for_K} for algebraic $K$-theory 
for the group $\IZ$ is true for trivial reasons since
$\IZ$ is virtually cyclic and hence the projection 
$\EGF{\IZ}{\VCyc} \to \pt$ is a homotopy equivalence. 
\end{example}

\begin{example}[\blue{The Farrell-Jones Conjecture 
for algebraic $K$-theory for the group $\IZ$ and the family $\Fin$}]
\label{exa:K-FJC_for_Z_and_Fin}
One may wonder what happens if we insert the family of finite
subgroups, i.e., whether the map induced by the projection
$\EGF{\IZ}{\Fin} = \eub{\IZ} \to \pt$
\begin{eqnarray}
    & H_n^{\IZ}(\eub{\IZ},\bfK_R) \to H_n^{\IZ}(\pt,\bfK_R) =
    K_n(R[\IZ]) &
    \label{K-assembly_for_Z_and_Fin}
\end{eqnarray}
is bijective. Since $\IZ$ is torsionfree, $\eub{\IZ}$ is the same
    as $E\IZ$ and the induction structure yields an isomorphism
$$H_n^{\IZ}(\eub{\IZ},\bfK_R) = H_n(BZ,\bfK_R) = H_n(S^1,\bfK_R) 
= K_n(R[\IZ]) \oplus K_{n-1}(R[\IZ]).$$ Hence the 
map~\eqref{K-assembly_for_Z_and_Fin} can be identified with the map
$$K_n(R)\oplus K_{n-1}(R) \to K_n(R[\IZ]).$$
However, by the \red{Bass-Heller Swan decomposition} we have the
isomorphism
$$K_n(R) \oplus K_{n-1}(R) \oplus  \NK_n(R) \oplus \NK_n(R))
\xrightarrow{\cong} K_n(R[t,t^{-1}]) \cong K_n(R[\IZ]).$$ Hence the
map \eqref{K-assembly_for_Z_and_Fin} is bijective if and only if
$\NK_n(R) = 0$.  We have $\NK_n(R) = 0$ under the assumption that $R$
is regular. This is the reason why we have required $R$ to be regular in the version of
the Farrell-Jones Conjecture for torsionfree
groups~\ref{con:FJC_for_K_torsionfree}.
\end{example}

\begin{definition}[\blue{Types of virtually cyclic groups}]
An infinite  virtually cyclic group $G$ is called of
\red{type $I$} if it admits an epimorphism onto $\IZ$ and of
\red{type $II$} if and only if admits an epimorphism onto $D_{\infty}$.
Let \red{$\VCyc_I$} be the family of virtually cyclic subgroups 
which are either finite or of type I.
\end{definition} 

An infinite virtually cyclic group is either of type I or of type {II}.
An infinite subgroups of a virtually cyclic subgroup of type I is again of type I.

\begin{theorem}[\green{L\"uck (2004), Quinn (2007), Reich (2007)}]
\label{the:passage_from_VCcy_I_to_VCyc}
The following maps are bijective for all $n \in \IZ$
\begin{eqnarray*}
  H_n^G(\EGF{G}{{\VCyc}_{I}};\bfK_R) & \to & H_n^G(\EGF{G}{\VCyc};\bfK_R);
   \\
   H_n^G(\eub{G};\bfL^{\langle - \infty \rangle}_R)
   & \to &
   H_n^G(\EGF{G}{\VCyc_{I}};\bfL^{\langle - \infty \rangle}_R).
   \end{eqnarray*}
\end{theorem}
\begin{proof}
See~\cite[Lemma~4.2]{Lueck(2005heis)} and~\cite{Reich(2007)}.
\end{proof}

\begin{theorem}[\green{Cappell (1973), Grunewald (2005), Waldhausen (1978)}]
\ \newline
\vspace{-4mm}\begin{enumerate}
\item The following maps are bijective for all $n \in \IZ$.
\begin{eqnarray*}
H_n^G(\eub{G};\bfK_{\IZ}) \otimes_{\IZ} \IQ & \to & 
H_n^G(\EGF{G}{\VCyc};\bfK_{\IZ})  \otimes_{\IZ} \IQ;
\\
H_n^G(\eub{G};\bfL^{\langle - \infty \rangle}_{R})\left[\frac{1}{2}\right] & \to &
H_n^G(\EGF{G}{\VCyc};\bfL^{\langle - \infty \rangle}_{R})\left[\frac{1}{2}\right];
\end{eqnarray*}

\item If $R$ is regular and $\IQ \subseteq R$, then  for all $n \in \IZ$ we get a bijection
$$ H_n^G(\eub{G};\bfK_R) \to  H_n^G(\EGF{G}{\VCyc};\bfK_R).$$
\end{enumerate}

\end{theorem}
\begin{proof}
See~\cite{Cappell(1974c)}, \cite[Theorem~5.6]{Grunewald(2006)},
\cite[Proposition~2.6 on page~686, Proposition~2.9 and 
Proposition~2.10 on page~688]{Lueck-Reich(2005)}.
\end{proof}

\begin{theorem}[\green{Bartels (2003)}]
\label{the:Bartels}
For every $n \in \IZ$ the two maps
\begin{eqnarray*}
 H_n^G(\eub{G};\bfK_R) & \to & H_n^G(\EGF{G}{\VCyc};\bfK_R);
 \\
 H_n^G(\eub{G};\bfL^{\langle - \infty \rangle}_R) & \to & H_n^G(\EGF{G}{\VCyc};\bfL^{\langle -
   \infty \rangle}_R),
 \end{eqnarray*}
are split injective.
\end{theorem}
\begin{proof}
See~\cite{Bartels(2003b)}.
\end{proof}

Hence we get (natural) isomorphisms
\begin{multline}
 H_n^G(\EGF{G}{\VCyc};\bfK_R) 
 \cong
 H_n^G(\eub{G};\bfK_R) \oplus H_n^G(\EGF{G}{\VCyc},\eub{G};\bfK_R);
\end{multline}
and
\begin{multline*}
H_n^G(\EGF{G}{\VCyc};\bfL^{\langle - \infty \rangle}_R)
\cong
 H_n^G(\eub{G};\bfL^{\langle - \infty \rangle}_R) \oplus
 H_n^G(\EGF{G}{\VCyc},\eub{G};\bfL^{\langle - \infty \rangle}_R).
\end{multline*}

The analysis of the terms $ H_n^G(\EGF{G}{\VCyc},\eub{G};\bfK_R)$ and
$H_n^G(\EGF{G}{\VCyc},\eub{G};\bfL^{\langle - \infty \rangle}_R)$ boils down to
investigating \red{Nil-terms} and \red{UNil-terms} in the sense of
\green{Waldhausen} and \green{Cappell}.
The analysis of the terms $ H_n^G(\eub{G};\bfK_R)$ and
$H_n^G(\eub{G};\bfL^{\langle - \infty \rangle}_R)$ is using the methods of the previous
lecture (e.g., equivariant  Chern characters).

\begin{remark}[\blue{Relating the torsionfree versions to the general versions}]
Obviously the general version of the Baum-Connes Conjecture~\ref{con:BCC} 
reduces in the torsionfree case to the version~\ref{con:BCC_torsionfree}
since for torsionfree $G$ we have $EG = \eub{G}$.

The general version of the Farrell-Jones Conjecture~\ref{con:FJC_for_K}
for $K$-theory  reduces in the torsionfree case to the 
version~\ref{con:FJC_for_K_torsionfree}
because of the Transitivity Principal~\ref{the:transitivity_principle}
since a torsionfree virtually cyclic group is isomorphic to $\IZ$
and for a regular ring $R$ the Bass-Heller-Swan decomposition shows that
the map $H_n^{IZ}(\eub{\IZ};\bfK_R) \to K_n(R[\IZ])$ is bijective (as explained in
Example~\ref{exa:K-FJC_for_Z_and_Fin}).

The general version of the Farrell-Jones Conjecture~\ref{con:FJC_for_L}
for $L$-theory  reduces in the torsionfree case to the 
version~\ref{con:FJC_for_L_torsionfree}
because of the Transitivity Principal~\ref{the:transitivity_principle}
since a torsionfree virtually cyclic group is isomorphic to $\IZ$
and the map $H_n^{\IZ}(\eub{\IZ};\bfL^{\langle -\infty \rangle}_R) 
\to L_n^{\langle -\infty \rangle}(R[\IZ])$ is bijective by
Theorem~\ref{the:passage_from_VCcy_I_to_VCyc}
\end{remark}

Next we explain in the case $G = SL_2(\IZ)$ how computations
are made possible by the Farrell-Jones and the Baum-Connes Conjecture.

\begin{example}[\blue{$K$-theory of $C^*_r(SL_2(\IZ))$ and of $\IZ[SL_2(\IZ)]$}]
From Example~\ref{exa:eub_for_SL_2(Z)} we obtain a $SL_2(\IZ)$-pushout
$$
\xymatrix@!C=15em{SL_2(\IZ)/(\IZ/2) \times \{-1,1\} \ar[d]
\ar[r]^-{F_{-1} \coprod F_1}
&
SL_2(\IZ)/(\IZ/4) \coprod SL_2(\IZ)/(\IZ/6) \ar[d]
\\
SL_2(\IZ)/(\IZ/2) \times [-1,1] \ar[r]
&
T = \underline{E}SL_2(\IZ)
}
$$
Let $\calh^?_*$ be an equivariant homology theory. Then the Mayer-Vietoris sequence 
applied to the $SL_2(\IZ)$-pushout above together with the induction structure
yields a long exact sequence
\begin{multline}
\cdots \to \calh^{\IZ/2}_{n}(\pt) 
\to \calh^{\IZ/4}_{n}(\pt) \oplus \calh^{\IZ/6}_{n}(\pt) \to
\calh^{SL_2(\IZ)}_{n}(\eub{SL_2(\IZ)}) 
\\
\to \calh^{\IZ/2}_{n-1}(\pt) 
\to \calh^{\IZ/4}_{n-1}(\pt) \oplus \calh^{\IZ/6}_{n-1}(\pt) \to \cdots.
\label{Mayer-Vietoris_for_SL_2(Z)_and_equiv_homo_theory}
\end{multline}
The Baum-Connes Conjecture~\ref{con:BCC} is known to be true for $SL_2(\IZ)$
(see for instance~\cite{Higson-Kasparov(2001)}).
Hence in the case, where $\calh^?_*$ is equivariant topological
$K$-theory, the long exact 
sequence~\eqref{Mayer-Vietoris_for_SL_2(Z)_and_equiv_homo_theory} reduces to the exact
sequences
\begin{multline*}
0 \to K_1(C^*_r(SL_2(\IZ))) \to R_{\IC}(\IZ/2) \to R_{\IC}(\IZ/4) \oplus R_{\IC}(\IZ/6)
\\
\to K_0(C^*_r(SL_2(\IZ))) \to 0,
\end{multline*}
where the map between the representation rings are induced by the obvious inclusions
of groups. Since the inclusion $\IZ/2 \to \IZ/6$ is split injective and
$R_{\IC}(\IZ/2) \cong \IZ^2$, $R_{\IC}(\IZ/4) \cong \IZ^4$ and
$R_{\IC}(\IZ/6) \cong \IZ^6$, we conclude
$$K_n(C^*_r(SL_2(\IZ))) \cong
\left\{
\begin{array}{lcl}
\IZ^8 && n \text{ even;}
\\
0 && n \text{ odd.}
\end{array}
\right.
$$
The Farrell-Jones Conjecture for $K$-theory is known to be true for $SL_2(\IZ)$
for any coefficient ring $R$ by~\cite{Bartels-Lueck-Reich(2007hyper)}
since it contains a finitely generated free subgroup of finite index and is hence
a hyperbolic group. Because of
Theorem~\ref{the:passage_from_VCcy_I_to_VCyc} and
Theorem~\ref{the:Bartels} we obtain an isomorphism
\begin{multline}
K_n(R[SL_2(\IZ)]) 
\cong  
H_n^{SL_2(\IZ)}(\EGF{SL_2(\IZ)}{\Fin};\bfK_R) 
\\
 \oplus H_n^{SL_2(\IZ)}(\EGF{SL_2(\IZ)}{\VCyc_I},\EGF{SL_2(\IZ)}{\Fin};\bfK_R).
\label{Bartels_for_K_ast(RSL_2(Z))} 
\end{multline}
Let $V \subseteq SL_2(\IZ)$ be a virtually cyclic subgroup of type I, i.e.,
there is an exact sequence $1 \colon F \to V \to \IZ \to 1$ for a finite subgroup
$F \subseteq V$. Since $SL_2(\IZ) \cong \IZ/4 \ast_{\IZ/2} \IZ/6$,
$F$ is conjugated to $\IZ/4$, $\IZ/6$ or the subgroup $\IZ/2$. Since
the normalizers of $\IZ/6$ and $\IZ/4$ are finite, $F$ must be subconjugated to $\IZ/2$.
Since $\IZ/2$ is the center of $SL_2(\IZ)$, the group $V$ is isomorphic
to $\IZ/2 \times \IZ$. The group $\NK_n(\IZ[\IZ/2])$ vanishes for $n \le 1$.
Using the Bass-Heller-Swan decomposition we see that
$H_n^{V}(\EGF{V}{\VCyc_I},\EGF{V}{\Fin};\bfK_{\IZ}) = 0$ for $n \le 1$.
An obvious modification of the Transitivity Principal~\ref{the:transitivity_principle}
(see~\cite[Theorem~1.4]{Bartels-Lueck(2007ind)})
implies that for $n \le 1$ the group 
$H_n^{SL_2(\IZ)}(\EGF{SL_2(\IZ)}{\VCyc_I},\EGF{SL_2(\IZ)}{\Fin};\bfK_{\IZ})$
vanishes. Thus from~\eqref{Bartels_for_K_ast(RSL_2(Z))}
we obtain an isomorphism for $n \le 1$.
$$K_n(\IZ[SL_2(\IZ)]) 
\cong  
H_n^{SL_2(\IZ)}(\EGF{SL_2(\IZ)}{\Fin};\bfK_{\IZ})$$
Hence the long exact sequence~\eqref{Mayer-Vietoris_for_SL_2(Z)_and_equiv_homo_theory}
yields the long exact sequence
\begin{multline*}
K_1(\IZ[\IZ/2]) \to K_1(\IZ[\IZ/4]) \oplus K_1(\IZ[\IZ/6]) \to K_1(\IZ[SL_2(\IZ)])
\to K_0(\IZ[\IZ/2]) 
\\
\to K_0(\IZ[\IZ/4]) \oplus K_0(\IZ[\IZ/6]) \to K_0(\IZ[SL_2(\IZ)])
\to K_{-1}(\IZ[\IZ/2]) 
\\
\to K_{-1}(\IZ[\IZ/4]) \oplus K_{-1}(\IZ[\IZ/6]) 
\to K_{-1}(\IZ[SL_2(\IZ)])
\to K_{-2}(\IZ[\IZ/2]) 
\\
\to K_{-2}(\IZ[\IZ/4]) \oplus K_{-2}(\IZ[\IZ/6]) 
\to K_{-2}(\IZ[SL_2(\IZ)])  \to \cdots
\end{multline*}
It induces the long exact sequence
\begin{multline*}
\Wh(\IZ/2) \to \Wh(\IZ/4) \oplus \Wh(\IZ/6) \to \Wh(SL_2(\IZ))
\to \widetilde{K}_0(\IZ[\IZ/2]) 
\\
\to \widetilde{K}_0(\IZ[\IZ/4]) \oplus \widetilde{K}_0(\IZ[\IZ/6]) 
\to \widetilde{K}_0(\IZ[SL_2(\IZ)])
\to K_{-1}(\IZ[\IZ/2]) 
\\
\to K_{-1}(\IZ[\IZ/4]) \oplus K_{-1}(\IZ[\IZ/6]) 
\to K_{-1}(\IZ[SL_2(\IZ)])
\to K_{-2}(\IZ[\IZ/2]) 
\\
\to K_{-2}(\IZ[\IZ/4]) \oplus K_{-2}(\IZ[\IZ/6]) 
\to K_{-2}(\IZ[SL_2(\IZ)])  \to \cdots
\end{multline*}
The groups $\Wh(\IZ/2)$, $\Wh(\IZ/4)$, $\Wh(\IZ/6)$,
$\widetilde{K}_0(\IZ[\IZ/2])$, $\widetilde{K}_0(\IZ[\IZ/4])$,
$\widetilde{K}_0(\IZ[\IZ/6])$, 
$\widetilde{K}_{-1}(\IZ[\IZ/2])$, $\widetilde{K}_{-1}(\IZ[\IZ/4])$
vanish, whereas $\widetilde{K}_{-1}(\IZ[\IZ/6]) \cong \IZ$
(see \green{Bass}~\cite[Theorem~10.6 on page~695]{Bass(1968)},
\green{Carter}~\cite{Carter(1980)},
\green{Cassou-Nogu\'es}~\cite{Cassou-Nogues(1973)},
\green{Curtis-Rainer}~\cite[Corollary~50.17 on page~253]{Curtis-Reiner(1987)}),
\green{Oliver}~\cite[Theorem~14.1 on page~328]{Oliver(1989)}.
The groups $K_n(\IZ[H])$ vanish for all $n \ge -2$ and all finite groups $H$
(see~\green{Carter}~\cite{Carter(1980)}). Hence we get
$$\begin{array}{lclcl}
\Wh(SL_2(\IZ)) & \cong  & 0; && 
\\
\widetilde{K}_0(\IZ[SL_2(\IZ)]) & \cong & 0; &&
\\
K_{-1}(\IZ[SL_2(\IZ)]) & \cong & \IZ: &&
\\
K_{n}(\IZ[SL_2(\IZ)]) & \cong & 0 && \text{for } n \le -2.
\end{array}
$$
\end{example}

Next we show that the Farrell-Jones Conjecture and the Baum-Conjecture imply
certain other well-known conjectures.

\begin{conjecture}[\blue{Novikov Conjecture}]
The \emphred{Novikov Conjecture for $G$} predicts for a closed oriented manifold
$M$ together with a map $f \colon M \to BG$ that for any $x \in H^*(BG)$
 the \red{higher signature}
$$\red{\sign_x(M,f)} := \langle \call(M) \cup f^*x,[M]\rangle$$
is an oriented homotopy invariant of $(M,f)$,  i.e.,  for
every orientation preserving homotopy equivalence of closed
oriented manifolds $g \colon M_0 \to M_1$ and homotopy equivalence
$f_i \colon M_0 \to M_1$ with $f_1 \circ g \simeq f_2$ we have
$$\sign_x(M_0,f_0) = \sign_x(M_1,f_1).$$
\end{conjecture}

\begin{theorem}[\blue{The Farrell-Jones, the Baum-Connes and the Novikov Conjecture}]
Suppose that one of the following assembly maps
\begin{eqnarray*}
H_n^G(\EGF{G}{\VCyc},\bfL_R^{\langle -\infty\rangle}) 
& \to &
H_n^G(\pt,\bfL_R^{\langle-\infty\rangle}) = L_n^{\langle-\infty\rangle}(RG);
\\
K_n^G(\eub{G}) = H_n^G(\EGF{G}{\Fin},\bfK^{\topo}) 
& \to & 
H_n^G(\pt,\bfK^{\topo}) = K_n(C^*_r(G)),
\end{eqnarray*}
is rationally injective.

Then the Novikov Conjecture holds for the group $G$.
\end{theorem}
\begin{proof}
See for instance~\cite[Proposition~3.1 and Proposition~3.5 on page~699]{Lueck-Reich(2005)} 
and~\cite[Proposition~6.6 on page~300]{Ranicki(1995b)}.
\end{proof}

For more information about the Novikov Conjecture we refer for instance
to~\cite{Carlsson(2004)}, 
\cite{Carlsson-Pedersen(1995a)},
\cite{Davis(2000)},
\cite{Farrell(2002)}, 
\cite{Ferry-Ranicki-Rosenberg(1995)}, 
\cite{Kreck-Lueck(2005)},
\cite{Ranicki(1992)} and 
\cite{Rosenberg(1995)}. 

\begin{theorem}[\blue{Induction from finite subgroups}, 
\green{Bartels-L\"uck-Reich (2007)}]
\label{the:colim_fin}
\ \\
\begin{enumerate}
\item
Let $R$ be a regular ring such that the order of any finite subgroup
of $G$ is invertible in $R$. Suppose $G \in \calfj_K(R)$.
Then the map given by induction from finite
subgroups of $G$
$$\colim_{\OrGF{G}{\Fin}} K_0(RH) \to K_0(RG) $$
is bijective;

\item
Let $F$ be a field of characteristic $p$ for a prime number $p$.
Suppose that $G \in \calfj_K(F)$. Then the map
$$\colim_{\OrGF{G}{\Fin}} K_0(FH)[1/p] \to
K_0(FG)[1/p]$$
is bijective;

\item If $G \in \calf_K(\IZ)$, then 
the canonical map
$$\colim_{\OrGF{G}{\Fin}} K_{-1}(\IZ H) \to K_{-1}(\IZ G)$$
is bijective;

\item If $G \in \calf_K(\IZ)$, then 
$$K_n(\IZ G) = 0 \text{ for } n \le -2.$$
\end{enumerate}
\end{theorem}
\begin{proof}
See \green{Bartels-L\"uck-Reich}~\cite[Theorem~0.5]{Bartels-Lueck-Reich(2007appl)},
\cite[1.65 on page~260]{Farrell-Jones(1993a)}, and
\green{L\"uck-Reich}~\cite[Section~3.1.1 on page~690]{Lueck-Reich(2005)}.
\end{proof}

\begin{theorem}[\blue{Permutation Modules}, \green{Bartels-L\"uck-Reich (2007)}]
Suppose that $G \in \calfj_K(\IQ)$.
Then for every finitely generated projective $\IQ[G]$-module $P$ there exists
integers $k \ge 1$ and $l \ge 0$ and finitely many finite subgroups $H_1$, $H_2$, $\ldots $, $H_r$
such that
$$P^k \oplus \IQ[G]^l \cong_{\IQ[G]} \IQ[G/H_1] \oplus \IQ[G/H_2] \oplus \cdots \oplus \IQ[G/H_r].$$
\end{theorem}
\begin{proof}
Because of~\cite[Lemma~4.3 and Lemma~4.4]{Bartels-Lueck-Reich(2007appl)} 
it suffices to prove the claim
in the case, where $G$ is finite cyclic. This special case follows 
from \green{Segal}~\cite{Segal(1972)}.
\end{proof}

Next we introduce some notation.
$R$ be commutative ring and let $G$ be a group.
Let \red{$\class(G,R)$} be the $R$-module of \red{class functions} 
$G \to R$, i.e., functions
$G \to R$ which are constant on conjugacy classes.
Let $\tr_{RG} \colon RG \to \class(G,R)$ be the $R$-homomorphism
which sends $g \in G$ to the class function which takes the value one on the conjugacy 
class of $g$ and the value zero otherwise.
It extends to a map
$$\red{\tr_{RG} \colon M_n(RG) \to \class(G,R)}$$
by taking the sums of the values of the diagonal entries.

Let $P$ be a finitely generated $RG$-module.
Choose a finitely generated projective $RG$-module $Q$ and an isomorphism
$\phi \colon RG^n \xrightarrow{\cong} P \oplus Q$. 
Let $A \in M_n(RG)$ be a matrix
such that 
$\phi^{-1} \circ (\id_P \oplus 0) \circ \phi \colon RG^n \to RG^n$ is given by $A$.

\begin{definition}[\blue{Hattori-Stallings rank}]
\label{def:Hattor-Stallings_rank}
Define the \red{Hattori-Stallings rank} of $P$ to be the class function
$$\red{\HS_{RG}(P)} := \tr_{RG}(A).$$
\end{definition}

This definition is independent of the choice of $Q$ and $\phi$.
Let $G$ be a finite group and let $F$ be a field of
characteristic zero.  Then a finitely generated $RG$-module
$P$ is the same as a finite dimensional $G$-representation over
$F$  and the Hattori-Stallings rank can be identified with the
character of the $G$-representation (see~\eqref{relating_charac_and_HS-rank}).

\begin{conjecture}[\blue{Bass Conjecture}]
\label{con:Bass_Conjecture_for_integral_domains}
Let $R$ be a commutative integral domain and let $G$ be a group. 
Let $g \not= 1$ be an element in $G$. 
Suppose that either the order $|g|$ is
infinite or that the order $|g|$ is finite and not invertible in $R$.

Then the \emphred{Bass Conjecture} predicts that 
for every finitely generated projective $RG$-module $P$ the value of
its \emphred{Hattori-Stallings rank} $\HS_{RG}(P)$ at $(g)$ is trivial.
\end{conjecture}

If $G$ is finite,  the Bass Conjecture~\ref{con:Bass_Conjecture_for_integral_domains} 
reduces to a theorem of \green{Swan (1960)}
(see~\cite[Theorem~8.1]{Swan(1960a)}, \cite[Corollary~4.2]{Bass(1979)}).

The next results follows from the argument in\cite[Section~5]{Farrell-Linnell(2003b)}.

\begin{theorem}[\green{Linnell-Farrell (2003)}] 
  Let $G$ be a group. Suppose that
  $$\colim_{\OrGF{G}{\Fin}} K_0(FH)  \otimes_{\IZ} \IQ \to K_0(FG) \otimes_{\IZ} \IQ $$
  is surjective for
  all fields $F$ of prime characteristic. (This is true if $G \in \calfj_K(F)$
  for every field $F$ of prime characteristic).

Then the Bass Conjecture is satisfied for every integral domain $R$.
\end{theorem}

\begin{remark}[\blue{Geometric interpretation of the Bass Conjecture}]
The Bass Conjecture~\ref{con:Bass_Conjecture_for_integral_domains} can be interpreted
topologically. Namely, the Bass Conjecture~\ref{con:Bass_Conjecture_for_integral_domains}
is true for a finitely presented group $G$ in the case $R = \IZ$ if and only if every
homotopy idempotent selfmap of an oriented smooth closed manifold whose dimension is
greater than 2 and whose fundamental group is isomorphic to $G$ is homotopic to 
a selfmap which has precisely one fixed point 
(see \green{Berrick-Chatterji-Mislin}~\cite{Berrick-Chatterji-Mislin(2007)}). 

The Bass Conjecture~\ref{con:Bass_Conjecture_for_integral_domains} for $G$ 
in the case $R = \IZ$ (or $R = \IC$) also implies for a
finitely dominated $CW$-complex with fundamental group $G$ that its Euler characteristic
agrees with the $L^2$-Euler characteristic of its universal covering
(see \green{Eckmann}~\cite{Eckmann(1996b)}).
\end{remark}

Next we present another version of the Bass Conjecture.
Let $F$ be a field of characteristic zero. Fix an integer $m \ge 1$.
Let $F(\zeta_m) \supset F$ be the Galois extension given by adjoining
the primitive $m$-th root of unity $\zeta_m$ to $F$. Denote by $\Gamma
(m,F)$ the Galois group of this extension of fields, i.e., the group
of automorphisms $\sigma\colon F(\zeta_m) \to F(\zeta_m)$ which induce
the identity on $F$. It can be identified with a subgroup of $\IZ/m^*$
by sending $\sigma$ to the unique element $u({\sigma}) \in \IZ/m^*$
for which $\sigma(\zeta_m) = \zeta_m^{u(\sigma)}$ holds.  Let $g_1$
and $g_2$ be two elements of $G$ of finite order.  We call them
\emphred{$F$-conjugate} if for some (and hence all) positive integers $m$
with $g_1^m = g_2^m = 1$ there exists an element $\sigma$ in the
Galois group $\Gamma (m,F)$ with the property that $g_1^{u(\sigma)}$
and $g_2$ are conjugate. Two elements $g_1$ and $g_2$ are
$F$-conjugate for $F = \IQ$, $\IR$ or $\IC$ respectively if the cyclic
subgroups $\langle g_1 \rangle$ and $\langle g_2 \rangle$ are
conjugate, if $g_1$ and $g_2$ or $g_1$ and $g_2^{-1}$ are conjugate,
or if $g_1$ and $g_2$ are conjugate respectively.

Denote by \red{$\con_F(G)_f$} the set of $F$-conjugacy classes $(g)_F$ of
elements $g \in G$ of finite order.  Let \red{$\class_F(G)_f$} be the
$F$-vector space with the set $\con_F(G)_f$ as basis, or,
equivalently, the $F$-vector space of functions $\con_F(G)_f \to F$
with finite support. 

\begin{conjecture}[\blue{Bass Conjecture for fields of characteristic zero as coefficients}]
\label{con:Bass_Conjecture_for_FG}
Let $F$ be a field of characteristic zero and let $G$ be a group.  The
Hattori-Stallings  (see Definition~\ref{def:Hattor-Stallings_rank})
induces an isomorphism
$$\HS_{FG} \colon K_0(FG) \otimes_{\IZ} F \to \class_F(G)_f.$$
\end{conjecture}

\begin{lemma} \label{lem:HS_iso_for_finite_G}
  Suppose that $F$ is a field of characteristic zero and $G$ is a
  finite group.  Then Conjecture~\ref{con:Bass_Conjecture_for_FG} is true.
\end{lemma}
\begin{proof}
  Since $G$ is finite, an $FG$-module is a finitely generated
  projective $FG$-module if and only if it is a (finite-dimensional)
  $G$-representation with coefficients in $F$ and $K_0(FG)$ is the
  same as the representation ring $R_F(G)$.  The Hattori-Stallings
  rank $\HS_{FH}(V)$ and the character $\chi_V$ of a
  $G$-representation $V$ with coefficients in $F$ are related by the
  formula
  \begin{eqnarray}
  \chi_V(g) & = &|Z_G\langle g \rangle| \cdot \HS_{FG}(V)(g)
  \label{relating_charac_and_HS-rank}
  \end{eqnarray}
  for $g \in G$, where $Z_G\langle g \rangle$ is the centralizer of $g$ in
  $G$.  Hence Lemma~\ref{lem:HS_iso_for_finite_G} follows from
  representation theory, see for instance \cite[Corollary~1 on page~96]{Serre(1997)}.
\end{proof}

Here is a conjecture related to the Bass Conjecture
\begin{conjecture}
\label{con:Bass_change_of_rings}
Let $R$ be an integral domain with quotient field $F$.
Suppose that no prime divisor of the order of a finite subgroup of
$G$ is a unit in $R$.
Then the change of rings homomorphism 
$$K_0(RG) \otimes_{\IZ} \IQ \to K_0(FG) \otimes_{\IZ} \IQ$$
factorizes as
$$K_0(RG) \otimes_{\IZ} \IQ \to K_0(R) \otimes_{\IZ} \IQ 
\to K_0(F) \otimes_{\IZ} \IQ \to K_0(FG) \otimes_{\IZ} \IQ.$$
\end{conjecture}

\begin{theorem}[\green{Bartels-L\"uck-Reich (2007)}]
Let $R$ be an integral domain with quotient field $F$.
Suppose that no prime divisor of the order of a finite subgroup of
$G$ is a unit in $R$. Suppose that $G$ belongs to $\calf_K(R)$.

Then Conjecture~\ref{con:Bass_change_of_rings} is true for $G$ and $R$.
\end{theorem}
\begin{proof}
See~\cite[Theorem~0.10]{Bartels-Lueck-Reich(2007appl)}.
\end{proof}

More information and further references 
about the Bass Conjecture can be found for instance in
\cite{Bass(1976)},
\cite[Section 7]{Berrick-Chatterji-Mislin(2004)},
\cite{Burger-Valette(1998)},
\cite{Eckmann(1986)},
\cite{Eckmann(1996b)},
\cite{Farrell-Linnell(2003b)},
\cite{Linnell(1983b)} 
\cite[Subsection 9.5.2]{Lueck(2002)}, 
and 
\cite[page 66ff]{Mislin-Valette(2003)}.

\begin{conjecture}[\blue{Vanishing of Bass-Nil-groups}]
\label{con:Vanishing_of_Nil_for_regular_rings_R}
Let $R$ be a regular ring with $\IQ \subseteq R$. 
Then we get
for all groups $G$ and all $n \in \IZ$ that
$$\NK_n(RG) = 0.$$
\end{conjecture}

The relation of this conjecture to the Farrell-Jones Conjecture
is discussed in~\cite[Section~6.3]{Bartels-Lueck-Reich(2007appl)}. 

Next we discuss some connections of the Farrell-Jones Conjecture to
$L^2$-invariants. For more information and some explanations 
about $L^2$-invariants we refer for instance 
to \green{L\"uck}~\cite{Lueck(2002)}.

The \red{$L^2$-torsion} of a closed Riemannian manifold $M$ is 
defined in terms of the heat kernel on the universal covering.
If $M$ is hyperbolic and has odd dimension, its
$L^2$-torsion is up to a (non-vanishing) dimension constant
its volume (see~\cite{Hess-Schick(1998)}).

\begin{conjecture}[\blue{Homotopy invariance of $L^2$-torsion}]
Let $X$ and $Y$ be det-$L^2$-acyclic finite $G$-$CW$-complexes, 
which are $G$-homotopy equivalent.

Then their \emph{$L^2$-torsion} agree:
$$\rho^{(2)}(X;\caln(G)) = \rho^{(2)}(Y;\caln(G)).$$
\end{conjecture}

The conjecture above allows to extend the notion of volume to
hyperbolic groups whose $L^2$-Betti numbers all vanish.

\begin{theorem}[\green{L\"uck (2002)}]
Suppose that $G \in \calfj_K(\IZ)$. Then $G$ satisfies the
Conjecture above.
\end{theorem}
\begin{proof}
See~\cite[Theorem~0.14]{Bartels-Lueck-Reich(2007appl)}.
\end{proof}

\begin{remark}[\blue{$p$-adic Fuglede-Kadison determinant}] \green{Deninger}
  defines a \red{$p$-adic Fuglede-Kadison determinant} for a group $G$
  and relates it to \red{$p$-adic entropy} provided that
  $\Wh^{\IF_p}(G) \otimes_{\IZ} \IQ$ is trivial.
\end{remark}

\begin{remark}[\blue{Atiyah Conjecture}]
The surjectivity of the map
$$\colim_{\OrGF{G}{\Fin}} K_0(\IC H) \to K_0(\IC G) $$
plays a role \red{(33 \%)} in a program to prove the \red{Atiyah
  Conjecture}.  It says that for a closed Riemannian manifold with
torsionfree fundamental group the $L^2$-Betti numbers of its universal
covering are all integers.

The Atiyah Conjecture is rather surprising in view of the analytic
definition of the $L^2$-Betti numbers by
$$b_p^{(2)}(M) := \lim_{t \to \infty} 
\int_F e^{-t\widetilde{\Delta}_p}(\widetilde{x},\widetilde{x})
dvol_{\widetilde{M}},$$ where $F$ is a fundamental domain for the
$\pi_1(M)$-action on $\widetilde{M}$.
\end{remark}

Next we explain the relation of the Baum-Connes Conjecture 
to the Gromov-Lawson-Rosenberg Conjecture.

\begin{definition}[\blue{Bott manifold}]
A \emphred{Bott manifold}
is any simply connected closed $\Spin$-manifold $B$ of dimension $8$ whose
$\widehat{A}$-genus $\widehat{A}(B)$ is $8$.
\end{definition}

We fix a choice of a Bott manifold. (The particular choice does not matter.)
Notice that the index defined in terms of the Dirac  operator
$\ind_{C^*_r(\{1\};\IR)}(B) \in KO_8(\IR) \cong \IZ$ is a generator
and the product with this element induces the Bott periodicity isomorphisms
$KO_n(C^*_r(G;\IR)) \xrightarrow{\cong}  KO_{n+8}(C^*_r(G;\IR))$.
In particular
\begin{eqnarray*}
\ind_{C^*_r(\pi_1(M);\IR)}(M) & = & \ind_{C^*_r(\pi_1(M \times B);\IR)}(M \times B),
\end{eqnarray*}
if we identify $KO_n(C^*_r(\pi_1(M);\IR)) = KO_{n+8}(C^*_r(\pi_1(M);\IR))$ 
via Bott periodicity. If $M$ carries a Riemannian metric
with positive scalar curvature, then the index
$$\ind_{C^*_r(\pi_1(M);\IR)}(M) \in KO_n(C^*_r(\pi_1(M);\IR)),$$
which is defined in terms of the Dirac operator on the universal covering,
must vanish by the \red{Bochner-Lichnerowicz formula}.

\begin{conjecture}[\blue{(Stable) Gromov-Lawson-Rosenberg Conjecture}]
Let $M$ be a closed connected $\Spin$-manifold of dimension $n \ge 5$.

Then $M \times B^k$ carries for some integer $k \ge
0$ a Riemannian metric with positive scalar curvature if and only if
$$\ind_{C^*_r(\pi_1(M);\IR)}(M)  =  0 \hspace{5mm} \in KO_n(C^*_r(\pi_1(M);\IR)).$$
\end{conjecture}

\begin{theorem}[\green{Stolz (2002)}]
\label{the:Stolz}
Suppose that the assembly map for the real version of the Baum-Connes Conjecture
$$H_n^G(\eub{G};\bfKO^{\topo}) \to KO_n(C^*_r(G;\IR))$$
is injective for the group $G$.

Then the Stable Gromov-Lawson-Rosenberg Conjecture is true for all closed
$\Spin$-manifolds of dimension $\ge 5$ with $\pi_1(M) \cong G$.
\end{theorem}
\begin{proof}
See~\cite{Stolz(2002)}.
\end{proof}

The requirement $\dim(M) \ge 5$ in Theorem~\ref{the:Stolz} is essential in the
Stable Gromov-Lawson-Rosenberg Conjecture,
since in dimension four
new obstructions, the \red{Seiberg-Witten invariants}, occur.

Since the Baum-Connes Conjecture  is true for finite groups (for the trivial reason
that $\eub{G} = \pt$ for  finite groups $G$),
the Stable Gromov-Lawson Conjecture
holds for finite fundamental groups by Theorem~\ref{the:Stolz}.

\begin{remark}[\blue{The unstable version of the Gromov-Lawson-Rosenberg Conjecture}]
The \red{unstable version} of the Gromov-Lawson-Rosenberg Conjecture
says that $M$ carries a Riemannian metric with positive scalar curvature if and only if
the index $\ind_{C^*_r(\pi_1(M);\IR)}(M)$ vanishes. 
\green{Schick(1998)}~\cite{Schick(1998e)}
has constructed counterexamples to the unstable version using
minimal hypersurface methods due to \green{Schoen and Yau}.
It is not known whether the unstable version is true or false for
finite fundamental groups.
\end{remark}

\begin{question}[\blue{Status}]
For which groups are the Farrell-Jones Conjecture and 
the Baum-Connes Conjecture known to be true? 

What are open interesting cases?
\end{question}

\begin{question}[\blue{Methods of proof}]
What are the methods of proof?
\end{question}

\begin{question}[\blue{Relations}]
What are the relations between the Farrell-Jones Conjecture and the 
Baum-Connes Conjecture?
\end{question}

%%%%%%%%%%%%%%%%%%%%%%%%%%%%%%%%%%%%%%%%%%%%%%%%%%%%%%%%%%%%%%%%%%%%%%%%%%%%%%%%
%%%%%%%%%%%%%%%%%%%%%%%%%%%%%%%%%%%%%%%%%%%%%%%%%%%%%%%%%%%%%%%%%%%%%%%%%%%%%%%%
%%%%%%%%%%%%%%%%%%%%%%%%%%%%%%%%%%%%%%%%%%%%%%%%%%%%%%%%%%%%%%%%%%%%%%%%%%%%%%%%

\section{Summary, status and outlook}

The outline of this section is:

\begin{itemize}
\item We present other versions of the Isomorphism Conjecture;
\item We give a summary about the status of the Farrell-Jones 
and the Baum-Connes Conjecture:
\item We discuss open questions and problems.
\end{itemize}
    
\begin{conjecture}[\blue{Isomorphism Conjecture}]
\label{con:Isomorphism_Conjecture}
Let $\calh^?_*$ be an equivariant homology theory.
It satisfies the
\emphred{Isomorphism Conjecture} for the group $G$ and the family $\calf$
if the projection $\EGF{G}{\calf} \to \pt$ induces for all $n \in \IZ$ a bijection
$$\calh_n^G(\EGF{G}{\calf}) \to \calh_n^G(\pt).$$
\end{conjecture}

\begin{example}[\blue{The Farrell-Jones and the Baum-Connes Conjecture 
as special cases of the Isomorphism  Conjecture}]
The Farrell-Jones Conjecture for $K$-theory or $L$-theory
respectively with coefficients in $R$ 
(see~\ref{con:FJC_for_K} and~\ref{con:FJC_for_L} respectively) is the Isomorphism
Conjecture~\ref{con:Isomorphism_Conjecture} for $\calh^?_* = H^?_*(-;\bfK_R)$ or $\calh^?_* =
H^?_*(-;\bfL_R^{\langle -\infty \rangle})$ respectively and $\calf =
\VCyc$.  

The Baum-Connes Conjecture~\ref{con:BCC} is the Isomorphism
Conjecture~\ref{con:Isomorphism_Conjecture} 
for $\calh^?_* = K^?_* = H^?_*(-;\bfK^{\topo})$ and
$\calf = \Fin$.
\end{example}

There are functors $\calp$ and $A$ which assign to a space $X$ the
\red{space of pseudo-isotopies} and its \red{$A$-theory}.
Composing it with the functor sending a groupoid to its classifying space
yields functors $\bfP$ and $\bfA$ from $\Groupoids$ to $\Spectra$.
Thus we obtain equivariant homology theories \red{$H^?_*(-;\bfP)$} and
\red{$H^?_*(-;\bfA)$}. They satisfy  $H_n^G(G/H;\bfP) = \pi_n(\calp(BH))$ 
and  $H_n^G(G/H;\bfA) = \pi_n(A(BH))$.

Pseudo-isotopy and $A$-theory are important theories. 
In particular they
are closely related to the \red{space of selfhomeomorphisms} and the
\red{space of selfdiffeomorphisms} of closed manifolds.
For more information about $A$-theory and  pseudoisotopy we refer for instance
to~\cite{Burghelea-Lashof(1977)}, \cite[Section~9]{Dwyer-Weiss-Williams(2003)}), 
\cite{Hatcher(1978)},  \cite{Igusa(1988)}, \cite{Waldhausen(1978)}, 
\cite{Waldhausen(1985)}.

\begin{conjecture}[\blue{The Farrell-Jones Conjecture for pseudo-isotopies and $A$-theory}]
\emphred{The Farrell-Jones Conjecture for pseudo-isotopies and $A$-theory}
respectively is the Isomorphism Conjecture for $H^?_*(-;\bfP)$ and
$H^?_*(-;\bfA)$ respectively for the family $\VCyc$.
\end{conjecture}

\begin{theorem}[\blue{Relating pseudo-isotopy and $K$-theory}]
\label{the:relating_pseudo-isotopy_K-theory}
The rational versions of the $K$-theoretic Farrell-Jones
Conjecture for coefficients in $\IZ$
and of the Farrell-Jones Conjecture  for Pseudoisotopies are equivalent.

In degree $n \le 1$ this is even true integrally.
\end{theorem}
\begin{proof}
See~\cite[1.6.7 on page~261]{Farrell-Jones(1993a)}.
\end{proof}

There are functors \red{$\bfTHH$} and \red{$\bfTC$} which assign to a ring
(or more generally to an $\bfS$-algebra) a spectrum describing its
\red{topological Hochschild homology} and its \red{topological cyclic homology}.
These functors play an important role in $K$-theoretic computations.
Composing it with the functor sending a groupoid to
a kind of group ring yields functors $\bfTHH_R$ and 
$\bfTC_R$ from $\Groupoids$ to $\Spectra$. Thus we obtain equivariant 
homology theories \red{$H^?_*(-;\bfTHH_R)$} and
\red{$H^?_*(-;\bfTC_R)$}.
They satisfy $H_n^G(G/H;\bfTHH_R) = \THH_n(RH)$ and  $H_n^G(G/H;\bfTC_R) = \TC_n(RH)$.

\begin{conjecture}[\blue{The Farrell-Jones Conjecture for topological 
Hochschild homology and cyclic homology}]
\emphred{The Farrell-Jones Conjecture for topological Hochschild
  homology and for topological cyclic
homology}
respectively is the Isomorphism Conjecture for $H^?_*(-;\bfTHH)$ and
$H^?_*(-;\bfTC)$ respectively for the family $\Cyc$ of cyclic subgroups.
\end{conjecture}

We can apply the functor topological $K$-theory also to
Banach algebras such that $l^1(G)$.
Let $\bfK^{\topo}_{l^1}$ be the functor from $\Groupoids$ to $\Spectra$
which assign to a groupoid the topological $K$-theory spectrum of
its $l^1$-algebra. We obtain an equivariant homology theory
\red{$H^?_*(-;\bfK^{\topo}_{l^1})$}.
It satisfies $H_n^G(G/H,\bfK^{\topo}_{l^1}) = K_n(l^1(H))$.

\begin{conjecture}[\blue{Bost Conjecture}]
\label{con:BostC}
The \emphred{Bost Conjecture} is the Isomorphism Conjecture for
$H^?_*(-;\bfK^{\topo}_{l^1})$ and the family $\Fin$.
\end{conjecture}

\begin{remark}[\blue{Relating the Baum-Connes Conjecture and the Bost Conjecture}]
The assembly map appearing in the Bost Conjecture~\ref{con:BostC}
$$H_n^G(\eub{G};\bfK^{\topo}_{l^1}) \to H_n^G(\pt;\bfK^{\topo}_{l^1}) = K_n(l^1(G))$$
composed with the  change of algebras homomorphism
$$K_n(l^1(G)) \to K_n(C^*_r(G))$$
is precisely the assembly map appearing in the 
Baum-Connes Conjecture~\ref{con:BCC}
$$H_n^G(\eub{G};\bfK^{\topo}) = H_n^G(\eub{G};\bfK^{\topo}_{l^1}) \to
H_n^G(\pt;\bfK^{\topo}) = K_n(C^*_r(G)).$$
\end{remark}

\begin{remark}[\blue{Relating the Farrell-Jones Conjecture for $L$-theory and the
Baum-Connes Conjecture}]
\label{rem:relating_FJC_andBC}
We discuss the relation between the Farrell-Jones Conjecture for $L$-theory
and the Baum-Connes Conjecture.
Mainly these come from the sequence of inclusions of rings
$$\IZ G \to \IR G \to C^*_r(G;\IR) \to C^*_r(G)$$
and the change of theories from algebraic to topological $K$-theory
and from algebraic $L$-theory to topological $K$-theory for $C^*$-algebras.
Namely, we obtain the following commutative diagram
$$
\begin{CD}
H_n^G(\EGF{G}{\Fin};\bfL^p_{\IZ}) [1/2] @>>> L^p_n(\IZ G)[1/2]
\\
@VV \cong V @VV \cong V
\\
H_n^G(\EGF{G}{\Fin};\bfL^p_{\IQ}) [1/2] @>>> L^p_n(\IQ G)[1/2]
\\
@VV \cong  V @VVV
\\
H_n^G(\EGF{G}{\Fin};\bfL^p_{\IR}) [1/2] @>>> L^p_n(\IR G)[1/2]
\\
@VV \cong  V @VVV
\\
H_n^G(\EGF{G}{\Fin};\bfL^p_{C_r^*(?;\IR )}) [1/2] @>>> L^p_n(C_r^*(G;\IR))[1/2]
\\
@VV \cong  V @VV \cong  V
\\
H_n^G(\EGF{G}{\Fin};\bfK^{\topo}_{\IR}) [1/2]
@>>> K_n (C_r^*(G;\IR))[1/2]
\\
@VVV @VVV
\\
H_n^G(\EGF{G}{\Fin};\bfK^{\topo}) [1/2]
@>>> K_n (C_r^*(G))[1/2]
\end{CD}
$$
The arrows marked with $\cong$ are known to be bijective
(see~\cite[page~376]{Ranicki(1981)}, 
\cite[Proposition~22.34 on page~252]{Ranicki(1992)},
\cite{Rosenberg(1995)}).
If $G$ satisfies the Farrell-Jones Conjecture~\ref{con:FJC_for_L} 
for $L$-theory for $R = \IZ$ and $R = \IR$ 
and the Baum-Connes Conjecture~\ref{con:BCC} for both the real and the complex case, 
then all horizontal arrows are bijective and hence all arrows
except the two lowest vertical ones are isomorphisms. 
The Baum-Connes Conjecture for the complex case does imply the
Baum-Connes Conjecture for the real case (see \green{Baum-Karoubi}~\cite{Baum-Karoubi(2004)}).
\end{remark}

\begin{theorem}[\blue{Rational computations of $K$-groups}, \green{L\"uck (2002)}]
\label{the:rational_computation_of_K_ast}
Let $G$ be a group. 
Let $T$ be the set of conjugacy
classes $(g)$ of elements $g \in G$ of finite order.

Then there is a commutative diagram
\\[3mm]
\xycomsquare{\bigoplus_{p+q=n} \bigoplus_{(g) \in T}
H_p(BC_G\langle g \rangle;\IC) \otimes_{\IZ}  K_q(\IC)}{}
{K_n(\IC G) \otimes_{\IZ} \IC}{}{}
{\bigoplus_{p+q= n} \bigoplus_{(g) \in T}
H_p(BC_G\langle g \rangle;\IC) \otimes_{\IZ}  K_q^{\topo}(\IC)}{}
{K_n^{\topo}(C_r^*(G)) \otimes_{\IZ} \IC}
\end{theorem}
\begin{proof}
See~\cite[Theorem~0.5]{Lueck(2002b)}.
\end{proof}

The horizontal arrows can be identified with the assembly maps
occurring in the Farrell-Jones Conjecture~\ref{con:FJC_for_K} and the Baum-Connes
Conjecture~\ref{con:BCC} by the equivariant Chern character.
In particular they are isomorphisms if these conjecture hold for $G$.

\begin{remark}[\blue{Splitting principle.}]
The calculation of the relevant $K$-and $L$-groups often
split into a \red{universal group homology part} which is independent of the theory,
and a second part
which essentially depends on the theory in question and the coefficients.
\end{remark}

\begin{remark}[\blue{Integral Computations}]
In contrast to general rational computations such as
the one appearing in Theorem~\ref{the:rational_computation_of_K_ast},
complete integral computations of $K_n(\IZ G)$, $L_n(\IZ G)$ or $K_n(C^*_r(G))$
seem to be possible only in special cases.
Here are some examples, where some of these groups are computed.
They are always based on the assumption
that the Farrell-Jones Conjecture for algebraic $K$ or $L$-theory 
or the Baum-Connes Conjecture is true 
what is in most cases known to be true.
\\[1mm]
\begin{tabular}{|p{70mm}|p{45mm}|}
\hline
Three-dimensional Heisenberg group and finite extensions
&
\green{L\"uck}~\cite{Lueck(2005heis)}
\\
\hline
$2$-dimensional crystallographic groups and more general
cocompact NEC-groups & 
\green{L\"uck-Stamm}~\cite{Lueck-Stamm(2000)},
\green{Pearson}~\cite{Pearson(1998)})
\\
\hline
Three-dimensional crystallographic groups
&
\green{Alves-Ontaneda}~\cite{Alves-Ontaneda(2006)}
\\
\hline
Fuchsian groups & 
\green{Berkhove-Juan-Pineda-Pearson}~\cite{Berkhove-Juan-Pineda-Pearson(2001)},
\green{Davis-L\"uck}~\cite{Davis-Lueck(2003)},
\green{L\"uck-Stamm}~\cite{Lueck-Stamm(2000)},
\\
\hline
Extensions $1 \to \IZ^n \to G \to F \to 1$
for finite $F$ and free conjugation action of 
$F$ in $\IZ^n$
&
\green{Davis-L\"uck}~\cite{Davis-Lueck(2003)},
\green{L\"uck-Stamm}~\cite{Lueck-Stamm(2000)}
\\
\hline
One relator groups &
\green{Davis-L\"uck}~\cite{Davis-Lueck(2003)}
\\
\hline
\\
$SL_3(\IZ)$ & 
\green{Sanchez-Garcia}~\cite{Sanchez-Garcia(2006SL)}, 
\green{Upadhyay}~\cite{Upadhyay(1996)}
\\
\hline
(Pure) braid groups
&
\green{Aravinda-Farrell-Roushon}~\cite{Aravinda-Farrell-Roushon(2000)},
\green{Farrell-Roushon}~\cite{Farrell-Roushon(2000)}
\\
\hline
fundamental groups of knot and link complements
&
\green{Aravinda-Farrell-Roushon}~\cite{Aravinda-Farrell-Roushon(1997)}
\\
\hline
Certain Coxeter groups 
&
\green{Lafont-Ortiz}~\cite{Lafont-Ortiz(2007)},  
\green{Sanchez-Garcia}~\cite{Sanchez-Garcia(2006Cox)}
\\
\hline
\end{tabular}
\end{remark}

Next we discuss the status of the various Conjectures such as the
Farrell-Jones Conjecture and the Baum-Connes Conjecture.

\begin{theorem}[\green{Bartels-L\"uck-Reich (2007)}]
\label{the:FJC_K_hyperbolic}
Let $R$ be a ring. Then every subgroup of a hyperbolic group
belongs to $\calfj_K(R)$.
\end{theorem}
\begin{proof}
See~\cite{Bartels-Lueck-Reich(2007hyper)}.
\end{proof}

\begin{theorem}[\green{Bartels-L\"uck}]
\label{the:FJC_L_hyperbolic}
Let $R$ be a ring with involution. Then every subgroup of a hyperbolic group
belongs to $\calfj_L(R)$.
\end{theorem}
\begin{proof}
The proof will appear in the paper~\cite{Bartels-Lueck(2008Borel)}
which is in preparation.
\end{proof}

\begin{theorem}[\green{Bartels-Echterhoff-Reich (2007)}]
\label{the:FJC_K_colim}
Let $R$ be a ring (with involution). Let $\{G_i \mid i\in I\}$ be a
directed system of groups (with not necessarily injective structure maps).
Let $G$ be a subgroup of the colimit $\colim_{i \in I} G_i$.

\begin{enumerate}

\item 
Suppose that for all $i \in I$ and every subgroup $H \subseteq G_i$ 
we have $H \in \calf_K(R)$. Then $G \in \calfj_K(R)$;

\item 
Suppose that for all $i \in I$ and every subgroup $H \subseteq G_i$ 
we have $H \in \calf_L(R)$. Then $G \in \calfj_L(R)$;

\item 
Suppose that for all $i \in I$ and every subgroup $H \subseteq G_i$ 
the Bost Conjecture holds for $H$. Then 
the Bost Conjecture holds for $G$.

\end{enumerate}
\end{theorem}
\begin{proof}
See~\cite[Theorem~0.8]{Bartels-Echterhoff-Lueck(2007colim)}.
\end{proof}

\begin{corollary}\label{cor:colimits_of_hyperbolic_groups}
Let $\{G_i \mid i\in I\}$ be a
directed system of hyperbolic groups (with not necessarily injective structure maps).
Let $G$ be the colimit $\colim_{i \in I} G_i$. Let $H \subseteq G$ be any subgroup of $G$.
Let $R$ be a ring (with involution). 

Then $H \in \calfj_K(R)$, $H \in \calfj_L(R)$ and $H$ satisfies the 
Bost Conjecture~\ref{con:BostC}.
\end{corollary}
\begin{proof}
In the case $\calfj_K(R)$ and $\calfj_L(R)$ one has just to combine
Theorem~\ref{the:FJC_K_hyperbolic}, Theorem~\ref{the:FJC_L_hyperbolic}, 
and Theorem~\ref{the:FJC_K_colim}.
In the case of the Bost Conjecture one has to use
Theorem~\ref{the:FJC_K_colim} and the results of
Lafforgue~\cite{Lafforgue(2002)}).
Details can be found in~\cite[Theorem~0.9]{Bartels-Echterhoff-Lueck(2007colim)}.
\end{proof}

\begin{example}[\blue{Groups satisfying the hypothesis of 
Corollary~\ref{cor:colimits_of_hyperbolic_groups}}]
\label{exa:exotic_groups}
The groups appearing in Theorem~\ref{the:FJC_K_hyperbolic}
are certainly wild in \green{Bridson's} \red{universe of groups}
(see~\cite {Bridson(2006ICM)}).
Many recent constructions of groups with exotic properties are given by
colimits of directed systems of hyperbolic groups. Examples are
\begin{itemize}
\item \red{groups with expanders} in the sense of \green{Gromov};

\item \red{Lacunary hyperbolic groups} in the sense of 
\green{Olshanskii-Osin-Sapir}~\cite{Olshanskii-Osin-sapir};

\item \red{Tarski monsters}, i.e., infinite groups 
whose proper subgroups are all finite cyclic of $p$-power order for a given prime $p$;
\end{itemize}

Notice that \green{Gromov's} \red{groups with expanders}
belong to $\calfj_K(R)$ for all $R$, whereas 
the Baum-Connes Conjecture with coefficients is not true for them by 
\green{Higson-Lafforgue-Skandalis}~\cite{Higson-Lafforgue-Skandalis(2002)}.
\end{example}

\begin{remark}[\blue{Twisted coefficients}]
The results above do extend to the more general case, where one allows
twisted group rings or more general \red{crossed product rings $R \ast G$}
in the setting of the Farrell-Jones Conjecture and 
\red{coefficients in a $G$-$C^*$-algebra}
in the setting of the Bost Conjecture. There are also so called 
\red{fibered versions} of the Farrell-Jones Conjecture and of 
an Isomorphism Conjecture in general
(see for instance~\cite[Definition~1.2]{Bartels-Lueck(2007ind)},
\cite[Definition~1.1]{Bartels-Lueck(2006)},
\cite[1.7]{Farrell-Jones(1993a)}.
In these more advanced settings with coefficients or the fibered setting
one has that the class of groups for which the conjectures with coefficients 
or the fibered version are true
is closed under taking finite direct products and taking subgroups.

Proofs of these claims can be found
in~\cite[Lemma~1.3]{Bartels-Lueck(2007ind)},
\cite[Lemma~1.2]{Bartels-Lueck(2006)},
\cite[Theorem~2.5 and Theorem~3.17]{Chabert-Echterhoff(2001b)},
\cite[Theorem~A.8 on page~289]{Farrell-Jones(1993a)},
\cite[5.5.4]{Lueck-Reich(2005)},
\cite[Corollary~7.12]{Oyono-Oyono(2001)}.
\end{remark}

\begin{example}[\blue{Torsionfree hyperbolic groups}]
If $G$ is a torsionfree hyperbolic group and $R$ any ring, then we get 
from Theorem~\ref{the:FJC_K_hyperbolic} as explained 
in~\cite[page~2]{Bartels-Lueck-Reich(2007hyper)}
an isomorphism
$$H_n(BG;\bfK(R)) \oplus \biggl(\bigoplus_{\substack{(C), C \subseteq G, C \not= 1\\
C \text{ maximal cyclic}}} \NK_n(R)\biggr) \; \xrightarrow{\cong} \;K_n(RG).$$
\end{example}

\begin{remark}[\blue{Program for CAT(0)-groups}]
\green{Bartels and L\"uck} have a program to prove $G \in \calfj_K(R)$ 
and $G \in \calfj_L(R)$ if $G$ acts
properly and cocompactly on a simply connected CAT(0)-space.
This would imply $G \in\calfj_K(R)$ and $G \in\calfj_L(R)$  for all \red{subgroups $G$
of cocompact lattices in almost connected Lie groups}  and for all
\red{limit groups} $G$.
\end{remark}

\begin{theorem}[\green{Mineyev-Yu (2002)}]
\label{the:BCC_hyperbolic}
Every subgroup of a hyperbolic group belongs to $\calbc$.
\end{theorem}
\begin{proof}
See~\cite{Mineyev-Yu(2002)}.
\end{proof}

\begin{definition}[\blue{a-T-menable group}]
  A group $G$ is \emphred{a-T-menable}, or, equivalently, has the \emphred{Haagerup
    property} if $G$ admits a metrically proper isometric action on some affine Hilbert
  space.
\end{definition}

The class of a-T-menable groups is closed under taking subgroups, under extensions with
finite quotients and under finite products. It is not closed under semi-direct products.
Examples of a-T-menable groups are:

\begin{itemize}
\item countable amenable groups;

\item countable free groups;

\item
discrete subgroups of $SO(n,1)$ and $SU(n,1)$;

\item
Coxeter groups;

\item
countable groups acting properly on trees, products of trees, or simply connected CAT(0)
cubical complexes.
\end{itemize}

A group $G$ has \emphred{Kazhdan's property (T)}
if, whenever it acts isometrically on some affine Hilbert
space, it has a fixed point. An infinite a-T-menable group does not have property (T).
Since $SL(n,\IZ)$ for $n \ge 3$ has property (T), it cannot be a-T-menable.

\begin{theorem}[\green{Higson-Kasparov(2001)}]
\label{the:Higson-Kasparov}
A group $G$ which is a-T-menable satisfies
the Baum Connes Conjecture (with coefficients).
\end{theorem}
\begin{proof}
See~\cite{Higson-Kasparov(2001)}.
\end{proof}

\begin{theorem}[\green{Farrell-Jones (1993)}]
\label{the:FJC_sub_cocom_disc_sub_almost_conn_Lie}
Let $G$ be a subgroup of a cocompact lattice in an almost connected Lie group.
Then the \red{Farrell-Jones Conjecture for pseudo-isotopy} is true for $G$.
\end{theorem}
\begin{proof}
See~\cite[Theorem~2.1 on page~263]{Farrell-Jones(1993a)}.
\end{proof}

\begin{theorem}[\green{L\"uck-Reich-Rognes-Varisco (2007)}]
The \red{Farrell-Jones Conjecture for topological Hochschild homology} is true for
all groups.
\end{theorem}
\begin{proof}
See~\cite{Lueck-Reich-Rognes-Varisco(2007)}.
\end{proof}

For more information about the theorems above and further results we refer
to the talks by \green{Bartels}, \green{Rosenthal} and \green{Varisco}.

\begin{remark}[\blue{Borel-Conjecture}]
Recall that the Borel Conjecture~\ref{con:borel} is true for a closed
$n$-dimensional manifold $M$ with fundamental group $G$
if $G$ belongs to both $\calfj_K(\IZ)$ and $\calfj_K(\IZ)$ and
$n \ge 5$ (see Theorem~\ref{the:The_Farrell-Jones_Conjecture_and_the_Borel_Conjecture}).
Recall from Corollary~\ref{cor:colimits_of_hyperbolic_groups} that
any subgroup of a colimit over a directed system of hyperbolic groups
(with not necessarily injective structure maps) satisfy this assumption
and that very exotic groups occur in this way (see 
Example~\ref{exa:exotic_groups}).
\end{remark}

Here are other groups for which the Borel Conjecture has been proved.

\begin{theorem}[\green{Farrell-Jones}]
The \red{Borel Conjecture} and the \red{$L$-theoretic Farrell-Jones Conjecture
with coefficients in $\IZ$} are true for a group $G$ if
one of the following conditions are satisfied:

\begin{itemize}

\item$G$ is the fundamental group of a closed
Riemannian manifold with non-positive curvature;

\item $G$ is the fundamental group of a complete
Riemannian manifold with pinched negative curvature;

\item $G$ is a torsionfree subgroup of $GL(n,\IR)$.

\end{itemize}
\end{theorem}
\begin{proof}
See~\cite{Farrell-Jones(1993c)}, \cite{Farrell-Jones(1998)}.
\end{proof}

For more information we refer  to~\cite[Section~5]{Lueck-Reich(2005)},

In the following table we list prominent classes of groups and state
whether they are known to satisfy the Farrell-Jones
Conjectures~\ref{con:FJC_for_K} and~\ref{con:FJC_for_L} and the
Baum-Connes Conjecture~\ref{con:BCC} or versions of them. Some of the
classes are redundant. A question mark means that the author does not
know about a corresponding result.  A phrase like injectivity or after
inverting $2$ is true means that the corresponding assembly map is
injective or is bijective after inverting $2$. 
The reader should keep in mind that there may exist 
results of which the authors are not aware. The following 
table is an updated version
of the one appearing in~\cite[5.3]{Lueck-Reich(2005)}.
\\[3mm]
\begin{center}
\begin{tabular}{|| p{31mm}| p{27mm}| p{27mm}| p{27mm}||}
\hline \hline
type of group & 
Baum-Connes Conjecture~\ref{con:BCC}
&
Farrell-Jones Conjecture~\ref{con:FJC_for_K}
for $K$-theory
&
Farrell-Jones Conjecture~\ref{con:FJC_for_L}
for $L$-theory
\\
\hline\hline
a-T-menable groups  
& true with coefficients (see Theorem~\ref{the:Higson-Kasparov})
& {?} 
& injectivity is true after inverting $2$ for $R = \IZ$
(see Remark~\ref{rem:relating_FJC_andBC})
\\
\hline
amenable groups 
& true with coefficients (see Theorem~\ref{the:Higson-Kasparov})
& {?} 
& injectivity is true after inverting $2$ for $R = \IZ$
(see Remark~\ref{rem:relating_FJC_andBC})
\\
\hline
elementary amenable groups 
& true with coefficients (see Theorem~\ref{the:Higson-Kasparov})
& true fibered for a ring with finite characteristic $N$ after inverting $N$
(see~\cite[Theorem~0.3]{Bartels-Lueck-Reich(2007appl)})
& true fibered after inverting $2$ for $R = \IZ$
(see~\cite[Lemma~1.12  and Lemma~7.1]{Bartels-Lueck-Reich(2007appl)}
or see~\cite[Theorem~5.2]{Farrell-Linnell(2003a)})
\\
\hline
virtually poly-cyclic
& true with coefficients (see Theorem~\ref{the:Higson-Kasparov})
&  true rationally for $R = \IZ$, true fibered for $R = \IZ$ in the range $n \le 1$
(see~\cite[Remark~5.3]{Lueck-Reich(2005)}.
& true fibered after inverting $2$ for $R = \IZ$
(see~\cite[Lemma~1.12 and Lemma~7.1]{Bartels-Lueck-Reich(2007appl)}
or see~\cite[Theorem~5.2]{Farrell-Linnell(2003a)})
\\
\hline
torsion free virtually solvable subgroups of $GL(n,\IC)$ 
& true with coefficients (see Theorem~\ref{the:Higson-Kasparov})
& true in the range $\le 1$ \cite[Theorem~1.1]{Farrell-Linnell(2003a)}
& true fibered after inverting $2$ \cite[Theorem~5.2]{Farrell-Linnell(2003a)}
\\
\hline
discrete subgroups of Lie groups with finitely many path components 
& injectivity true 
(see~\cite[Theorem~5.9 and Remark~5.11 on page~718]{Lueck-Reich(2005)})
& injectivity is true for the family $\Fin$ and all rings $R$ 
(see~\cite{Bartels-Rosenthal(2006)})
& injectivity is true for the family $\Fin$ and  rings $R$ with vanishing
$K_n(RH)$ for $n \le -2$ and $H\subseteq G$ finite
(see~\cite{Bartels-Rosenthal(2006)})
\\
\hline
subgroups of groups which are discrete cocompact subgroups of 
Lie groups with finitely many path components 
& injectivity is true 
(see~\cite[Theorem~5.9 and Remark~5.11 on page~718]{Lueck-Reich(2005)})
& true rationally,
  true fibered in the range $n \le 1$ 
(see~\cite[1.6.7 on page~261 and Theorem~2.1 on page~263]{Farrell-Jones(1993a)}.)
& injectivity is true for the family $\Fin$  and rings $R$ with vanishing
$K_n(RH)$ for $n \le -2$ and $H\subseteq G$ finite
(see~\cite{Bartels-Rosenthal(2006)})
\\
\hline
linear groups 
& injectivity is true (see~\cite{Guentner-Higson-Weinberger(2005)}) 
& {?} 
& 
injectivity is true after inverting $2$ for $R = \IZ$
(see Remark~\ref{rem:relating_FJC_andBC})
\\
\hline
finitely generated  subgroup of $GL_n(k)$
for a global field $k$ 
& injectivity is true (see~\cite{Guentner-Higson-Weinberger(2005)}) 
& injectivity is true for $R = \IZ$ (see~\cite{Ji(2006torsion)})
& injectivity is true for $R = \IZ$ (see~\cite{Ji(2006torsion)})
\\
\hline
torsion free discrete subgroups of $GL(n,\IR)$ 
& injectivity is true (see~\cite{Guentner-Higson-Weinberger(2005)}) 
& true in the range $n \le 1$ (see~\cite{Farrell-Jones(1998)} and 
also~\cite[Theorem~5.5 on page~722]{Lueck-Reich(2005)})
& true for $R = \IZ$ (see~\cite{Farrell-Jones(1998)} and 
also~\cite[Theorem~5.5 on page~722]{Lueck-Reich(2005)})
\\
\hline\hline
\end{tabular}

\begin{tabular}{|| p{31mm}| p{27mm}| p{27mm}| p{27mm}||}
\hline \hline
type of group 
& 
Baum-Connes Conjecture~\ref{con:BCC}
&
Farrell-Jones Conjecture~\ref{con:FJC_for_K}
for $K$-theory
&
Farrell-Jones Conjecture~\ref{con:FJC_for_L}
for $L$-theory  
\\
\hline\hline
Groups with finite $\eub{G}$ and
finite asymptotic dimension
& injectivity is true with coefficients
(see~\cite[Theorem~1.1]{Higson(2000)},
\cite[Theorem~1,1 and Lemma~4.3]{Higson-Roe(2000a)},
& injectivity is true for the family $\Fin$ and all rings $R$  (see~\cite{Bartels-Rosenthal(2006)})
& injectivity is true for the family $\Fin$ and  rings $R$ with vanishing
$K_n(RH)$ for $n \le -2$ and $H\subseteq G$ finite
(see~\cite{Bartels-Rosenthal(2006)})
\\
\hline\hline
$G$ acts properly and isometrically on a complete Riemannian manifold $M$ with 
non-positive sectional curvature
& rational injectivity is true (see~\cite{Kasparov(1988)})
& {?}
& injectivity is true after inverting $2$ for $R = \IZ$
(see Remark~\ref{rem:relating_FJC_andBC})
\\
\hline
$\pi_1(M)$ for a complete Riemannian manifold $M$ with 
non-positive sectional curvature
& rational injectivity is true (see~\cite{Kasparov(1988)})
& {?}
& injectivity is true for $R = \IZ$ (see~\cite[Corollary~2.3]{Ferry-Weinberger(1991)}
\\
\hline
$\pi_1(M)$ for a complete Riemannian manifold $M$ with 
non-positive sectional curvature which is A-regular
& rational injectivity is true (see~\cite{Kasparov(1988)})
& true in the range $n \le 1$ for $R = \IZ$
(see~\cite[Proposition~0.10 and Lemma~0.12]{Farrell-Jones(1998)})
& true for $R = \IZ$ (see~\cite{Farrell-Jones(1998)})
\\
\hline
$\pi_1(M)$ for a complete Riemannian manifold $M$ with 
pinched negative sectional curvature
& rational injectivity is true (see~\cite{Kasparov(1988)})
& true in the range $n \le 1$ and true rationally for $= \IZ$
(see~\cite[Proposition~0.10 and Lemma~0.12 and page 216]{Farrell-Jones(1998)})
& true for $R = \IZ$ (see~ and 
also~\cite[Theorem~5.5 on page~722]{Lueck-Reich(2005)})
\\
\hline
$\pi_1(M)$ for a closed Riemannian manifold $M$ with 
non-positive sectional curvature
& rational injectivity is true (see~\cite{Kasparov(1988)})
& true fibered in the range $n \le 1$, true rationally for $R = \IZ$
(see~\cite{Farrell-Jones(1991)}).
& true for $R = \IZ$ (see~\cite{Farrell-Jones(1998)} and 
also~\cite[Theorem~5.5 on page~722]{Lueck-Reich(2005)})
\\
\hline
$\pi_1(M)$ for a closed Riemannian manifold $M$ with 
 negative sectional curvature
& true for all subgroups (see~\cite{Mineyev-Yu(2002)})
& true for all coefficients $R$ (see~\cite{Bartels-Reich(2005JAMS)})
& true for $R = \IZ$ (see~\cite{Farrell-Jones(1998)} and 
also~\cite[Theorem~5.5 on page~722]{Lueck-Reich(2005)})
\\
\hline\hline
\end{tabular}

%%%%%%%%%%%%%%%%%%%%%%%%%%%%%%%%%%%%%%%%%%%%%%%%%%%%%%%%%%%%%%%%%%%%%%%%%%%%%%%%%%%%

\begin{tabular}{|| p{31mm}| p{27mm}| p{27mm}| p{27mm}||}
\hline \hline
type of group 
& 
Baum-Connes Conjecture~\ref{con:BCC}
&
Farrell-Jones Conjecture~\ref{con:FJC_for_K}
for $K$-theory
&
Farrell-Jones Conjecture~\ref{con:FJC_for_L}
for $L$-theory  
\\
\hline\hline
subgroups of directed colimits of
word hyperbolic groups 
& {?}
& true for all $R$ (see~\cite{Bartels-Lueck-Reich(2007hyper)} 
and~\cite[Theorem~0.9]{Bartels-Echterhoff-Lueck(2007colim)})
& true for all $R$ (see~\cite{Bartels-Lueck(2008Borel)} 
and~\cite[Theorem~0.9]{Bartels-Echterhoff-Lueck(2007colim)})
\\
\hline
subgroups of word hyperbolic groups 
& true (see~\cite{Mineyev-Yu(2002)})
& true for all $R$ (see~\cite{Bartels-Lueck-Reich(2007hyper)})
& true for all $R$ (see~\cite{Bartels-Lueck(2008Borel)})
\\
\hline
one-relator groups 
& true with coefficients (see~\cite{Oyono-Oyono(2001b)})
& rational injectivity is true for $R = \IZ$ or for regular
$R$ with $\IQ \subseteq R$
(see~\cite{Bartels-Lueck(2006)})
& true after inverting $2$ for all $R$
(see~\cite[Proposition~0.9 and Theorem~0.13]{Bartels-Lueck(2006)}),
true after inverting $2$ for $R = \IZ$ fibered
(see~\cite{Roushon(2007)})
\\
\hline
torsion free one-relator groups 
& true with coefficients (see~\cite{Oyono-Oyono(2001b)})
& true for $R$ regular 
\cite[Theorem~19.4 on page~249 and Theorem~19.5  on page~250]{Waldhausen(1978a)}
&  true after inverting $2$ for all $R$
(see~\cite[Corollary~8]{Cappell(1974c)}, \cite[Proposition~0.9 and Theorem~0.13]{Bartels-Lueck(2006)}),
true after inverting $2$ for $R = \IZ$ fibered
(see~\cite{Roushon(2007)})
\\
\hline
$3$-manifold groups 
& {?}
& true fibered for $R= \IZ$ in the range $n \le 1$
(see~\cite[Corollary~4.2]{Roushon(2006)} and~\cite[Corollary~1.1.5]{Roushon(2006polysurf)})
& {?}
\\
\hline
Haken $3$-manifold groups (in particular knot groups) 
& true with coefficients (see~\cite[Theorem~5.23]{Mislin-Valette(2003)})

& true for $R$ regular 
(see~\cite[Theorem~19.4 on page~249 and Theorem~19.5  on page~250]{Waldhausen(1978a)})
& true after inverting $2$ for all $R$ (see~\cite[Corollary~8]{Cappell(1974c)})
\\
\hline
$SL(n,\IZ), n \geq 3$ 
& injectivity is true (see~\cite{Guentner-Higson-Weinberger(2005)}) 
& injecivity is true for the family $\Fin$  and $R = \IZ$ (see~\cite{Ji(2006torsion)})
& injecivity is true for the family $\Fin$ and $R = \IZ$ (see~\cite{Ji(2006torsion)})
\\
\hline 
Artin's braid group $B_n$
& true with coefficients (see~\cite[Theorem~5.25]{Mislin-Valette(2003)}, \cite{Schick(2007fingrext)})
& true for $R = \IZ$ fibered in the range $n \leq 1$, true for $R = \IZ$
rationally (see~\cite{Farrell-Roushon(2000)})
& injectivity is true after inverting $2$ for $R = \IZ$
(see Remark~\ref{rem:relating_FJC_andBC})
\\
\hline 
pure braid group $C_n$
& true with coefficients  (see~\cite[Theorem~5.25]{Mislin-Valette(2003)}, \cite{Schick(2007fingrext)})
& true for $R = \IZ$ in the range $n \le 1$ (see~\cite{Aravinda-Farrell-Roushon(2000)}) 
& true after inverting $2$ for all $R$
(see~\cite[Proposition~0.9 and Theorem~0.13]{Bartels-Lueck(2006)})
\\
\hline
Thompson's group $F$ 
& true with coefficients \cite{Farley(2003)} 
& {?} 
& injectivity is true after inverting $2$ for $R = \IZ$
(see Remark~\ref{rem:relating_FJC_andBC})
\\
\hline\hline
\end{tabular}
\end{center}

\begin{remark}[\blue{Open cases}]
  \label{rem:open_cases}
  We mention some interesting groups or classes of groups for which
  the Conjectures are  still open.

  \begin{itemize}

  \item The Farrell-Jones Conjecture for $K$-theory~\ref{con:FJC_for_K}
    and for $L$-theory~\ref{con:FJC_for_L} and the Baum-Connes
    Conjecture~\ref{con:BCC} are to the authors's knowledge
    open for \red{$SL_n(\IZ)$ for $n \ge 3$},
    \red{mapping class groups} and \red{$\Out(F_n)$};

  \item The Farrell-Jones Conjecture for $K$-theory~\ref{con:FJC_for_K}
    and for $L$-theory~\ref{con:FJC_for_L} are to the author's
    knowledge open for \red{solvable groups} and
    \red{one-relator groups}, whereas the Baum-Connes
    Conjecture~\ref{con:BCC} is known for these groups.

  \item There are certain \red{groups with expanders}
    for which the Baum-Connes Conjecture~\ref{con:BCC} is 
    to the author's knowledge open 
    and the version with coefficients is
    actually false  (see
    \green{Higson-Lafforgue-Skandalis}~\cite{Higson-Lafforgue-Skandalis(2002)}).  
    The Farrell-Jones Conjecture for $K$-theory~\ref{con:FJC_for_K}
    and for $L$-theory~\ref{con:FJC_for_L} are known for these
    groups since they are examples of directed colimits of hyperbolic
    groups.

  \end{itemize}
\end{remark}

\begin{remark}[\blue{Possible candidates for counterexamples}]
  It is not known whether there are counterexamples to the
  Farrell-Jones Conjecture or the Baum-Connes Conjecture.  There seems
  to be no promising candidate of a group $G$ which is a potential
  counterexample to the $K$- or $L$-theoretic Farrell-Jones Conjecture
  or the Bost Conjecture. We cannot name a property or a lack of a certain property
  of a group which may be a reason for this group to be
  counterexample. There are many groups with rather exotic properties
  for which these Conjectures are known to be true.
\end{remark}

\begin{remark}[\blue{The suspicious Baum-Connes Conjecture}]
  The Baum-Connes Conjecture is the one for which it is most likely
  that there may exist a counterexample.  One reason is the \red{existence
  of counterexamples to the version with coefficients} (see
  \green{Higson-Lafforgue-Skandalis}~\cite{Higson-Lafforgue-Skandalis(2002)}).
  Another reason is that $K_n(C^*_r(G))$ has certain \red{failures
    concerning functoriality} which do not occur for
  $K_n^G(\eub{G})$.  For instance $K_n(C^*_r(G))$ is not known to be
  functorial for arbitrary group homomorphisms since the reduced group
  $C^*$-algebra is not functorial for arbitrary group homomorphisms.
  These failures are not present for the Farrell-Jones and the
  Bost Conjecture, i.e.,  for $K_n(RG)$, $L^{\langle - \infty
    \rangle}(RG)$ and $K_n(l^1(G))$.
\end{remark}

\begin{remark}[\blue{Methods of proof}]
  Most of the proofs of the Farrell-Jones Conjecture use methods from
  \red{controlled topology}. Roughly speaking, controlled topology
  means that one considers free modules with a basis and thinks of
  these basis elements as sitting in a metric space. Then a map
  between such modules can be visualized by arrows between these basis
  elements. Control means that these arrows are small.  Our
  homological approach to the assembly map is good for \red{structural
    investigations} but not for proofs. For proofs of the
  Farrell-Jones Conjecture or the Baum-Connes Conjecture it is often
  helpful to get some \red{geometric input}.  In the Farrell-Jones
  setting the door to geometry is opened by interpreting the assembly
  map as a \red{forget control map}.  The task to show for instance
  surjectivity is to manipulate a representative of the $K$-or
  $L$-theory class such that its class is unchanged but one has
  \red{gained control}. This is done by geometric constructions which
  yield \red{contracting maps}.  These constructions are possible if
  some geometry connected to the group is around, such as negative
  curvature.  We refer to the lectures of \green{Bartels} for such
  controlled methods.

  The approach using \red{topological cyclic homology} goes back to
  \green{B\"ockstedt-Hsiang-Madsen}. It is of \red{homotopy theoretic
    nature}.  We refer to the lecture of \green{Varisco} for more
  information about that approach.
  
  The methods of proof for the Baum-Connes Conjecture are of
  \red{analytic nature}.  The most prominent one is the
  \red{Dirac-Dual-Dirac method} based on $KK$-theory due to
  \green{Kasparov}. \red{$KK$-theory} is a bivariant theory together
  with a product.  The assembly map is given by multiplying with a
  certain element in a certain $KK$-group.  The essential idea is to
  construct another element in a dual $KK$-group which implements the
  inverse of the assembly map.
  
  The analytic methods for the proof of the Baum-Connes Conjecture do
  not seem to be applicable to the Farrell-Jones setting.  One would
  hope for a transfer of methods from the Farrell-Jones setting to the
  Baum-Connes Conjecture.  So far not much has happened in this
  direction.
\end{remark}

%\bibliographystyle{abbrv}
%\bibliography{dbdef,dbpub,dbpre,dbextra_hangzhou}

\def\cprime{$'$} \def\polhk#1{\setbox0=\hbox{#1}{\ooalign{\hidewidth
  \lower1.5ex\hbox{`}\hidewidth\crcr\unhbox0}}}

\end{document}